\documentclass
{amsbook}
\usepackage{sty}
\usepackage[section]{placeins}
\usepackage{MnSymbol}

\usepackage[pdftex]{hyperref}
\hypersetup{colorlinks,citecolor=blue,linktocpage,hyperindex=true,backref=true}

\DeclareMathAlphabet      {\mathbf}{OT1}{cmr}{bx}{it}
\DeclareMathAlphabet      {\mathbfx}{OT1}{cmr}{bx}{n}

\usepackage[all]{xy}
\usepackage{graphicx}

\usepackage{stmaryrd, tikz}

\numberwithin{section}{chapter}

\begin{document}

\hypersetup{pageanchor=false}

\author{\Large\textbf{Aidan Sims}\\\today}
\address{School of Mathematics and Applied Statistics\\University of Wollongong\\Wollongong NSW 2522\\Australia}
\email{asims@uow.edu.au}

\title{\huge\textbf{Hausdorff \'etale groupoids and their $C^*$-algebras}}
\maketitle

\thispagestyle{empty} \cleardoublepage

\hypersetup{pageanchor=true}

\pagenumbering{roman}

\tableofcontents

\section*{Foreword}
These notes were written as supplementary material for a five-hour lecture series
presented at the Centre de Recerca Mathem\`atica at the Universitat Aut\`onoma de
Barcelona from the 13th to the 17th of March 2017. This was part of the Intensive
Research Program \emph{Operator algebras: dynamics and interactions} which ran from March
to July 2017. I thank the CRM for the excellent support, facility, and atmosphere it
provided. The whole IRP provided an exceptional environment for research and research
interactions, and it was a real pleasure to be a part of the enthusiastic and productive
research activity that was going on. The place was really buzzing.

The intention of these notes is to give a brief overview of some key topics in the area
of $C^*$-algebras associated to \'etale groupoids. The scope has been deliberately
contained to the case of \'etale groupoids with the intention that much of the
representation-theoretic technology and measure-theoretic analysis required to handle
general groupoids can be suppressed in this simpler setting. Because these notes are
based on a short course, they feature only a small selection of topics, chosen for their
relevance and interest to participants in the masterclass. My choice to include or omit
any particular topic is not a comment on the interest of that topic in general---it was a
question of developing something consistent and coherent that could sensibly be presented
in a week's worth of lectures, and would hopefully seem sensible in the context of the
masterclass and the other two lecture series being presented. So, for example, I have not
included any discussion here of inverse semigroups and their connections with \'etale
groupoids even though, arguably, inverse semigroups and \'etale groupoids are more or
less inseparable. I apologise wholeheartedly to all those people whose very nice work on
groupoids and groupoid $C^*$-algebras has every right to be discussed in a set of notes
like this but has not been mentioned.

My thanks to those participants in the masterclass who pointed out errors and possible
improvements on the draft version of the notes that was circulated at the time of the
lecture series. Special thanks to Kevin Brix from the University of Copenhagen for a
number of helpful corrections and suggestions, and to Valentin Deaconu for spotting a
number of typos. Also my thanks to CRM and to Birkh\"auser for agreeing to have a copy
posted on the arXiv.

\subsection*{Errata}
There are doubtless errors in these notes, despite my best efforts to weed them out and
the generous help I've had from others. I will maintain an up-to-date list of errata for
these notes as part of a common errata file for these notes and the other two sets of
notes in this volume at \url{http://www.uow.edu.au/~asims/2017crm/errata.pdf}. Please let
me know by email to \href{mailto:asims@uow.edu.au}{asims@uow.edu.au} if you find any
typos or other errors---I'll be very grateful.

\clearpage

\pagenumbering{arabic}\setcounter{page}{1}

\chapter{Introduction}

Groupoids are algebraic objects that behave like a group except that the multiplication
operation is only partially defined. Topological groupoids provide a useful unifying
model for groups and group actions, and equivalence relations induced by continuous maps
between topological spaces. They also provide a good algebraic model for the quotient of
a topological space by a group or semigroup action in instances where the quotient space
itself is, topologically, poorly behaved---for example, the quotient of a shift-space
determined by the shift map, or the quotient of the circle by an irrational rotation.

The collection $\mathcal{G}^{(0)}$ of idempotent elements in a groupoid $\mathcal{G}$ is
called its unit space, since these are precisely the elements $x$ that satisfy $x \gamma
= \gamma$ and $\eta x = \eta$ whenever these products are defined. This leads to a
natural fibred structure of $\mathcal{G}$ over $\mathcal{G}^{(0)}$: the fibre over a unit
$x$ is the collection $\mathcal{G}_x$ of elements $\gamma$ for which the product $\gamma
x$ is defined. If $\mathcal{G}$ is a topological groupoid, then $\mathcal{G}^{(0)}$, as
well as each $\mathcal{G}_x$, is a topological space in the subspace topology, and it is
often helpful to think of the subspaces $\mathcal{G}_x$ as transverse to
$\mathcal{G}^{(0)}$. In the special case that $\mathcal{G}$ is a group, its unit space
has just one element $e$, and then $\mathcal{G} = \mathcal{G}_e$. So the analogues, in
the setting of groupoids, of topological properties of groups, typically involve
corresponding topological conditions on the sets $\mathcal{G}_x$. In particular, the
analogue of a discrete group is a groupoid in which the sets $\mathcal{G}_x$ are all
discrete in a coherent way. More specifically, we ask that the map that sends each
$\gamma$ to the unique element $s(\gamma)$ for which $\gamma s(\gamma)$ makes sense
should be a local homeomorphism. Renault \cite{Ren80} called such groupoids ``\'etale"
(which means something like ``loose" or ``spread out"), and the terminology has stuck.

The study of $C^*$-algebras associated to groupoids was initiated by Renault in his PhD
thesis \cite{Ren80}, and was motivated by earlier work, particularly that of Feldman and
Moore \cite{FM1, FM2, FM3} for von Neumann algebras. Groupoid $C^*$-algebras have been
studied intensively ever since, and provide useful and concrete models of many classes of
$C^*$-algebras. The construction and study of groupoid $C^*$-algebras in general is
fairly involved and requires a significant amount of representation-theoretic background.
But just as the study of $C^*$-algebras of discrete groups and their crossed products
requires less background than the study of group $C^*$-algebras in general, so the
restriction to the situation of \'etale groupoids significantly reduces the overheads
involved in studying groupoid $C^*$-algebras. And, as Renault realised from the outset,
\'etale-groupoid $C^*$-algebras are sufficient to capture a lot of examples. For example,
all crossed products of commutative $C^*$-algebras by discrete-group actions; all AF
algebras; all Cuntz--Krieger algebras and graph algebras; all Kirchberg algebras in the
UCT class; and many others.

My intention in preparing these notes was to develop a concise account of the elementary
theory of \'etale groupoids and their $C^*$-algebras that minimises the
represent\-ation-theoretic and analytic background needed. For this reason, I have
chosen, for example, not to include any discussion of non-Hausdorff groupoids, even
though there are many good reasons for studying these (for example, Exel and Pardo's
$C^*$-algebras associated to  self-similar actions of groups on graphs admit
\'etale-groupoid models, but these groupoids are frequently non-Hausdorff). I have also
skimmed over some of the more technical aspects of groupoid theory (for example the
question of amenability). I have tried to include enough examples along the way, together
with two sections that outline a couple of important applications of the theory, to
illustrate what is going on with key concepts.

\smallskip

We start in Chapter~\ref{ch:groupoids} with a discussion of groupoids themselves: the
axioms and set up, a number of illustrative examples, the definition of a topological
groupoid, and a discussion of the \'etale condition.

In Chapter~\ref{ch:C*-algs}, we describe the construction of the convolution algebra of
an \'etale groupoid, then of its two $C^*$-algebras---the full $C^*$-algebra and the
reduced $C^*$-algebra---and finally of the notion of equivalence of \'etale groupoids and
Renault's equivalence theorem. We already see advantages to sticking to \'etale groupoids
here, since we are able to get through all of this material using fairly elementary
techniques, and in particular without having to include a treatment of Renault's
Disintegration Theorem, which is one of the mainstays of groupoid $C^*$-algebra theory in
general, but quite a complicated piece of work.

In Chapter~\ref{ch:structure}, we discuss some of the elementary structure theory of
groupoid $C^*$-algebras. As mentioned earlier, I have chosen to skim over the details of
amenability for groupoids, though I have tried to give enough references to help the
interested reader find out more. There is a whole book on the subject of amenability for
groupoids \cite{AR}, and unlike the situation for groups, it's far from a done deal. My
focus in discussing amenability has been to describe its $C^*$-algebraic consequences,
and some standard techniques for showing that a given groupoid is amenable. We then go on
to discuss effective groupoids (these are like topologically free group actions). Again,
things simplify significantly in the \'etale setting, and we are able to present a short
and self-contained proof that every nonzero ideal of the reduced $C^*$-algebra of an
effective \'etale (Hausdorff) groupoid must have nonzero intersection with the abelian
subalgebra of $C_0$-functions on the unit space. We follow this with a discussion of
invariant sets of units and the ideal structure of the $C^*$-algebra of an amenable
\'etale groupoid. We use the results we have put together on ideal structure to
characterise the amenable \'etale (Hausdorff) groupoids whose $C^*$-algebras are simple.
To finish Chapter~\ref{ch:structure}, we describe Anantharaman-Delaroche's notion of
locally contracting groupoids, and prove that the reduced $C^*$-algebra of a locally
contracting groupoid is purely infinite.

Interestingly, modulo treating amenability as a black box, all of the structure theory
developed in these first few chapters is done without recourse to any heavy machinery
like Renault's Disintegration theorem. Though probably known to, or at least expected by,
experts, I am not aware of such an approach having appeared in print previously.

In Chapter~\ref{ch:Weyl gpd} we discuss beautiful results of Renault, extending earlier
work of Kumjian, that provide a $C^*$-algebraic version of Feldman--Moore theory, and
then go on to discuss an application of this machinery to the classification of Fell
algebras.

In Section~\ref{sec:KR-theory}, we discuss the notion of a twist $\mathcal{E}$ over an
\'etale groupoid $\mathcal{G}$, and of the associated full and reduced $C^*$-algebras.
These can be thought of as the analogue, for groupoids, of the reduced and full twisted
$C^*$-algebras of a discrete group with respect to a $\mathbb{T}$-valued 2-cocycle. The
details here begin to get significantly more complicated than in the previous three
sections, and so I have given an overview with sketches of proofs rather than a detailed
treatment of all the results. We discuss Renault's definition of a Cartan pair of
$C^*$-algebras, indicate how a twist over an effective \'etale groupoid gives rise to a
Cartan pair of $C^*$-algebras, and outline Renault's proof that the twist can be
recovered from its Cartan pair, so that Cartan pairs are in bijection with twists over
effective \'etale groupoids. This implies, for example, due to work of Barlak and Li,
that any separable nuclear $C^*$-algebra that admits a Cartan subalgebra belongs to the
UCT class.

We then wrap up in Section~\ref{sec:DD} by outlining an application of the groupoid
technology we have developed to the classification of Fell algebras. Fell algebras are
Type~I $C^*$-algebras that generalise the continuous trace $C^*$-algebras, which in turn
are the subject of the famous Dixmier--Douady classification theorem. The classical
approach to the Dixmier--Douady theorem does not work well for Fell algebras, but another
approach is available: we sketch how to construct, from each Fell algebra $A$, a Cartan
pair $(C,D)$ in which $C$ is Morita equivalent to $A$. We then outline how to make the
collection of isomorphism classes of twists over a given groupoid $\mathcal{G}$ into a
group $\operatorname{Tw}_\mathcal{G}$. If $\mathcal{G}$ is the equivalence relation
determined by a local homeomorphism of a locally compact Hausdorff space $Y$ onto a
locally locally compact, locally Hausdorff space $X$, then its twist group is isomorphic
to a second sheaf-cohomology group of $X$. In particular, the pair $(C, D)$ discussed
above determines an element of $H^2(\widehat{A}, \mathcal{S})$; moreover this class is
independent of any of the choices involved in our constructions, so we can regard it as
an invariant $\delta(A)$ of $A$. The main result discussed in the section says that
$\delta(A)$ is a complete invariant of $A$, and also that the range is exhausted in the
sense that every element of $H^2(X, \mathcal{S})$ can be obtained as $\delta(A)$ for some
Fell algebra $A$ with spectrum $X$. In this section, I have been very brief. I provide no
detailed proofs, and very few proof sketches, and instead try to present the big picture.
The details can be found in \cite{aHKS}.

\smallskip

These notes are just a brief introduction to a small part of the theory of groupoid
$C^*$-algebras. There are many useful references for the more general theory. The theory
began with Renault \cite{Ren80}, and this remains the definitive text. Exel \cite{Exel}
and Paterson \cite{Paterson} have both given excellent discussions of \'etale
groupoids---particularly as they relate to inverse semigroups---but in the non-Hausdorff
setting, where the details are a little trickier. Renault's equivalence theorem for
(full) groupoid $C^*$-algebras first appeared in print in the work of
Muhly--Renault--Williams \cite{MRW}, but the approach used here, via linking groupoids,
is based on \cite{SW} and also owes a lot to many valuable conversations I have had with
Alex Kumjian. It is also closely related to \cite{Paravicini, StadlerOuchie, Tu}. A trove
of information amount about amenability of groupoids is contained in
Anantharaman-Delaroche and Renault's book \cite{AR} on the topic. There are countless
other very useful references that I have forgotten, or that are hard to come by (for
example Paul Muhly's excellent but lamentably unfinished book on the subject). I
apologise to the surely long list of people whose work I have overlooked (despite its
being eminently worthy of mention and attention) in this brief and far-from-comprehensive
discussion.

What I believe \emph{is} missing from the literature is an elementary and self-contained
introduction to the $C^*$-algebras of \'etale Hausdorff groupoids (though  Exel's paper
\cite{Exel} does contain an excellent treatment of the construction of the groupoid
$C^*$-algebra for non-\'etale Hausdorff groupoids); these notes should go a little way to
filling this gap. I hope that they give a flavour for the subject and a useful reference
for those who find themselves in the enviable position of having all their groupoids turn
out to be Hausdorff and \'etale. I think that most of the arguments in the first three
sections here (with the exceptions of anything about amenability, and of
Anantharaman--Delaroche's pure-infiniteness result) were developed from scratch; but of
course the results are not new, and the treatment reflects the many ideas and techniques
that I have accumulated both from the literature, and from discussion and collaboration
with many people including Jon Brown, Lisa Clark, Valentin Deaconu, Ruy Exel, Astrid an
Huef, Alex Kumjian, Paul Muhly and Dana Williams. Some parts are more identifiably
attributable to ideas I learned from others: in particular, the elementary approach to
the construction of the universal $C^*$-algebra $C^*(\mathcal{G})$ was showed to me by
Robin Deeley during a series of beautiful graduate-level lectures he gave at the research
event \emph{Refining $C^*$-algebraic invariants for dynamics using $KK$-theory} at the
MATRIX facility of the University of Melbourne in July 2016. Robin tells me that the idea
came to him in turn from lecture notes of Ian Putnam.

\chapter{\texorpdfstring{\'Etale}{Etale} groupoids}\label{ch:groupoids}

\section{What is a groupoid?}

The following definition of a groupoid comes from \cite{Hahn} (see
\cite[page~7]{Paterson}); Hahn himself attributes it to a conversation with G. Mackey.
This is a fairly minimal set of axioms, so optimal for the purposes of checking whether a
given object is a groupoid, but I refer the reader forward to Remark~\ref{rmk:gpd verbose
axioms} for an equivalent, but less efficient, list of axioms that might provide a little
more intuition.

\begin{definition}\label{dfn:gpd}
A \emph{groupoid} is a set $\mathcal{G}$ together with a distinguished subset
$\mathcal{G}^{(2)} \subseteq \mathcal{G} \times \mathcal{G}$, a multiplication map
$(\alpha,\beta) \mapsto \alpha\beta$ from $\mathcal{G}^{(2)}$ to $\mathcal{G}$ and an
inverse map $\gamma \mapsto \gamma^{-1}$ from $\mathcal{G}$ to $\mathcal{G}$ such that
\begin{enumerate}
\item\label{it:gpd0} $(\gamma^{-1})^{-1} = \gamma$ for all $\gamma \in \mathcal{G}$;
\item\label{it:gpd1} if $(\alpha,\beta)$ and $(\beta,\gamma)$ belong to
    $\mathcal{G}^{(2)}$, then $(\alpha\beta,\gamma)$ and $(\alpha,\beta\gamma)$
    belong to $\mathcal{G}^{(2)}$, and $(\alpha\beta)\gamma = \alpha(\beta\gamma)$;
    and
\item\label{it:gpd2} $(\gamma,\gamma^{-1}) \in \mathcal{G}^{(2)}$ for all $\gamma \in
    \mathcal{G}$, and for all $(\gamma,\eta) \in \mathcal{G}^{(2)}$, we have
    $\gamma^{-1} (\gamma \eta) = \eta$ and $(\gamma\eta)\eta^{-1} = \gamma$.
\end{enumerate}
\end{definition}

Axiom~(\ref{it:gpd1}) shows that for products of three groupoid elements, there is no
ambiguity in dropping the parentheses (as we do in groups), and simply writing
$\alpha\beta\gamma \coloneq (\alpha\beta)\gamma$.

To get a feeling for groupoids, we begin by exploring some of the consequences of the
above axioms.

Given a groupoid $\mathcal{G}$ we shall write $\mathcal{G}^{(0)} \coloneq
\{\gamma^{-1}\gamma \mid \gamma \in \mathcal{G}\}$ and refer to elements of
$\mathcal{G}^{(0)}$ as \emph{units} and to $\mathcal{G}^{(0)}$ itself as the \emph{unit
space}. Since $(\gamma^{-1})^{-1} = \gamma$ for all $\gamma$, we also have
$\mathcal{G}^{(0)} = \{\gamma\gamma^{-1} \mid \gamma \in \mathcal{G}\}$. We define $r,s :
\mathcal{G} \to \mathcal{G}^{(0)}$ by
\[
r(\gamma) \coloneq \gamma\gamma^{-1}\qquad\text{ and }\qquad s(\gamma) \coloneq \gamma^{-1}\gamma
\]
for all $\gamma \in \mathcal{G}$.

\begin{lemma}\label{lem:unique inverse}
If $\mathcal{G}$ is a groupoid and $\gamma \in \mathcal{G}$, then $(r(\gamma),\gamma)$
and $(\gamma, s(\gamma))$ belong to $\mathcal{G}^{(2)}$, and
\[
r(\gamma)\gamma = \gamma = \gamma s(\gamma).
\]
We have $r(\gamma^{-1}) = s(\gamma)$ and $s(\gamma^{-1}) = r(\gamma)$. Moreover,
$\gamma^{-1}$ is the unique element such that $(\gamma,\gamma^{-1}) \in
\mathcal{G}^{(2)}$ and $\gamma \gamma^{-1} = r(\gamma)$, and also the unique element such
that $(\gamma^{-1},\gamma) \in \mathcal{G}^{(2)}$ and $\gamma^{-1}\gamma = s(\gamma)$.
\end{lemma}
\begin{proof}
The first statement, is just axiom~(\ref{it:gpd2}) with $\eta = \gamma^{-1}$.

We have $r(\gamma^{-1}) = \gamma^{-1}(\gamma^{-1})^{-1} = \gamma^{-1}\gamma = s(\gamma)$
by axiom~(\ref{it:gpd0}).

If $(\gamma,\alpha) \in \mathcal{G}^{(2)}$ and $\gamma\alpha = r(\gamma) =
\gamma\gamma^{-1}$, then axiom~(\ref{it:gpd1}) shows that $(\gamma^{-1}\gamma, \alpha)
\in \mathcal{G}^{(2)}$ and $\alpha = \gamma^{-1}\gamma\alpha = \gamma^{-1}r(\gamma) =
\gamma^{-1}s(\gamma^{-1}) = \gamma^{-1}$. A similar argument shows that $\alpha\gamma =
s(\gamma)$ only for $\alpha = \gamma^{-1}$.
\end{proof}

We also quickly see that groupoids have cancellation.

\begin{lemma}\label{lem:gpd cancellation}
Let $\mathcal{G}$ be a groupoid. Suppose that $(\alpha,\gamma), (\beta,\gamma) \in
\mathcal{G}^{(2)}$ and that $\alpha \gamma = \beta\gamma$. Then $\alpha = \beta$.
Similarly if $(\gamma,\alpha), (\gamma,\beta) \in \mathcal{G}^{(2)}$ and $\gamma\alpha =
\gamma\beta$ then $\alpha = \beta$.
\end{lemma}
\begin{proof}
If $\alpha\gamma = \beta\gamma$, then axioms (\ref{it:gpd1})~and~(\ref{it:gpd2}) show
that $\alpha = \alpha\gamma\gamma^{-1} = \beta\gamma\gamma^{-1} = \beta$.
\end{proof}

\begin{lemma}\label{lem:gpd properties}
Let $\mathcal{G}$ be a groupoid. Then $(\alpha,\beta) \in \mathcal{G}^{(2)}$ if and only
if $s(\alpha) = r(\beta)$. We have
\begin{enumerate}
\item\label{it:comp r,s} $r(\alpha\beta) = r(\alpha)$ and $s(\alpha\beta) = s(\beta)$
    for all $(\alpha,\beta) \in \mathcal{G}^{(2)}$;
\item\label{it:comp inv} $(\alpha\beta)^{-1} = \beta^{-1}\alpha^{-1}$ for all
    $(\alpha,\beta) \in \mathcal{G}^{(2)}$; and
\item\label{it:unit r,s} $r(x) = x = s(x)$ for all $x\in \mathcal{G}^{(0)}$.
\end{enumerate}
\end{lemma}
\begin{proof}
First suppose that $s(\alpha) = r(\beta)$; that is, $\alpha^{-1}\alpha =
\beta\beta^{-1}$. In particular, we have $(\alpha, \beta\beta^{-1}) =
(\alpha,\alpha^{-1}\alpha) \in \mathcal{G}^{(2)}$. Since $(\beta\beta^{-1},\beta) \in
\mathcal{G}^{(2)}$ as well, we deduce from axiom~(\ref{it:gpd1}) that
$(\alpha,\beta\beta^{-1}\beta)$, which is just $(\alpha,\beta)$, belongs to
$\mathcal{G}^{(2)}$. On the other hand, if $(\alpha,\beta) \in \mathcal{G}^{(2)}$, then
$(\alpha^{-1}, \alpha\beta) \in \mathcal{G}^{(2)}$ with $\alpha^{-1}\alpha \beta = \beta
= r(\beta) \beta$, and so Lemma~\ref{lem:gpd cancellation} shows that $s(\alpha) =
\alpha^{-1}\alpha = r(\beta)$.

For~(\ref{it:comp r,s}) note that axiom~(\ref{it:gpd1}) shows that
$(r(\alpha),\alpha\beta)$ belongs to $\mathcal{G}^{(2)}$ and $r(\alpha)(\alpha\beta) =
(r(\alpha)\alpha)\beta = \alpha\beta = r(\alpha\beta)(\alpha\beta)$. So
Lemma~\ref{lem:gpd cancellation} shows that $r(\alpha) = r(\alpha\beta)$. A similar
argument gives $s(\beta) = s(\alpha\beta)$.

For~(\ref{it:comp inv}) we use~(\ref{it:comp r,s}) to see that $(\beta^{-1},\alpha^{-1})$
and $(\alpha\beta, \beta^{-1}\alpha^{-1})$ belong to $\mathcal{G}^{(2)}$. Since
$r(\beta^{-1}) = s(\beta)$ and since~(\ref{it:comp r,s}) gives $s(\alpha\beta\beta^{-1})
= s(\beta^{-1}) = r(\alpha^{-1})$, we can use~(\ref{it:comp r,s}) twice more to see that
the products $(\alpha\beta\beta^{-1})\alpha^{-1}$ and
$(\alpha\beta)(\beta^{-1}\alpha^{-1})$ make sense and are equal. We have
$(\alpha\beta\beta^{-1})\alpha^{-1} = \alpha \alpha^{-1} = r(\alpha) = r(\alpha\beta)$,
and so uniqueness in the final statement of Lemma~\ref{lem:unique inverse} implies that
$\beta^{-1}\alpha^{-1} = (\alpha\beta)^{-1}$.

For~(\ref{it:unit r,s}), fix $x \in \mathcal{G}^{(0)}$, say $x = \gamma^{-1}\gamma$.
Then~(\ref{it:comp r,s}) shows that $r(x) = r(\gamma^{-1}) = s(\gamma) =
\gamma^{-1}\gamma = x$. Similarly, $s(x) = s(\gamma^{-1}\gamma) = s(\gamma) =
\gamma^{-1}\gamma = x$.
\end{proof}

In many places you will find the definition of a groupoid summarised with the pithy ``a
groupoid is a small category with inverses." The above results should convince you that
this is equivalent to Definition~\ref{dfn:gpd}. It's a slick definition, but if it means
much to you, then you probably already knew what a groupoid was anyway\dots

Lemma~\ref{lem:gpd properties} shows that if $\mathcal{G}$ is a groupoid, then
$\mathcal{G}^{(2)} = \{(\alpha,\beta) \in \mathcal{G} \times \mathcal{G} \mid s(\alpha) =
r(\beta)\}$. When describing a groupoid (in an example) it is often convenient, and more
helpful, to make use of this: we typically specify the set $\mathcal{G}$, the
distinguished subset $\mathcal{G}^0$ and the maps $r,s : \mathcal{G} \to
\mathcal{G}^{(0)}$ first; and then specify an associative multiplication map from
$\{(\alpha,\beta) \mid s(\alpha) = r(\beta) \in \mathcal{G}^{(0)}\}$ to $\mathcal{G}$
satisfying $r(\alpha\beta) = r(\alpha)$, $s(\alpha\beta) = s(\beta)$ and $r(\alpha)\alpha
= \alpha = \alpha s(\alpha)$, and specify an inverse map satisfying $s(\alpha^{-1}) =
r(\alpha)$, $r(\alpha^{-1}) = s(\alpha)$ and $\alpha^{-1}\alpha = s(\alpha)$ and
$\alpha\alpha^{-1} = r(\alpha)$. Using the results above, you should be able to convince
yourself that this is equivalent to Definition~\ref{dfn:gpd}. We will specify groupoids
this way throughout these notes.

\begin{remark}\label{rmk:gpd verbose axioms}
An earlier version of these notes contained the following definition of a groupoid: A
groupoid is a set $\mathcal{G}$ with a distinguished subset $\mathcal{G}^{(0)}$, maps
$r,s : \mathcal{G} \to \mathcal{G}^{(0)}$, a map $(\alpha,\beta) \mapsto \alpha\beta$
from $\{(\alpha,\beta) \in \mathcal{G} \times \mathcal{G} \mid s(\alpha) = r(\beta)\}$ to
$\mathcal{G}$ and a map $\gamma \mapsto \gamma^{-1}$ from $\mathcal{G}$ to $\mathcal{G}$
with the following properties:
\begin{itemize}
\item[(G1)] $r(x) = x = s(x)$ for all $x\in \mathcal{G}^{(0)}$;
\item[(G2)] $r(\gamma)\gamma = \gamma = \gamma s(\gamma)$ for all $\gamma \in
    \mathcal{G}$;
\item[(G3)] $r(\gamma^{-1}) = s(\gamma)$ and $s(\gamma^{-1}) = r(\gamma)$ for all
    $\gamma\in \mathcal{G}$;
\item[(G4)] $\gamma^{-1}\gamma = s(\gamma)$ and $\gamma\gamma^{-1} = r(\gamma)$ for
    all $\gamma \in \Gamma$;
\item[(G5)] $r(\alpha\beta) = r(\alpha)$ and $s(\alpha\beta) = s(\beta)$ whenever
    $s(\alpha) = r(\beta)$; and
\item[(G6)] $(\alpha\beta)\gamma = \alpha(\beta\gamma)$ whenever $s(\alpha) =
    r(\beta)$ and $s(\beta) = r(\gamma)$.
\end{itemize}
Lemmas \ref{lem:unique inverse}~and~\ref{lem:gpd properties} show that every groupoid has
these properties, and it is not hard to check that given the structure above, putting
$\mathcal{G}^{(2)} := \{(\alpha,\beta) \in \mathcal{G} \times \mathcal{G} \mid s(\alpha)
= r(\beta)\}$ yields a groupoid according to Definition~\ref{dfn:gpd}. (The only thing to
check is that (G1)--(G6) force $(\gamma^{-1})^{-1} = \gamma$; and since~(G2),
(G4)~and~(G6) ensure cancellativity just as in Lemma~\ref{lem:gpd cancellation}, this
follows from the observation that (G3)~and~(G4) give $(\gamma^{-1})^{-1}\gamma^{-1} =
s(\gamma^{-1}) = r(\gamma) = \gamma\gamma^{-1}$.) In the end, I decided to stick with
Definition~\ref{dfn:gpd} because it is the definition that appears in the literature. But
I still feel that (G1)--(G6), while not maximally efficient, are a good way to present
groupoids to an audience unfamiliar with them and trying to come to grips with them: this
presentation emphasises the point of view of a groupoid as a collection of reversible
arrows, carrying an associative composition operation, between points in
$\mathcal{G}^{(0)}$ (see Example~\ref{eg:trans gpd} below).
\end{remark}

\begin{corollary}\label{cor:idempotents}
If $\mathcal{G}$ is a groupoid then $\mathcal{G}^{(0)} = \{\gamma \in \mathcal{G} \mid
(\gamma,\gamma) \in \mathcal{G}^{(2)}\text{ and } \gamma^2 = \gamma\}$.
\end{corollary}
\begin{proof}
Lemma~\ref{lem:unique inverse} combined with Lemma~\ref{lem:gpd properties}(\ref{it:unit
r,s}) gives $\subseteq$. For $\supseteq$, suppose that $(\gamma,\gamma) \in
\mathcal{G}^{(2)}$ and $\gamma^2 = \gamma$. Then $\gamma^2 = \gamma = \gamma s(\gamma)$
by Lemma~\ref{lem:unique inverse}, and then Lemma~\ref{lem:gpd cancellation} shows that
$\gamma = s(\gamma) \in \mathcal{G}^{(0)}$.
\end{proof}

\begin{example}
Every group $\Gamma$ can be viewed as a groupoid, with $\Gamma^{(0)} = \{e\}$,
multiplication given by the group operation, and inversion the usual group inverse. A
groupoid is a group if and only if its unit space is a singleton.
\end{example}

\begin{example}[Group bundles]\label{ex:group bundle}
Let $X$ be a set, and for each $x \in X$, let $\Gamma_x$ be a group. Let $\mathcal{G}
\coloneq \bigcup_{x \in X} \{x\} \times \Gamma_x$. This is a groupoid with
$\mathcal{G}^{(0)} = \{(x, e_{\Gamma_x}) \mid x \in X\}$ identified with $X$, $r(x, g) =
x = s(x, g)$, $(x, g)(x, h)= (x,gh)$ and $(x,g)^{-1} = (x, g^{-1})$.
\end{example}

\begin{example}[Matrix groupoids]
Fix $N \ge 1$. Define
\[
R_N \coloneq \{1, \dots, N\} \times \{1, \dots, N\}.
\]
Put $R_N^{(0)} = \{(i,i) \mid i \le N\}$, $r(i,j) = (i,i)$, $s(i,j) = (j,j)$ and
$(i,j)(j,k) = (i,k)$. Then $R_N$ is a groupoid, and $(i,j)^{-1} = (j,i)$ for all $i,j$.
We usually identify $R_N^{(0)}$ with $\{1, \dots, N\}$ in the obvious way.
\end{example}

\begin{example}
There was nothing special about $\{1, \dots, N\}$. For any set $X$, the set $R_X \coloneq
X \times X$ is a groupoid with operations analogous to those above. Again, we identify
$R_X^{(0)}$ with $X$.
\end{example}

\begin{example}[Equivalence relations]
More generally again, if $R$ is an equivalence relation on a set $X$, then $R^{(0)}
\coloneq \{(x,x) \mid x \in X\}$ is contained in $R$ by reflexivity; we identify
$R^{(0)}$ with $X$ again. The maps $r(x,y) = x$, $s(x,y) = y$, $(x,y)(y,z) = (x,z)$ and
$(x,y)^{-1} = (y,x)$ make $R$ into a groupoid.
\end{example}

If $R$ is an equivalence relation on $X$, then the map $\gamma \mapsto (r(\gamma),
s(\gamma))$ from $R$ to $R^{(0)} = X$ is the identity map from $R$ to $R$.

Given groupoids $\mathcal{G}$ and $\mathcal{H}$, we call a map $\phi : \mathcal{G} \to
\mathcal{H}$ a \emph{groupoid homomorphism} if $(\phi \times \phi)(\mathcal{G}^{(2)})
\subseteq \mathcal{H}^{(2)}$ and $\phi(\alpha)\phi(\beta) = \phi(\alpha\beta)$ for all
$(\alpha,\beta) \in \mathcal{G}^{(2)}$.

\begin{lemma}
If $\mathcal{G}$ and $\mathcal{H}$ are groupoids and $\phi : \mathcal{G} \to \mathcal{H}$
is a groupoid homomorphism, then $\phi(\mathcal{G}^{(0)}) \subseteq \mathcal{H}^{(0)}$.
We have $\phi(r(\gamma)) = r(\phi(\gamma))$, $\phi(s(\gamma)) = s(\phi(\gamma))$ and
$\phi(\gamma^{-1}) = \phi(\gamma)^{-1}$ for all $\gamma \in \mathcal{G}$.
\end{lemma}
\begin{proof}
For $u \in \mathcal{G}^{(0)}$ we have $\phi(u)^2 = \phi(u^2) = \phi(u)$, so $\phi(u)$ is
idempotent and therefore a unit by Corollary~\ref{cor:idempotents}. For $\gamma \in
\mathcal{G}$ we have
\[
\phi(\gamma)\phi(s(\gamma)) = \phi(\gamma s(\gamma)) = \phi(\gamma) = \phi(\gamma)s(\phi(\gamma))
\]
and similarly $\phi(r(\gamma)) \phi(\gamma) = r(\phi(\gamma)) \phi(\gamma)$. So
Lemma~\ref{lem:gpd cancellation} shows that $\phi(s(\gamma)) = s(\phi(\gamma))$ and
$\phi(r(\gamma)) = r(\phi(\gamma))$. We then have $\phi(\gamma)\phi(\gamma^{-1}) =
\phi(\gamma\gamma^{-1}) = \phi(s(\gamma)) = s(\phi(\gamma))$, and so the uniqueness
assertion in Lemma~\ref{lem:unique inverse} shows that $\phi(\gamma^{-1}) =
\phi(\gamma)^{-1}$.
\end{proof}

\begin{lemma}\label{lem:R(G)}
If $\mathcal{G}$ is a groupoid, then $R = R(\mathcal{G}) \subseteq \mathcal{G}^{(0)}
\times \mathcal{G}^{(0)}$ defined by $R = \{(r(\gamma), s(\gamma)) \mid \gamma \in
\mathcal{G}\}$ is an equivalence relation on $\mathcal{G}^{(0)}$, and $\gamma \mapsto
(r(\gamma), s(\gamma))$ is a surjective groupoid homomorphism from $\mathcal{G}$ to $R$.
We call $R$ the \emph{equivalence relation of $\mathcal{G}$}.
\end{lemma}
\begin{proof}
We have $(x,x) = (r(x), s(x)) \in R$ for all $x \in \mathcal{G}^{(0)}$, so $R$ is
reflexive. We have $(s(\gamma), r(\gamma)) = (r(\gamma^{-1}), s(\gamma^{-1}))$ for each
$\gamma$, so $R$ is symmetric. And if $(x,y), (y,z) \in R$, say $x = r(\alpha)$, $y =
s(\alpha)$, and $y = r(\beta)$, $z = s(\beta)$, then $(x,z) = (r(\alpha\beta),
s(\alpha\beta)) \in R$, so $R$ is transitive. That is, $R$ is an equivalence relation.
The map $\gamma \mapsto (r(\gamma), s(\gamma))$ is surjective by definition of $R$, and
is a homomorphism by Lemma~\ref{lem:gpd properties}(\ref{it:comp r,s}).
\end{proof}

We say that a groupoid $\mathcal{G}$ is \emph{principal} if $\gamma \mapsto (r(\gamma),
s(\gamma))$ is injective.

\begin{lemma}\label{lem:principal<->EqRel}
A groupoid $\mathcal{G}$ is algebraically isomorphic to an equivalence relation if and
only if it is principal, in which case it is algebraically isomorphic to
$R(\mathcal{G})$.
\end{lemma}
\begin{proof}
We just saw that equivalence relations are always principal, so the ``only if"
implication is clear. For the ``if" implication, suppose that $\mathcal{G}$ is principal.
Lemma~\ref{lem:R(G)} shows that $\gamma \mapsto (r(\gamma), s(\gamma))$ is a surjective
groupoid homomorphism onto $R(\mathcal{G})$, and it is injective because $\mathcal{G}$ is
principal.
\end{proof}

From the preceding lemma, it may seem a little strange, in the first instance, to make
the distinction between an ``equivalence relation" and a ``principal groupoid." Indeed,
algebraically there is no difference. But when we start introducing topology into the mix
the distinction makes sense. The term ``equivalence relation" is reserved for principal
groupoids that have the relative topology inherited from the product topology on
$\mathcal{G}^{(0)} \times \mathcal{G}^{(0)}$, whereas ``principal groupoids" may have
finer topologies.

\begin{example}[Transformation groupoids]\label{eg:trans gpd}
Let $X$ be a set, and let $\Gamma$ be a group acting on $X$ by bijections. Let
$\mathcal{G} \coloneq \Gamma \times X$, put $\mathcal{G}^{(0)} = \{e\} \times X$
(identified with $X$ in the obvious way), define $r(g,x) \coloneq g\cdot x$ and $s(g,x)
\coloneq x$, and define $(g, h\cdot x)(h,x) \coloneq (gh, x)$ and $(g,x)^{-1} \coloneq
(g^{-1}, g\cdot x)$. Then $\mathcal{G}$ is a groupoid, called the \emph{transformation
groupoid}.
\end{example}

\begin{example}[Deaconu--Renault groupoids]
Let $X$ be a set, $\Gamma$ an abelian group, and $S \subseteq \Gamma$ a subsemigroup of
$\Gamma$ that contains $0$. Suppose that $S$ acts on $X$ in the sense that we have maps
$x \mapsto s \cdot x$ from $X \to X$ satisfying $s\cdot(t\cdot x) = (s+t)\cdot x$, and $0
\cdot x = x$ for all $x$. Let
\[
\mathcal{G} \coloneq \{(x, s-t, y) \mid s \cdot x = t\cdot y\}.
\]
Put $\mathcal{G}^{(0)} = \{(x,e,x) \mid x \in X\}$ and identify it with $X$. Define $r(x,
g, y) = x$ and $s(x, g, y) = y$, and put $(x, g, y)^{-1} = (y, -g, x)$. Suppose that $s_1
\cdot x = t_1 \cdot y$ and $s_2 \cdot y = t_2 \cdot z$. Then
\[
(s_1 + s_2) \cdot x
    = s_2 \cdot (s_1 \cdot x)
    = s_2 \cdot (t_1 \cdot y)
    = (t_1 + s_2) \cdot y
    = t_1 \cdot (s_2 \cdot y)
    = t_1 \cdot (t_2 \cdot z)
    = (t_1 + t_2) \cdot z.
\]
So we can define $(x, g, y)(y, h, z) = (x, gh, z)$. Under these operations, $\mathcal{G}$
is a groupoid.
\end{example}

\section{Isotropy}

Before discussing the isotropy of a groupoid, we introduce some standard notation.

\textbf{Notation.} For $x \in \mathcal{G}^{(0)}$ we write $\mathcal{G}_x \coloneq
\{\gamma \in \mathcal{G} \mid s(\gamma) = x\}$, and $\mathcal{G}^x \coloneq \{\gamma \in
\mathcal{G} \mid r(\gamma) = x\}$. In some articles you might see these sets denoted
$\mathcal{G} x$ and $x \mathcal{G}$; this is helpful and sensible notation, but the
literature is well-established and the superscript-subscript notation is quite standard,
so we'll stick with it here. We write $\mathcal{G}^y_x \coloneq \mathcal{G}_x \cap
\mathcal{G}^y$.

We will, however, write $UV$ for $\{\alpha\beta \mid \alpha \in U, \beta \in V, s(\alpha)
= r(\beta)\}$ for any pair of subsets $U, V$ of a groupoid $\mathcal{G}$.

\smallskip

If $\mathcal{G}$ is a groupoid, then the \emph{isotropy subgroupoid} of $\mathcal{G}$, or
just the isotropy of $\mathcal{G}$, is the subset $\operatorname{Iso}(\mathcal{G}) =
\bigcup_{x \in \mathcal{G}^{(0)}} \mathcal{G}^x_x = \{\gamma \in \mathcal{G} \mid
r(\gamma) = s(\gamma)\}$. It is straightforward to see that the isotropy subgroupoid
really is a subgroupoid; indeed, it is a group bundle as in Example~\ref{ex:group
bundle}.

Clearly $\mathcal{G}^{(0)} \subseteq \operatorname{Iso}(\mathcal{G})$.

\begin{lemma}
A groupoid $\mathcal{G}$ is principal if and only if $\operatorname{Iso}(\mathcal{G}) =
\mathcal{G}^{(0)}$.
\end{lemma}
\begin{proof}
If $\mathcal{G}$ is principal and $\gamma \in \operatorname{Iso}(\mathcal{G})$, then $x =
r(\gamma)$ satisfies $(r(\gamma), s(\gamma)) = (r(x), s(x))$, and since $\mathcal{G}$ is
principal it follows that $\gamma = x \in \mathcal{G}^{(0)}$. Now suppose that
$\operatorname{Iso}(\mathcal{G}) = \mathcal{G}^{(0)}$, and that $(r(\gamma), s(\gamma)) =
(r(\alpha), s(\alpha))$. Then $\alpha\gamma^{-1} \in \operatorname{Iso}(\mathcal{G})$ and
therefore $\alpha\gamma^{-1} = r(\alpha)$. So Lemma~\ref{lem:unique inverse} forces
$\alpha = \gamma$.
\end{proof}

\begin{example}
If $X$ is a set, and $\Gamma$ is a group acting on $X$, then, as usual, the isotropy
subgroup of $\Gamma$ at $x \in X$ is $\Gamma_x \coloneq \{g \in \Gamma \mid g \cdot x =
x\}$. The isotropy subgroupoid of the transformation groupoid $\mathcal{G}$ is then
$\bigcup_{x \in X} \{x\} \times \Gamma_x$, the union of the isotropy subgroups associated
to points $x \in X$.
\end{example}

\begin{lemma}\label{lem:isotropy isomorphisms}
If $\mathcal{G}$ is a groupoid and $\gamma \in \mathcal{G}$, then the map
$\operatorname{Ad}_\gamma : \alpha \mapsto \gamma\alpha\gamma^{-1}$ is a group
isomorphism from $\mathcal{G}^{s(\gamma)}_{s(\gamma)}$ to
$\mathcal{G}^{r(\gamma)}_{r(\gamma)}$.
\end{lemma}
\begin{proof}
The map $\operatorname{Ad}_{\gamma^{-1}}$ is plainly an inverse for
$\operatorname{Ad}_\gamma$; and
\[
\operatorname{Ad}_\gamma(\alpha)\operatorname{Ad}_\gamma(\beta) = \gamma \alpha \gamma^{-1}\gamma \beta \gamma^{-1} = \operatorname{Ad}_\gamma(\alpha\beta).\qedhere
\]
\end{proof}

\section{Topological groupoids}

Since we are going to be interested here in $C^*$-algebras, we will want to topologise
our groupoids.

\begin{definition}
A \emph{topological groupoid} is a groupoid $\mathcal{G}$ endowed with a locally compact
topology under which $\mathcal{G}^{(0)} \subseteq \mathcal{G}$ is Hausdorff in the
relative topology, the maps $r$, $s$ and $\gamma \mapsto \gamma^{-1}$ are continuous, and
the map $(g,h) \to gh$ is continuous with respect to the relative topology on
$\mathcal{G}^{(2)}$ as a subset of $\mathcal{G} \times \mathcal{G}$.
\end{definition}

We naturally expect that the unit space will be closed in a topological groupoid; but in
fact, this is true only when $\mathcal{G}$ is Hausdorff.

\begin{lemma}\label{lem:go closed}
If $\mathcal{G}$ is a topological groupoid, then $\mathcal{G}^{(0)}$ is closed in
$\mathcal{G}$ if and only if $\mathcal{G}$ is Hausdorff.
\end{lemma}
\begin{proof}
First suppose that $\mathcal{G}$ is Hausdorff, and suppose that $(x_i)_{i \in I}$ is a
net in $\mathcal{G}^{(0)}$ such that $x_i \to \gamma \in \mathcal{G}$. Since $r$ is
continuous, we have $x_i = r(x_i) \to r(\gamma) \in \mathcal{G}^{(0)}$. Since
$\mathcal{G}$ is Hausdorff, this limit point is unique, and we deduce that $\gamma =
r(\gamma)$.

Now suppose that $\mathcal{G}^{(0)}$ is closed. To see that $\mathcal{G}$ is Hausdorff,
it suffices to show that convergent nets have unique limit points. For this, suppose that
$(\gamma_i)_{i \in I}$ is a net and that $\gamma_i \to \alpha$ and $\gamma_i \to \beta$.
By continuity, we then have $\gamma_i^{-1}\gamma_i \to \alpha^{-1}\beta$. Since each
$\gamma_i^{-1}\gamma_i = s(\gamma_i) \in \mathcal{G}^{(0)}$ and $\mathcal{G}^{(0)}$ is
closed, we deduce that $\alpha^{-1}\beta \in \mathcal{G}^{(0)}$. Hence $\alpha = \beta$.
\end{proof}

\smallskip

Although there are many interesting and important examples without these properties, in
these notes, all the topological groupoids that I discuss will be second-countable and
Hausdorff as topological spaces. So in these notes, $\mathcal{G}^{(0)}$ is always a
closed subset of $\mathcal{G}$.

\smallskip

\begin{example}[Discrete groupoids]
Every groupoid is a topological groupoid in the discrete topology.
\end{example}

\begin{example}[Topological equivalence relations]
If $X$ is a second-countable Hausdorff space, and $R$ is an equivalence relation on $X$,
then $R$ is a topological groupoid in the relative topology inherited from $X \times X$.
\end{example}

The previous example, combined with Lemma~\ref{lem:principal<->EqRel}, shows that if
$\mathcal{G}$ is a principal groupoid, then any second-countable Hausdorff topology on
$\mathcal{G}^{(0)}$ induces a topological-groupoid structure on $\mathcal{G}$. Further,
if $\mathcal{G}$ is a principal topological groupoid, then $\gamma \mapsto (r(\gamma),
s(\gamma))$ is a bijective continuous map from $\mathcal{G}$ to the topological
equivalence relation $R(\mathcal{G})$. So the topology on $\mathcal{G}$ must be finer
than the one inherited from the product topology on $\mathcal{G}^{(0)} \times
\mathcal{G}^{(0)}$. It can be strictly finer:

\begin{example}\label{ex:2infty relation}
Let $X \coloneq \prod^\infty_{i=1} \{0,1\}$, viewed as right-infinite strings of $0$'s
and $1$'s, and given the product topology. Define an equivalence relation $R$ on $X$ by
$(x,y) \in R$ if and only if there exists $n \in \mathbb{N}$ such that $x_j = y_j$ for
all $j \ge n$. Let $\mathcal{R}_{2^\infty} \coloneq \{(x, y) \mid x \sim y\}$ be an
algebraic copy of $R$. For finite words $v,w \in \{0,1\}^n$, define $Z(v,w) = \{(vx, wx)
\mid x \in X\} \subseteq \mathcal{R}_{2^\infty}$. Observe that
\[
Z(v,w) \cap Z(v', w')
    = \begin{cases}
        Z(v,w) &\text{ if $v = v'u$ and $w = w' u$ for some $u$}\\
        Z(v',w') &\text{ if $v' = vu$ and $w' = wu$ for some $u$}\\
        \emptyset &\text{ otherwise.}
    \end{cases}
\]
So the $Z(v,w)$ form a base for a topology.

We claim that $\mathcal{R}_{2^\infty}$ is a topological groupoid in this topology. It is
Hausdorff because if $(x,y) \not= (x',y')$ then there is an $n$ such that $\big(x(0,n),
y(0,n)\big) \not= \big(x'(0,n), y'(0,n)\big)$, and then $Z\big(x(0,n), y(0,n)\big)$ and
$Z\big(x'(0,n), y'(0,n)\big)$ separate these points. The sets $Z(v,w)$ are also compact:
the map $x \mapsto (v x, w x)$ is a bijective continuous map from the compact space $X$
to the Hausdorff space $Z(v,w)$, and therefore a homeomorphism. The maps $r,s$ restrict
to homeomorphisms of $Z(v,w)$ onto $Z(v)$ and $Z(w)$, so they are continuous. Inversion
is clearly continuous. Multiplication is continuous because the pre-image of $Z(v,w)$ is
$\bigcup_{|y|=|v|} \bigcup_u \big(Z(vu, yu) \times Z(yu, wu)\big)\cap
\mathcal{R}_{2^\infty}$. This proves the claim.

We now claim that this topology is finer than the one inherited from the product
topology. To see this, consider the sequence $\gamma_n = \big((0^n10^\infty,
0^\infty)\big)^\infty_{n=1}$. In the product topology, we have $\gamma_n \to (0^\infty,
0^\infty)$. But the basic neighbourhoods of $(0^\infty, 0^\infty)$ are of the form
$Z(0^m, 0^m)$, and we have $\gamma_n \not\in Z(0^m, 0^m)$ for $n > m$. So $\gamma_n
\not\to (0^\infty, 0^\infty)$ in $\mathcal{R}_{2^\infty}$.
\end{example}

As we will see later, this example is important: its $C^*$-algebra is
$M_{2^\infty}(\mathbb{C})$.

\begin{example}
Suppose that $X$ is a second-countable locally compact Hausdorff space and that $\Gamma$
is a locally compact group acting on $X$ by homeomorphisms as in
\cite[Section~1.1]{WilliamsNotes}. Then the transformation groupoid $\mathcal{G}$ is a
topological groupoid in the product topology.
\end{example}

In these notes, a \emph{local homeomorphism} from $X$ to $Y$ is a continuous map $h : X
\to Y$ such that every $x \in X$ has an open neighbourhood $U$ such that $h(U) \subseteq
Y$ is open and $h : U \to h(U)$ is a homeomorphism.

\begin{example}
If $X$ is a second-countable locally compact Hausdorff space, $\Gamma$ is a discrete
abelian group, $S$ is a subsemigroup of $\Gamma$ containing $0$, and $S$ acts on $X$ by
local homeomorphisms, then the Deaconu--Renault groupoid becomes a locally compact
Hausdorff groupoid in the topology with basic open sets
\[
Z(U, p, q, V) = \{(x, p-q, y) \mid x \in U, y \in V, p\cdot x = q\cdot y\}
\]
indexed by pairs $U,V$ of open subsets of $X$ and pairs $p,q \in S$.
\end{example}

\section{\'Etale groupoids}

In these notes, we will focus on \'etale groupoids. These are the analogue, in the
groupoid world, of discrete groups.

\begin{definition}
A topological groupoid $\mathcal{G}$ is \emph{\'etale} if the range map $r : \mathcal{G}
\to \mathcal{G}$ is a local homeomorphism.
\end{definition}

Note: a subtle but important point is that $r$ is a local homeomorphism as a map from
$\mathcal{G}$ to $\mathcal{G}$; not just from $\mathcal{G}$ to $\mathcal{G}^{(0)}$ in the
relative topology. The first important consequence is the following.

\begin{lemma}\label{lem:go open}
If $\mathcal{G}$ is an \'etale groupoid, then $\mathcal{G}^{(0)}$ is open in
$\mathcal{G}$.
\end{lemma}
\begin{proof}
For each $\gamma \in \mathcal{G}$ chose an open $U_\gamma$ containing $\gamma$ such that
$r : U_\gamma \to r(U_\gamma)$ is a homeomorphism onto an open set. Then
$\mathcal{G}^{(0)} = \bigcup_\gamma r(U_\gamma)$ is open.
\end{proof}

\begin{example}
Every discrete groupoid is \'etale.
\end{example}

\begin{example}
The principal groupoid of Example~\ref{ex:2infty relation} is \'etale: we proved that $r$
is a homeomorphism of $Z(u,v)$ onto $Z(u) \subseteq \mathcal{G}^{(0)}$ for each $(u,v)$.
\end{example}

\begin{example}
A transformation groupoid is \'etale if and only if the acting group $\Gamma$ is
discrete.
\end{example}

\begin{example}
The Deaconu--Renault groupoid associated to an action by local homeomorphisms is always
\'etale: For each $(x, p, q, y) \in \mathcal{G}$ we can choose open neighbourhoods $U$ of
$x$ and $V$ of $y$ such that $u \mapsto p\cdot u$ is a homeomorphism from $U$ to $p\cdot
U$, and $v \mapsto q \cdot v$ is a homeomorphism of $V$ onto $p \cdot V$. Let $W = p\cdot
U \cap q\cdot V$ and let $U' \coloneq \{u \in U \mid p \cdot u \in W\}$ and $V' \coloneq
\{v \in V \mid q \cdot v \in W\}$. Then $Z(U', p,q, V')$ is an open neighbourhood of $(x,
p, q, y)$, and the range map restricts to a homeomorphism of this neighbourhood onto
$U'$.
\end{example}

\begin{example}[Graph groupoids]\label{ex:graph gpd}
Let $E$ be a directed graph with vertex set $E^0$, edge set $E^1$ and direction of edges
described by range and source maps $r,s : E^1 \to E^0$. Assume that $E$ is row-finite
with no sources (so $r^{-1}(v)$ is finite and nonempty for every vertex $v$). See
\cite{Raeburn} for background on graphs in the context of graph $C^*$-algebras; we will
use Raeburn's conventions and notation for graphs throughout. Let $E^\infty$ denote the
space of right-infinite paths in $E$, so $E^\infty = \{x_1 x_2 x_3 \cdots \mid x_i \in
E^1, r(x_{i+1}) = s(x_i)\}$. Give $E^\infty$ the topology inherited from the product
space $\prod_{i=1}^\infty E^1$. Then $E^\infty$ is a totally-disconnected locally compact
Hausdorff space, and the sets $Z(\mu) = \{\mu x \mid x \in E^\infty, r(x_1) = s(\mu)\}$
form a base of compact open sets for the topology. The map $\sigma : E^\infty \to
E^\infty$ given by $\sigma(x)_i = x_{i+1}$ is a local homeomorphism (it restricts to a
homeomorphism on $Z(\mu)$ whenever $|\mu| \ge 1$), so induces an action of $\mathbb{N}$
by local homeomorphisms. The associated Deaconu--Renault groupoid $\mathcal{G}_E = \{(x,
m-n, y) \mid \sigma^m(x) = \sigma^n(y)\}$ is called the \emph{graph groupoid} of $E$.
\end{example}

Since $\gamma \mapsto \gamma^{-1}$ is continuous and self-inverse, if $\mathcal{G}$ is
\'etale, then $s : \mathcal{G} \to \mathcal{G}^{(0)}$ is also a local homeomorphism. So
there are plenty of open sets on which $r,s$ are both homeomorphisms.

\begin{definition}
A subset $B$ of an \'etale groupoid $\mathcal{G}$ is a \emph{bisection} if there is an
open set $U$ containing $B$ such that $r : U \to r(U)$ and $s : U \to s(U)$ are both
homeomorphisms onto open subsets of $\mathcal{G}^{(0)}$.
\end{definition}

\begin{lemma}\label{lem:bisection base}
Let $\mathcal{G}$ be a second-countable Hausdorff \'etale groupoid. Then $\mathcal{G}$
has a countable base of open bisections.
\end{lemma}
\begin{proof}
Choose a countable dense subset $\{\gamma_n\}$ of $\mathcal{G}$. For each $\gamma_n$,
choose countable neighbourhood bases $\{U_{n,i}\}_i$ and $\{V_{n,i}\}_i$ at $\gamma_n$
such that $r$ is a homeomorphism of each $U_{n,i}$ onto an open set, and $s$ is a
homeomorphism of each $V_{n,i}$ onto an open set. Then $\{U_{n,i} \cap V_{n,i} \mid n,i
\in \mathbb{N}\}$ is a countable base of open bisections.
\end{proof}

\begin{corollary}\label{cor:fibres discrete}
If $\mathcal{G}$ is an \'etale groupoid, then each $\mathcal{G}_x$ and each
$\mathcal{G}^x$ is discrete in the relative topology.
\end{corollary}
\begin{proof}
For each $\gamma \in \mathcal{G}_x$, choose an open bisection $U_\gamma$ containing
$\gamma$. Then $U_\gamma \cap \mathcal{G}_x = \{\gamma\}$, and so $\{\gamma\}$ is open in
$\mathcal{G}_x$.
\end{proof}

We finish this section with the important observation that multiplication is an open map
in an \'etale groupoid. The following quick proof was shown to me by Dana Williams.

\begin{lemma}
If $\mathcal{G}$ is a topological groupoid and $r$ is an open map, then the
multiplication map on $\mathcal{G}$ is open. In particular, if $\mathcal{G}$ is \'etale,
then multiplication is an open map.
\end{lemma}
\begin{proof}
Fix open sets $U, V \subseteq \mathcal{G}$, and an element $(\alpha,\beta) \in U \times V
\cap \mathcal{G}^{(2)}$. Fix a sequence $\gamma_i$ converging to $\alpha\beta$; it
suffices to show that the $\gamma_i$ eventually belong to $UV$. Fix a descending
neighbourhood base $\{U_j\}_{j \in \mathbb{N}}$ for $\alpha$ contained in $U$. Since $r$
is an open map, each $r(U_j)$ is an open neighbourhood of $r(\alpha)$. Since $\gamma_i
\to \alpha\beta$, we have $r(\gamma_i) \to r(\alpha\beta) = r(\alpha)$, so for each $j$
we eventually have $r(\gamma_i) \in r(U_j)$. Choose $\alpha_i$ in $U$ with $r(\alpha_i) =
r(\gamma_i)$ and $\alpha_i \in U_j$ whenever $r(\gamma_i) \in r(U_j)$. Then $\alpha_i \to
\alpha$. Hence $\alpha_i^{-1}\gamma \to \beta$, and therefore $\alpha_i^{-1}\gamma \in V$
for large $i$. But then $\gamma_i = \alpha_i (\alpha_i^{-1}\gamma) \in UV$ for large $i$.
\end{proof}

\chapter{\texorpdfstring{$C^*$}{C*}-algebras and equivalence}\label{ch:C*-algs}

In this section, we will associate two $C^*$-algebras to each \'etale groupoid. As with
groups and dynamical systems, each groupoid has both a reduced $C^*$-algebra and a full
$C^*$-algebra. We first discuss the convolution product on $C_c(\mathcal{G})$ and then
its two key $C^*$-completions, $C^*(\mathcal{G})$ and $C^*_r(\mathcal{G})$. At the end of
the chapter, we discuss equivalence of groupoids and Renault's equivalence theorem for
their $C^*$-algebras.

\section{The convolution algebra}

\begin{proposition}
Let $\mathcal{G}$ be a second-countable locally compact Hausdorff \'etale groupoid. For
$f,g \in C_c(\mathcal{G})$ and $\gamma \in \mathcal{G}$, the set
\[
\{(\alpha,\beta) \in \mathcal{G}^{(2)} \mid \alpha\beta = \gamma\text{ and } f(\alpha)g(\beta) \not= 0\}
\]
is finite. The complex vector space $C_c(\mathcal{G})$ is a $^*$-algebra with
multiplication given by $(f*g)(\gamma) = \sum_{\alpha\beta = \gamma} f(\alpha)g(\beta)$
and involution $f^*(\gamma) = \overline{f(\gamma^{-1})}$. For $f,g \in C_c(\mathcal{G})$
we have $\operatorname{supp}(f*g) \subseteq
\operatorname{supp}(f)\operatorname{supp}(g)$.
\end{proposition}
\begin{proof}
If $\alpha\beta = \gamma$, then $\alpha \in G^{r(\gamma)}$ and $\beta \in G_{s(\gamma)}$.
We saw in Corollary~\ref{cor:fibres discrete} that these are discrete sets, so their
intersections with the compact sets $\operatorname{supp}(f)$ and $\operatorname{supp}(g)$
are finite. The rest is routine.
\end{proof}

\begin{remark}
The convolution formula can be (and often is) equivalently reformulated as
\[
(f*g)(\gamma) = \sum_{\alpha \in \mathcal{G}^{r(\gamma)}} f(\alpha) g(\alpha^{-1}\gamma).
\]
\end{remark}

We will see later that the multiplication operation, and also the $C^*$-norms, are
simpler for elements of $C_c(\mathcal{G})$ whose supports are contained in a bisection
than for general elements. So it's helpful to know that when $\mathcal{G}$ is \'etale,
such functions span the convolution algebra.

\begin{lemma}\label{lem:bisection span}
Suppose that $\mathcal{G}$ is a second-countable locally compact Hausdorff \'etale
groupoid. Then
\[
C_c(\mathcal{G}) = \operatorname{span}\{f \in C_c(\mathcal{G}) \mid \operatorname{supp}(f) \text{ is a bisection}\}.
\]
\end{lemma}
\begin{proof}
Fix $f \in C_c(\mathcal{G})$. By Lemma~\ref{lem:bisection base}, we can cover
$\operatorname{supp}(f)$ with open bisections, and then use compactness to pass to a
finite subcover $U_1, \dots, U_n$. Choose a partition of unity $\{h_i\}$ on $\bigcup U_i$
subordinate to the $U_i$. The pointwise products $f_i \coloneq f \cdot h_i$ belong to
$C_c(\mathcal{G})$ with $\operatorname{supp}(f_i) \subseteq U_i$, and we have $f = \sum_i
f_i$.
\end{proof}

One reason why the preceding lemma is so useful is because convolution is very easy to
compute for functions supported on bisections.

\begin{lemma}\label{lem:bisection conv}
Suppose that $\mathcal{G}$ is a second-countable locally compact Hausdorff \'etale
groupoid. If $U,V \subseteq \mathcal{G}$ are open bisections and $f,g \in
C_c(\mathcal{G})$ satisfy $\operatorname{supp}(f) \subseteq U$ and
$\operatorname{supp}(g) \subseteq V$, then $\operatorname{supp}(f*g) \subseteq UV$ and
for $\gamma = \alpha\beta \in UV$, we have
\begin{equation}\label{eq:bisection conv}
    (f*g)(\gamma) = f(\alpha)g(\beta).
\end{equation}
We have $C_c(\mathcal{G}^{(0)}) \subseteq C_c(\mathcal{G})$. If $f \in C_c(\mathcal{G})$
is supported on a bisection, then $f^**f \in C_c(\mathcal{G}^{(0)})$ is supported on
$s(\operatorname{supp}(f))$ and $(f^**f)(s(\gamma)) = |f(\gamma)|^2$ for all $\gamma \in
\operatorname{supp}(f)$. Similarly $f*f^*$ is supported on $r(\operatorname{supp}(f))$
and $(f*f^*)(r(\gamma)) = |f(\gamma)|^2$ for $\gamma \in \operatorname{supp}(f)$. For $f
\in C_c(\mathcal{G})$ and $h \in C_c(\mathcal{G}^{(0)})$, we have
\[
(h*f)(\gamma) = h(r(\gamma)) f(\gamma)\quad\text{ and }\quad (f*h)(\gamma) = f(\gamma)h(s(\gamma)).
\]
\end{lemma}
\begin{proof}
We have $f(\gamma) = \sum_{\eta\zeta = \gamma} f(\eta)g(\zeta)$. For any $\eta,\zeta$
appearing in the sum, we have $r(\eta) = r(\gamma)$ and $s(\zeta) = s(\gamma)$. Since $f$
and $g$ are supported on bisections, and since $\alpha \in \mathcal{G}^{r(\gamma)}$ and
$\beta \in \mathcal{G}^{s(\gamma)}$ are contained in $\operatorname{supp}(f)$ and
$\operatorname{supp}(g)$ respectively, it follows that the only term in the sum that can
be nonzero is $f(\alpha)g(\beta)$.

Since $\mathcal{G}^{(0)}$ is open by Lemma~\ref{lem:go open}, we can regard
$C_c(\mathcal{G}^{(0)})$ as a subalgebra of $C_c(\mathcal{G})$ in the usual way: for $f
\in C_c(\mathcal{G}^{(0)})$ the corresponding element of $C_c(\mathcal{G})$ agrees with
$f$ on $\mathcal{G}^{(0)}$ and vanishes on its complement. The remaining statements
follow from the convolution formula~\eqref{eq:bisection conv}.
\end{proof}

\begin{example}
Consider the matrix groupoid $R_N$. Since this is a finite discrete groupoid, we have
$C_c(R_N) = \operatorname{span}\{1_{(i,j)} \mid i,j \le N\}$. Lemma~\ref{lem:bisection
conv} shows that $1_{(i,j)}1_{(k,l)} = \delta_{j,k} 1_{(i,l)}$. So the $1_{(i,j)}$ are
matrix units, and $C_c(R_N) \cong M_N(\mathbb{C})$.
\end{example}

\begin{example}
If $\Gamma$ is a discrete group, regarded as a groupoid, then its convolution algebra as
described above is identical to the usual group algebra $C_c(\Gamma)$.
\end{example}

\begin{example}\label{eg:trans gpd conv}
Let $\Gamma$ be a discrete group acting on a compact space $X$, and let $\mathcal{G}$ be
the associated transformation groupoid. Let $\alpha$ be the action of $\Gamma$ on $C(X)$
induced by the $\Gamma$ action on $X$. Let $C_c(\Gamma, C(X))$ be the convolution algebra
of the $C^*$-dynamical system $(C(X), \Gamma, \alpha)$ described in
\cite[Section~1.3.2]{WilliamsNotes}. Then there is an isomorphism $\omega :
C_c(\mathcal{G}) \to C_c(\Gamma, C(X))$ given by $\omega(f)(g)(x) = f(g,x)$ for all $f
\in C_c(\mathcal{G})$, all $g \in \Gamma$ and all $x \in X$.
\end{example}

\section{The full \texorpdfstring{$C^*$}{C*}-algebra}

There are two ways to describe the full $C^*$-algebra of a discrete group. The first is
as the universal $C^*$-algebra generated by a unitary representation of $\Gamma$. The
second is as the universal $C^*$-algebra generated by a $^*$-representation of
$C_c(\Gamma)$.

There is a version of the first description for groupoids, which we will discuss briefly
later; but it's a little technical. The second description, on the other hand,
generalises nicely, and is the one we'll make use of in these notes---a luxury that we
can afford because we are sticking to \'etale groupoids throughout. The following
elementary construction of the full $C^*$-norm on the convolution algebra of an \'etale
groupoid was shown to me by Robin Deeley. It appears, for non-Hausdorff groupoids, in Ruy
Exel's paper \cite[Definition~3.17]{Exel}; thanks to Eduardo Scarparo for reminding me.

\begin{proposition}\label{prp:rep bound}
Let $\mathcal{G}$ be a second-countable locally compact Hausdorff \'etale groupoid. For
each $f \in C_c(\mathcal{G})$, there is a constant $K_f \ge 0$ such that $\|\pi(f)\| \le
K_f$ for every $^*$-representation $\pi : C_c(\mathcal{G}) \to \mathcal{B}(\mathcal{H})$
of $C_c(\mathcal{G})$ on Hilbert space. If $f$ is supported on a bisection, we can take
$K_f = \|f\|_\infty$.
\end{proposition}
\begin{proof}
Fix $f \in C_c(\mathcal{G})$ and use Lemma~\ref{lem:bisection span} to write $f =
\sum^n_{i=1} f_i$ with each $f_i$ supported on a bisection. Define $K_f \coloneq
\sum^n_{i=1} \|f_i\|_\infty$.

Suppose that $\pi$ is a $^*$-representation. Then $\pi|_{C_c(\mathcal{G}^{(0)})}$ is a
$^*$-representation of the commutative $^*$-algebra $C_c(\mathcal{G}^{(0)})$, and so
$\|\pi(h)\| \le \|h\|_\infty$ for every $h \in C_c(\mathcal{G}^{(0)})$.
Lemma~\ref{lem:bisection conv} implies first that each $f_i^* * f_i$ is supported on
$\mathcal{G}^{(0)}$ and second that $\|f_i^* * f_i\|_\infty = \|f_i\|_\infty^2$. So
\[
\|\pi(f_i)\|^2
    = \|\pi(f_i^**f_i)\|
    \le \|f_i^* * f_i\|_\infty
    = \|f_i\|_\infty^2,
\]
and so each $\|\pi(f_i)\| \le \|f_i\|_\infty$. Now the triangle inequality gives
$\|\pi(f)\| \le K_f$. If $f$ is supported on a bisection, then there is just one term in
the sum, so $K_f = \|f\|_\infty$.
\end{proof}

This allows us to define the universal $C^*$-algebra of an \'etale groupoid.

\begin{theorem}\label{thm:C*(G)}
Suppose that $\mathcal{G}$ is a second-countable locally compact Hausdorff \'etale
groupoid. There exist a $C^*$-algebra $C^*(\mathcal{G})$ and a $^*$-homomorphism
$\pi_{\max} : C_c(\mathcal{G}) \to C^*(\mathcal{G})$ such that
$\pi_{\max}(C_c(\mathcal{G}))$ is dense in $C^*(\mathcal{G})$, and such that for every
$^*$-representation $\pi : C_c(\mathcal{G}) \to \mathcal{B}(\mathcal{H})$ there is a
representation $\psi$ of $C^*(\mathcal{G})$ such that $\psi \circ \pi_{\max} = \pi$. The
norm on $C^*(\mathcal{G})$ satisfies
\[
\|\pi_{\max}(f)\| = \sup\big\{\|\pi(f)\| \mathbin{\big|} \pi\text{ is a $^*$-representation of $C_c(\mathcal{G})$}\big\}
\]
for all $f \in C_c(\mathcal{G})$.
\end{theorem}
\begin{proof}
For each $a \in C_c(\mathcal{G})$, Proposition~\ref{prp:rep bound} shows that the set
\[
\{\pi(a) \mid \pi \text{ is a $^*$-representation of $C_c(\mathcal{G})$}\}
\]
is bounded above, and it is nonempty because of the zero representation. So we can define
$\rho : C_c(\mathcal{G}) \to [0,\infty)$ by
\[
\rho(f) \coloneq \sup\big\{\|\pi(f)\| \mathbin{\big|} \pi : C_c(\mathcal{G}) \to \mathcal{B}(\mathcal{H})\text{ is a $^*$-representation}\big\}.
\]
It is routine to check that $\rho$ is a $^*$-seminorm satisfying the $C^*$-identity using
that each $f \mapsto \|\pi(f)\|$ has the same properties. So we can define
$C^*(\mathcal{G})$ to be the completion of the quotient of $C_c(\mathcal{G})$ by $N
\coloneq \big\{f \in C_c(\mathcal{G}) \mathbin{\big|} \|f\| = 0\big\}$ in the
pre-$C^*$-norm $\|\cdot\|$ induced by $\rho$. We define $\pi_{\max}(f) \coloneq f + N \in
C^*(\mathcal{G})$.

By construction of $\rho$, if $\pi$ is a $^*$-representation of $C^*(\mathcal{G})$, then
$\|\pi(f)\| \le \rho(f) = \|\pi_{\max}(f)\|$ for all $f \in C_c(\mathcal{G})$. This
implies that there is a well-defined norm-decreasing linear map $\psi : C^*(\mathcal{G})
\to \mathcal{B}(\mathcal{H})$ satisfying $\psi \circ \pi_{\max} = \pi$. Continuity then
ensures that this $\psi$ is a $C^*$-homomorphism.
\end{proof}

Of course, we expect that $\pi_{\max}$ is injective; we will prove this in the next
section.

It is not immediately obvious that the norm defined in Theorem~\ref{thm:C*(G)} agrees
with Renault's definition. This is because, to deal with non-\'etale groupoids, Renault
defines the universal norm, not as the supremum over all $^*$-representations of
$C_c(\mathcal{G})$, but as the supremum only over $^*$-representations of
$C_c(\mathcal{G})$ that are bounded with respect to the ``$I$-norm'' on
$C_c(\mathcal{G})$. When $\mathcal{G}$ is \'etale, the $I$-norm is given by
\[
\|f\|_I = \sup_{x \in \mathcal{G}^{(0)}} \max\Big\{ \sum_{\gamma \in \mathcal{G}_x} |f(\gamma)|, \sum_{\gamma \in \mathcal{G}^x} |f(\gamma)|\Big\}.
\]
(Think of it like a fibrewise $1$-norm.) Renault shows that boundedness in the $I$-norm
is equivalent to continuity in the inductive-limit topology on $C_c(\mathcal{G})$: the
topology obtained by regarding $C_c(\mathcal{G})$ as the inductive limit of the subspaces
$X_K \coloneq \big(\{f \in C(\mathcal{G}) \mid \operatorname{supp}(f) \subseteq K\},
\|\cdot\|_\infty\big)$ indexed by compact subsets $K$ of $\mathcal{G}$ (see, for example,
\cite[Definition~5.4 and Example~5.10]{Conway}). In the general non-\'etale setting, this
equivalence between $I$-norm boundedness and continuity in the inductive-limit topology
is nontrivial, and requires an appeal to Renault's Disintegration Theorem, which we will
discuss later.

So to make sure we are talking about the same universal $C^*$-algebra as Renault, we must
verify that every $^*$-representation of $C_c(\mathcal{G})$ is continuous in the
inductive-limit topology when $\mathcal{G}$ is \'etale; and for completeness we should
also prove that continuity in this topology implies boundedness with respect to the
$I$-norm. The nice proof of this latter fact given below was developed by Ben Maldon in
his honours thesis; it does not seem to have appeared in the literature previously.

\begin{lemma}
Suppose that $\mathcal{G}$ is a second-countable locally compact Hausdorff \'etale
groupoid. Then every $^*$-representation $\pi$ of $C_c(\mathcal{G})$ is both continuous
in the inductive-limit topology, and bounded in the $I$-norm.
\end{lemma}
\begin{proof}
Fix a nondegenerate $^*$-representation of $C_c(\mathcal{G})$. By, for example,
\cite[Proposition~5.7]{Conway}, to see that $\pi$ is continuous in the inductive-limit
topology, we just have to check that $\pi|_{X_K}$ is continuous for each compact $K
\subseteq \mathcal{G}$. To see this, fix a compact $K \subseteq \mathcal{G}$. We can
cover $K$ by open bisections, and then use compactness to obtain a finite subcover $K
\subseteq \bigcup^n_{i=1} U_i$. Fix a partition of unity $\{h_i\}$ for $K$ subordinate to
the $U_i$. Then for $f \in X_k$, Lemma~\ref{prp:rep bound} gives
\[
\|\pi(f)\|
    = \Big\| \sum_i \pi(h_i \cdot f)\Big\|
    \le \sum_{i=1}^n \|\pi(h_i \cdot f)\|
    \le \sum_{i=1}^n \|h_i \cdot f\|_\infty
    \le n \|f\|_\infty.
\]
So $\pi$ is Lipschitz on $X_K$ with Lipschitz constant at most $n$. This shows that $\pi$
is continuous in the inductive-limit topology.

To see that it is $I$-norm bounded, observe that if $f \in C_c(\mathcal{G})$, then
$\|f\|_\infty \le \|f\|_I$. So the inductive-limit topology is weaker than the $I$-norm
topology, and we deduce that $\pi$ is continuous for the $I$-norm. Since continuity is
equivalent to boundedness for linear maps on normed spaces, we deduce that $\pi$ is
$I$-norm bounded. Routine calculations show that the $I$-norm is a $^*$-algebra norm, so
$\pi$ extends to a $^*$-homomorphism from the Banach $^*$-algebra completion
$\overline{C_c(\mathcal{G})}^I$ of $C_c(\mathcal{G})$ in the $I$-norm into
$\mathcal{B}(\mathcal{H})$. Now we can apply spectral theory: Write $\rho_A : A \to
[0,\infty)$ for the spectral-radius function on a Banach algebra $A$. For each $f \in
C_c(\mathcal{G})$, we have $\|\pi(f)\|^2 = \|\pi(f^*f)\| =
\rho_{\mathcal{B}(\mathcal{H})}(\pi(f^*f)) \le \rho_{\overline{C_c(\mathcal{G})}^I}(f^*f)
\le \|f^*f\|_I \le \|f\|_I^2$.
\end{proof}

We will not really need the preceding result from here on in, but we can take comfort
that we are discussing the same family of representations as Renault is; so we can appeal
to his theorems at need.

\begin{example}
If $\mathcal{G}$ is a group, then $C^*(\mathcal{G})$ is the usual full group
$C^*$-algebra.
\end{example}

\begin{example}
Let $X$ be a compact Hausdorff space, and $\Gamma$ a discrete group acting on $X$. Then
the elements $U_g \coloneq 1_{\{g\} \times X}$ indexed by $g \in \Gamma$ belong to
$C_c(\mathcal{G})$, and there is an inclusion $\pi : C(\mathcal{G}^{(0)}) \to
C_c(\mathcal{G})$ such that $\pi(f)(g, x) = \delta_{e, g} f(x)$. For $f \in C(X)$ and $g
\in \Gamma$ we have $U_g \pi(f) U^*_g (h, x) = \sum_{\alpha\beta\gamma = (h, x)} 1_{\{g\}
\times X}(\alpha) \pi(f)(\beta) 1_{\{g\} \times X}^*(\gamma)$. This can only be nonzero
if $h = 0$, and then the only nonzero term occurs when $\alpha = (g, g^{-1} \cdot x)$ and
$\gamma = (g^{-1}, x)$. For this $\alpha,\beta$ we have $1_{\{g\} \times X}(\alpha)
\pi(f)(\beta) 1_{\{g\} \times X}^*(\gamma) = \pi(f)(e, \gamma^{-1} \cdot x) =
f(\gamma^{-1}\cdot x)$. So $U_g\pi(f) U^*_g \in \pi(C(X))$, and agrees with $\pi(x
\mapsto f(g^{-1} \cdot x))$. So the universal property of the crossed product $C(X)
\rtimes \Gamma$ (see \cite[Theorem~1.3.3]{WilliamsNotes} and \cite[Theorem~2.6.1]{TFB^2})
gives a homomorphism $C(X) \rtimes \Gamma \to C^*(\mathcal{G})$ that takes each
$i_\Gamma(g)$ to $U_g$ and each $i_{C(X)}(f)$ to $\pi(f)$. Conversely, for $f \in
C_c(\mathcal{G})$ and $g \in \mathcal{G}$, define $f_g \in C(X)$ by $f_g(x) = f(g,x)$.
Easy calculations show that the map $\psi(f) \coloneq \sum_{g \in \Gamma} i_{C(X)}(f_g)
i_\Gamma(g)$ gives a $^*$-homomorphism $\psi : C_c(\mathcal{G}) \to C(X) \rtimes \Gamma$,
and that $\psi$ is inverse to $\pi \times U$. So $\psi$ is an isomorphism
$C^*(\mathcal{G}) \cong C(X) \times \Gamma$.
\end{example}

\begin{example}
Consider the groupoid $\mathcal{R}_{2^\infty}$ of Example~\ref{ex:2infty relation}. For
each $n \ge 0$ and each pair $u,v \in \{0,1\}^n$, let $\theta_n(u,v) \coloneq 1_{Z(u,v)}
\in C_c(\mathcal{R}_{2^\infty})$. Easy calculations using Lemma~\ref{lem:bisection conv}
show that $\theta_n(u,v)^* = \theta_n(v,u)$ and $\theta_n(u,v)\theta_n(w,y) =
\delta_{v,w} \theta_n(u,y)$. So $A_n \coloneq \operatorname{span}\{\theta_n(u,v) \mid u,v
\in \{0,1\}^n\}$ is isomorphic to $M_{\{0,1\}^n}(\mathbb{C})$ via the map $\sum_{u,v}
a_{u,v} \theta_n(u,v) \mapsto (a_{u,v})_{u,v \in \{0,1\}^n}$. Since $Z(u,v) = Z(u0, v0)
\sqcup Z(u1,v1)$, we have $\theta_n(u,v) = \theta_{n+1}(u0,v0) + \theta_{n+1}(u1, v1)$,
and so $A_n \subseteq A_{n+1}$. If we identify $M_{\{0,1\}^{n+1}}(\mathbb{C})$ with
$M_2(M_{\{0,1\}^n}(\mathbb{C}))$ via
\[
(a_{u,v})_{u,v \in \{0,1\}^{n+1}} \mapsto
    \big( (a_{wi, yj})_{w,y \in \{0,1\}^n}\big)_{i,j \in \{0,1\}},
\]
then the inclusion $A_n \to A_{n+1}$ is compatible with the canonical block-diagonal
inclusion $M_{\{0,1\}^n}(\mathbb{C}) \to M_{\{0,1\}^{n+1}}(\mathbb{C})$ with
multiplicity~2. So the uniqueness of $M_{2^\infty}$ shows that $\overline{\bigcup_n A_n}
\cong M_{2^\infty}$. A straightforward application of the Stone-Weierstrass theorem shows
that $C(Z(u,v)) \subseteq \overline{\bigcup_n A_n}$ for all $|u| = |v|$, and we deduce
that that $C_c(R_{2^\infty}) \subseteq \overline{\bigcup_n A_n}$. Hence
$C^*(R_{2^\infty}) \cong M_{2^\infty}(\mathbb{C})$.
\end{example}

\begin{example}
If $\mathcal{G}$ is the graph groupoid of a directed graph $E$ as in
Example~\ref{ex:graph gpd}, then the characteristic functions $\{1_{Z(v,v)} \mid v \in
E^0\}$ and $\{1_{Z(e, s(e))} \mid e \in E^1\}$ constitute a Cuntz--Krieger family for $E$
that generates $C^*(\mathcal{G})$. The homomorphism $C^*(E) \to C^*(\mathcal{G})$ induced
by this family is an isomorphism.
\end{example}

\begin{remark}[Unitary representations]
To make sense of a unitary representation of $\mathcal{G}$, we proceed, very roughly, as
follows. A \emph{unitary representation} of $\mathcal{G}$ is a triple $(\mathcal{H}, \mu,
U)$ where $\mu$ is a Borel measure on $\mathcal{G}^{(0)}$, $\mathcal{H} = \bigsqcup_{x
\in \mathcal{G}^{(0)}} \mathcal{H}_x$ is a $\mu$-measurable field of Hilbert spaces, and
$U = \{U_\gamma \mid \gamma \in \mathcal{G}\}$ is a family of unitary operators $U_\gamma
: \mathcal{H}_{s(\gamma)} \to \mathcal{H}_{r(\gamma)}$ satisfying $U_\alpha U_\beta =
U_{\alpha\beta}$, $U_{\alpha^{-1}} = U_\alpha^*$, and $\gamma \mapsto \big(U_\gamma
\xi(s(\gamma)) \mid \eta(r(\gamma))\big)$ is measurable for each pair $\xi,\eta$ of
measurable sections of $\mathcal{H}$. Every unitary representation $(\mathcal{H}, \mu,
U)$ of $\mathcal{G}$ induces a $^*$-representation $\pi_{(\mathcal{H},\mu,U)}$ of
$C_c(\mathcal{G})$ on the direct integral $\int^{\oplus}_{\mathcal{G}^{(0)}}
\mathcal{H}_x\,d\mu(x)$ characterised by
\[
\big(\pi_{(\mathcal{H}, \mu, U)}(f) \xi \mid \eta\big)
    = \int_{\mathcal{G}^{(0)}} \sum_{\gamma \in \mathcal{G}_x} \big(f(\gamma)U_\gamma \xi(x) \mid \xi(r(\gamma))\big)\,d\mu(x).
\]
This representation is called the \emph{integrated form} of $(\mathcal{H}, \mu, U)$.
Renault's \emph{Disintegration Theorem} \cite[Theorem~II.1.21]{Ren80} says that every
representation is unitarily equivalent to the integrated form of a unitary
representation.
\end{remark}

\section{The reduced \texorpdfstring{$C^*$}{C*}-algebra}

To show that the map $\pi_{\max} : C_c(\mathcal{G}) \to C^*(\mathcal{G})$ of
Theorem~\ref{thm:C*(G)} is injective, we need to construct a representation that is
injective on $C_c(\mathcal{G})$.

If $\mathcal{G}$ were a group, we would use the regular representation of $\mathcal{G}$.
We aim to do the same thing for groupoids, but there are multiple regular representations
to consider.

\begin{proposition}\label{prp:reg reps}
Let $\mathcal{G}$ be a second-countable locally compact Hausdorff \'etale groupoid. For
each $x \in \mathcal{G}^{(0)}$, there is a $^*$-representation $\pi_x : C_c(\mathcal{G})
\to \mathcal{B}(\ell^2(\mathcal{G}_x))$ such that
\[
\pi_x(f)\delta_\gamma = \sum_{\alpha \in \mathcal{G}_{r(\gamma)}} f(\alpha)\delta_{\alpha\gamma}.
\]
This $\pi_x$ is called the \emph{regular representation} of $C_c(\mathcal{G})$ associated
to $x$.

For each $\eta \in \mathcal{G}$ the map $U_\eta : \ell^2(\mathcal{G}_{s(\eta)}) \to
\ell^2(\mathcal{G}_{r(\eta)})$ given by $U_\eta \delta_\gamma = \delta_{\gamma\eta^{-1}}$
is a unitary operator and we have $\pi_{r(\eta)} = U_\eta \pi_{s(\eta)} U^*_\eta$.
\end{proposition}
\begin{proof}
The first assertion is relatively straightforward to check using that the formula given
for $\pi_x(f)\delta_\gamma$ is really just the formula for $f * \delta_\gamma$ if the
convolution product is extended to not-necessarily-continuous functions.

The operators $U_\eta$ are certainly unitary, and for $f \in C_c(\mathcal{G})$,
\begin{align*}
U_\eta \pi_{s(\eta)}(f) U^*_\eta \delta_\gamma
    &= U_\eta \pi_{s(\eta)}(f) \delta_{\gamma\eta}\\
    &= \sum_{\alpha \in \mathcal{G}_{r(\gamma\eta)}} f(\alpha) U_\eta \delta_{\alpha\gamma\eta}
    = \sum_{\alpha \in \mathcal{G}_{r(\gamma\eta)}} f(\alpha) \delta_{\alpha\gamma}
    = \pi_{r(\eta)}(f) \delta_\gamma.\qedhere
\end{align*}
\end{proof}

\begin{definition}
Let $\mathcal{G}$ be a second-countable locally compact Hausdorff \'etale groupoid. The
\emph{reduced $C^*$-algebra}, denoted $C^*_r(\mathcal{G})$, of $\mathcal{G}$ is the
completion of
\[
\Big(\bigoplus_{x \in \mathcal{G}^{(0)}} \pi_x\Big)(C_c(\mathcal{G})) \subseteq \bigoplus_{x \in \mathcal{G}^{(0)}} \mathcal{B}(\ell^2(\mathcal{G}_x)).
\]
We write $\|\cdot\|_r$ for the $C^*$-norm on $C^*_r(\mathcal{G})$.
\end{definition}

The universal property of $C^*(\mathcal{G})$ yields a homomorphism, which we denote
$\pi_r$, from $C^*(\mathcal{G})$ to $C^*_r(\mathcal{G})$ such that $\pi_r \circ
\pi_{\max} = \bigoplus_x \pi_x$. In particular, $\|\cdot\|_r \le \|\cdot\|$ on
$C_c(\mathcal{G})$.

\begin{proposition}\label{prp:j map}
Let $\mathcal{G}$ be a second-countable locally compact Hausdorff \'etale groupoid. There
is an injective, norm-decreasing map $j : C^*_r(\mathcal{G}) \to C_0(\mathcal{G})$ such
that
\[
    j(a)(\gamma) = \big(\pi_{s(\gamma)}(a) \delta_{s(\gamma)} \mid \delta_\gamma\big)
\]
for all $a \in C^*_r(\mathcal{G})$ and $\gamma \in \mathcal{G}$. For $f \in
C_c(\mathcal{G})$, we have $j(f) = f$.
\end{proposition}
\begin{proof}
For $f \in C_c(\mathcal{G})$, we have
\[
j(f)(\gamma)
    = \Big(\sum_{\alpha \in \mathcal{G}_{r(\gamma)}} f(\alpha) \delta_{\alpha s(\gamma)} \mid \delta_\gamma\Big)
    = f(\gamma).
\]
The Cauchy--Schwarz inequality shows that $j$ is norm-decreasing from $\|\cdot\|_r$ to
$\|\cdot\|_\infty$. An $\varepsilon/3$-argument shows that each $j(a)$ belongs to
$C_0(\mathcal{G})$. If $a \in C^*_r(\mathcal{G})$ is nonzero, then there exists $x \in
\mathcal{G}^{(0)}$ such that $\pi_x(a) \not= 0$, and then there exist $\alpha,\beta \in
\mathcal{G}_x$ such that $\big(\pi_x(a) \delta_\alpha \mid \delta_\beta\big) \not= 0$.
Now Proposition~\ref{prp:reg reps} gives
\[
j(a)(\beta\alpha^{-1})
    = \big(\pi_{r(\alpha)}(a) \delta_{r(\alpha)} \mid \delta_{\beta\alpha^{-1}}\big)
    = \big(U^*_\alpha \pi_{r(\alpha)}(a) U_\alpha \delta_{\alpha} \mid \delta_{\beta}\big)
    = \big(\pi_{s(\alpha)} \delta_\alpha \mid \delta_\beta\big) \not= 0.
\]
Hence $j$ is injective.
\end{proof}

\begin{corollary}\label{cor:norm relation}
Suppose that $\mathcal{G}$ is a second-countable locally compact Hausdorff \'etale
groupoid. Then the homomorphisms $\pi_{\max} : C_c(\mathcal{G}) \to C^*(\mathcal{G})$ and
$\bigoplus_x \pi_x : C_c(\mathcal{G}) \to C^*_r(\mathcal{G})$ are both injective. For $f
\in C_c(\mathcal{G})$, we have $\|f\|_\infty \le \|f\|_r \le \|f\|$. If $f$ is supported
on a bisection, then we have equality throughout.
\end{corollary}
\begin{proof}
That $\|f\|_r \le \|f\|$ is by definition of the universal norm. That $\|f\|_\infty \le
\|f\|_r$ follows from Proposition~\ref{prp:j map}. Since $j(f) = f$ for all $f \in
C_c(\mathcal{G})$, it follows that $\pi_{\max}$ and $\bigoplus_x \pi_x$ are injective on
$C_c(\mathcal{G})$. If $f$ is supported on a bisection, then we have $\|f\|^2 = \|f^*
* f\|$. Since $f^* * f \in C_c(\mathcal{G}^{(0)})$, the uniqueness of the $^*$-algebra norm on
$C_c(\mathcal{G}^{(0)})$ gives $\|f\|^2 = \|f^* * f\|_\infty$, and this is precisely
$\|f\|_\infty^2$ by Lemma~\ref{lem:bisection conv}.
\end{proof}

Using the above Corollary, we see that we can apply the Stone--Weierstrass theorem to
establish surjectivity of a homomorphism into either $C^*(\mathcal{G})$ or
$C^*_r(\mathcal{G})$.

\begin{corollary}
Suppose that $\mathcal{G}$ is a second-countable locally compact Hausdorff \'etale
groupoid. Let $A$ be a $C^*$-algebra, and suppose that $\pi : A \to C^*(G)$ (or $\pi : A
\to C^*_r(\mathcal{G})$) is a homomorphism. Suppose that for each open bisection $U
\subseteq G$ and each pair of distinct points $\beta,\gamma \in U$, there exists $a \in
A$ such that $\pi(a) \in C_0(U)$, $\pi(a)(\beta) = 0$ and $\pi(a)(\gamma) = 1$. Then
$\pi$ is surjective.
\end{corollary}
\begin{proof}
A straightforward application of the Stone--Weierstrass theorem shows that $\pi(A)$
contains $C_0(\mathcal{G}^{(0)})$. Using this and the convolution formula in
Lemma~\ref{lem:bisection conv}, we see that for $U \subseteq \mathcal{G}$ an open
bisection, the set $\pi(A) \cap C_0(U)$ is closed under pointwise multiplication
(identify $C_0(U)$ with $C_0(r(U))$, and then note that if $f,g \in \pi(A) \cap C_0(U)$,
then $f \circ r^{-1} \in C_0(r(U)) \subseteq C_0(\mathcal{G}) \subseteq \pi(A)$, and $f
\cdot g = (f \circ r^{-1}) * g$. So another application of the Stone--Weierstrass theorem
combined with the fact that $\|\cdot\|_{C^*(\mathcal{G})}$ agrees with $\|\cdot\|_\infty$
on $C_c(U)$ shows that $C_c(U) \subseteq \pi(A)$. Now Lemma~\ref{lem:bisection span}
shows that $C_c(\mathcal{G}) \subseteq \pi(A)$. Since $\pi$ is a $C^*$-homomorphism, it
has closed range, and we deduce that $C^*(\mathcal{G}) \subseteq \pi(A)$.
\end{proof}

\begin{example}
If $\mathcal{G}$ is a group, then there is just one unit $e$, and the regular
representation $\pi_e$ is the usual left-regular representation of the group algebra. So
$C^*_r(\mathcal{G})$ is the usual reduced $C^*$-algebra. In particular,
$C^*_r(\mathcal{G}) = C^*(\mathcal{G})$ if and only if $\mathcal{G}$ is amenable.
\end{example}

\begin{example}
Let $R$ be a discrete equivalence relation. For distinct equivalence classes $S, T
\subseteq R^{(0)}$ under $R$, we have $C_c(S \times S) \perp C_c(T \times T)$ inside
$C_c(R)$, and so $C_c(R)$ is the direct sum of the $^*$-subalgebras $C_c(S)$. It follows
that the completions of these $^*$-subalgebras are direct summands in each of $C^*(R)$
and $C^*_r(R)$. For a fixed $S$, the elements $\{1_{(s, s')} \mid s,s' \in S\}$ form a
complete set of nonzero matrix units indexed by $S$, and generate both $C^*(S \times S)$
and $C^*_r(S \times S)$ as $C^*$-algebras. Since $\mathcal{K}(\ell^2(S))$ is the unique
$C^*$-algebra generated by a family of nonzero matrix units indexed by $S$, we deduce
that $C^*_r(S \times S) \cong C^*(S \times S) \cong \mathcal{K}(\ell^2(S))$.
\end{example}

\begin{example}
If $\mathcal{G}$ is the transformation groupoid for an action of a discrete group
$\Gamma$ on a compact Hausdorff space $X$, then the unit space $\mathcal{G}^{(0)}$ is
just $X$. Fix $x \in X = \mathcal{G}^{(0)}$. The isomorphism $\omega : C_c(\mathcal{G})
\to C_c(\Gamma, C(X))$ of Example~\ref{eg:trans gpd conv} intertwines the regular
representation $\pi_x$ of $C_c(\mathcal{G})$ with the induced representation of
$C_c(\Gamma, C(X))$ associated the character of $C(X)$ given by evaluation at $x$,
denoted $\operatorname{Ind} \operatorname{ev}_x$ in \cite[Example~2.4.3]{WilliamsNotes}.
So as discussed in that example, since $\bigoplus_x \operatorname{ev}_x$ is a faithful
representation of $C_0(X)$, we see that $C^*_r(\mathcal{G})$ is isomorphic to the reduced
crossed product $C(X) \times_{\alpha, r} \Gamma$.
\end{example}

\begin{remark}
There is an alternative, slightly slicker, approach to defining $C^*_r(\mathcal{G})$.
Define $\langle \cdot, \cdot \rangle_{C_0(\mathcal{G}^{(0)})} : C_c(\mathcal{G}) \times
C_c(\mathcal{G}) \to C_0(\mathcal{G}^{(0)})$ by $\langle f,
g\rangle_{C_0(\mathcal{G}^{(0)})} \coloneq (f^* * g)|_{\mathcal{G}^{(0)}}$. It is
straightforward to check that this is a $C_0(\mathcal{G}^{(0)})$-valued inner-product (in
particular, positive definite) on $C_c(\mathcal{G})$, so we can form the corresponding
Hilbert-module completion $X_\mathcal{G}$. Now the action of $C_c(\mathcal{G})$ on itself
by left multiplication extends to a homomorphism $L : C_c(\mathcal{G}) \to
\mathcal{L}(X_\mathcal{G})$, the $C^*$-algebra of adjointable operators on
$X_\mathcal{G}$. It is a fairly straightforward exercise, if you are familiar with
Hilbert modules, to verify that $\|L(f)\| = \|f\|_r$ for all $f \in C_c(\mathcal{G})$. So
the completion of $L(C_c(\mathcal{G}))$ in $\mathcal{L}(X_\mathcal{G})$ is isomorphic to
$C^*_r(\mathcal{G})$; indeed, we could have taken this as the definition of
$C^*_r(\mathcal{G})$.
\end{remark}

\section{Equivalence of groupoids}

Renault's notion of equivalence of groupoids closely reflects Morita equivalence for
$C^*$-algebras.

\begin{definition}
Suppose that $\mathcal{G}$ is an \'etale groupoid. A \emph{left $\mathcal{G}$-space} is a
locally compact Hausdorff space $X$ endowed with a continuous map $r : X \to
\mathcal{G}^{(0)}$ and a continuous map $\cdot : \mathcal{G} * X \coloneq \{(\gamma, \xi)
\in \mathcal{G} \times X \mid s(\gamma) = r(\xi)\} \to X$ such that
\begin{enumerate}
\item $r(\gamma \cdot \xi) = r(\gamma)$ for all $(\gamma,\xi) \in \mathcal{G} * X$;
\item $r(\xi)\cdot \xi = \xi$ for all $\xi \in X$; and
\item $\alpha \cdot (\beta \cdot \xi) = (\alpha\beta)\cdot \xi$ for all $(\beta,\xi)
    \in \mathcal{G} * X$ and $\alpha \in \mathcal{G}_{r(\beta)}$.
\end{enumerate}
A \emph{right $\mathcal{G}$-space} is defined similarly, but with a map $s : X \to
\mathcal{G}^{(0)}$ and the roles of $s$ and $r$ reversed in (1)--(3).

We say that $X$ is a \emph{proper} left $\mathcal{G}$-space if the map $(\gamma, \xi)
\mapsto (\gamma\cdot \xi, \xi)$ from $\mathcal{G} * X \to X \times X$ is a proper map. It
is \emph{free} if $\gamma \cdot \xi = \xi$ implies $\gamma = r(\xi)$.

If $\mathcal{G}$ and $\mathcal{H}$ are two \'etale groupoids, then a
$\mathcal{G}$--$\mathcal{H}$-equivalence is a locally compact Hausdorff space $X$ that is
simultaneously a free and proper left $\mathcal{G}$-space and a free and proper right
$\mathcal{H}$ space such that the left and right actions commute, and such that $r : X
\to \mathcal{G}^{(0)}$ and $s : X \to \mathcal{H}^{(0)}$ are open maps and induce
homeomorphisms $\tilde{r} : X/\mathcal{H} \cong \mathcal{G}^{(0)}$, and $\tilde{s} :
\mathcal{G}\backslash X \to \mathcal{H}^{(0)}$.
\end{definition}

\begin{lemma}
Suppose that $\mathcal{G}, \mathcal{H}$ are second-countable locally compact Hausdorff
\'etale groupoids. Let $X$ be a $\mathcal{G}$--$\mathcal{H}$-equivalence. If $\xi,\eta
\in X$ satisfy $r(\xi) = r(\eta)$, then there is a unique element $[\xi,
\eta]_\mathcal{H} \in \mathcal{H}$ such that $\xi \cdot [\xi,\eta]_\mathcal{H} = \eta$.
Likewise, if $s(\xi) = s(\eta)$, then there is a unique ${_\mathcal{G}[\xi,\eta]} \in
\mathcal{G}$ such that ${_\mathcal{G}[\xi,\eta]} \cdot \eta = \xi$.
\end{lemma}
\begin{proof}
It suffices to prove the first statement; the second is symmetric. That $r$ descends to a
homeomorphism $X/\mathcal{H} \to \mathcal{G}^{(0)}$ shows that there exists an element
$\lambda \in \mathcal{H}$ such that $\xi \cdot \lambda = \eta$. Freeness shows that this
$\lambda$ is unique.
\end{proof}

Given a $\mathcal{G}$--$\mathcal{H}$-equivalence $X$, there is a corresponding
$\mathcal{H}$--$\mathcal{G}$-equivalence $X^* \coloneq \{\xi^* \mid \xi \in X\}$ defined
by $r(\xi^*) = s(\xi)$, $s(\xi^*) = r(\xi)$, $\lambda \cdot \xi^* \coloneq (\xi \cdot
\lambda^{-1})^*$ and $\xi^* \cdot \gamma = (\gamma^{-1}\cdot x)^*$. Clearly $X^{**} \cong
X$ via the map $\xi^{**} \mapsto \xi$.

\begin{proposition}\label{prp:linking gpd}
Suppose that $\mathcal{G}$ and $\mathcal{H}$ are second-countable locally compact
Hausdorff \'etale groupoids. Let $X$ be a $\mathcal{G}$--$\mathcal{H}$-equivalence. Let
$L \coloneq \mathcal{G} \sqcup X \sqcup X^* \sqcup \mathcal{H}$, with the relative
topology. Then $L$ is an \'etale groupoid with
\begin{itemize}
\item $L^{(0)} = \mathcal{G}^{(0)} \sqcup \mathcal{H}^{(0)}$,
\item range and source maps inherited from those on $\mathcal{G}, X, X^*,
    \mathcal{H}$,
\item multiplication inherited from multiplication in $\mathcal{G}$ and
    $\mathcal{H}$, the actions of $\mathcal{G}$ and $\mathcal{H}$ on $X$ and $X^*$
    and with $\xi^* \eta \coloneq [\xi, \eta]_\mathcal{H}$ for $\xi,\eta \in X$ with
    $r(\xi) = r(\eta)$, and $\xi \eta^* \coloneq {_\mathcal{G}[\xi,\eta]}$ for
    $\xi,\eta \in X$ with $s(\xi) = s(\eta)$, and
\item $\xi^{-1} = \xi^*$ and $(\xi^*)^{-1} = \xi$ for $\xi \in X$.
\end{itemize}
We have $\mathcal{G}^{(0)} L \mathcal{G}^{(0)} = \mathcal{G}$ and $\mathcal{H}^{(0)} L
\mathcal{H}^{(0)} = \mathcal{H}$.
\end{proposition}
\begin{proof}
The proof that $L$ is a topological groupoid is routine but tedious. To see that it is
\'etale, we show that $r : X \to \mathcal{G}^{(0)}$ is a local homeomorphism; that $s$ is
a local homeomorphism follows from a similar argument.

We already know that $r$ is an open map, so we need only show that it is locally
injective. So fix $\xi \in X$ and sequences $\xi_n, \xi'_n \to \xi$ such that $r(\xi_n) =
r(\xi'_n)$ for all $n$. We must show that $\xi_n = \xi'_n$ for large $n$. We have
$[\xi_n, \xi'_n]_\mathcal{H} \to [\xi, \xi]_\mathcal{H} = s(\xi)$. Since $s(\xi) \in
\mathcal{H}^{(0)}$ and since $\mathcal{H}^{(0)}$ is open, we therefore have $[\xi_n,
\xi'_n]_\mathcal{H} \in \mathcal{H}^{(0)}$ for large $n$; that is $[\xi_n,
\xi'_n]_\mathcal{H} = s(\xi_n)$ for large $n$, and therefore $\xi'_n = \xi_n \cdot
s(\xi_n) = \xi_n$ for large $n$.
\end{proof}

We call the groupoid of Proposition~\ref{prp:linking gpd} the \emph{linking groupoid} of
$X$.

\begin{theorem}
Suppose that $\mathcal{G}, \mathcal{H}$ are second-countable locally compact Hausdorff
\'etale groupoids. Let $X$ be a $\mathcal{G}$--$\mathcal{H}$-equivalence. Let $L$ be the
linking groupoid of $X$. Then $P \coloneq 1_{\mathcal{G}^{(0)}}$ and $Q \coloneq
1_{\mathcal{H}^{(0)}}$ belong to $\mathcal{M}(C^*(L))$ and to $\mathcal{M}(C^*_r(L))$,
and are complementary full projections. We have $P C^*(L) P \cong C^*(\mathcal{G})$ and
$Q C^*(L) Q \cong C^*(\mathcal{H})$, and similarly $P C^*_r(L) P \cong
C^*_r(\mathcal{G})$ and $Q C^*_r(L) Q \cong C^*_r(\mathcal{H})$. In particular,
$C^*(\mathcal{G})$ and $C^*(\mathcal{H})$ are Morita equivalent via the imprimitivity
bimodule $P C^*(L) Q$ and likewise $C^*_r(\mathcal{G})$ and $C^*_r(\mathcal{H})$ are
Morita equivalent via the imprimitivity bimodule $P C^*_r(L) Q$.
\end{theorem}
\begin{proof}
Fix an increasing sequence $K_n$ of compact subsets of $\mathcal{G}^{(0)}$ with
$\bigcup_n K_n = \mathcal{G}^{(0)}$ and choose functions $e_n \le 1$ in
$C_0(\mathcal{G}^{(0)})$ with $e_n \equiv 1$ on $K_n$. For $f \in C_c(L)$, we have $e_n f
= f|_{\mathcal{G}^{(0)} L}$ for large $n$ (just take $n$ large enough so that
$r(\operatorname{supp}(f)) \subseteq K_n$). So the $e_n$ converge strictly to a
multiplier projection $P$ with the property that $Pf = 1_{\mathcal{G}^{(0)}} * f$ for $f
\in C_c(\mathcal{G})$. A similar argument gives $Q$, and it is clear that $P + Q = 1$.
The map $\pi_r$ is clearly nondegenerate, and its extension to
$\mathcal{M}(C^*(\mathcal{G}))$ takes $P$ and $Q$ to projections in
$\mathcal{M}(C^*_r(\mathcal{G}))$ with the same properties, and which we continue to call
$P$ and $Q$.

To see that $P C^*_r(L) P \cong C^*_r(\mathcal{G})$, first note that since $\mathcal{G}$
is an open subgroupoid of $L$, there is an inclusion of $C_c(\mathcal{G})$ in $C_c(L)$
that extends a function $f \in C_c(\mathcal{G})$ to $L$ by defining $f(\eta) = 0$ for
$\eta \not\in L$. We just have to show that this inclusion is isometric for the reduced
norms. For a fixed $x \in \mathcal{G}^{(0)}$, consider the regular representation
$\pi^L_x$ of $C_c(L)$ on $\ell^2(L_x)$. Let $R \in \mathcal{B}(\ell^2(L_x))$ be the
orthogonal projection onto $\overline{\operatorname{span}}\{\delta_\eta \mid r(\eta) \in
\mathcal{G}^{(0)}\}$. We have $R = \tilde{\pi}^L_x(P)$, and since $P f = fP = f$ for $f
\in C_c(\mathcal{G})$, we see that $\pi^L_x(f) = R \pi^L_x(f) R$ for $f \in
C_c(\mathcal{G})$. Since $x \in \mathcal{G}^{(0)}$, the set $\{\eta \in L_x \mid r(\eta)
\in \mathcal{G}^{(0)}\}$ is precisely $\mathcal{G}_x$. Using this, it is easy to see that
$R \pi^L_x|_{C_c(\mathcal{G})} R$ is a copy of the regular representation
$\pi^\mathcal{G}_x$ of $C_c(\mathcal{G})$. So $\|\pi^\mathcal{G}_x(f)\| = \|R \pi^L_x(f)
R\| \le \|\pi^L_x(f)\|$ for all $f \in C_c(\mathcal{G})$. This immediately shows that for
$f \in C_c(\mathcal{G})$ we have $\|f\|_{C^*_r(\mathcal{G})} \le \|f\|_{C^*_r(L)}$. For
the reverse inequality, it suffices to show that if $y \in L^{(0)} \setminus
\mathcal{G}^{(0)}$ then there exists $x \in \mathcal{G}^{(0)}$ such that $\|\pi^L_y(f)\|
= \|\pi^L_x(f)\|$ for all $f \in C_c(\mathcal{G})$. For this, first note that since $s :
X \to \mathcal{H}^{(0)}$ induces a homeomorphism of $\mathcal{G}\backslash X$ onto
$\mathcal{H}^{(0)}$, it is surjective, so we can fix $\eta \in X$ with $s(\eta) = y$. By
Proposition~\ref{prp:reg reps}, the representation $\pi^L_y$ is unitarily equivalent to
$\pi^L_{r(\eta)}$, and so $x \coloneq r(\eta) \in \mathcal{G}^{(0)}$ has the desired
property. An identical argument shows that $Q C^*(L) Q \cong C^*_r(\mathcal{H})$.

We now turn to universal $C^*$-algebras; again, it suffices to show that the inclusion
$C_c(\mathcal{G}) \hookrightarrow C_c(L)$ is isometric for the universal norm. Certainly
the inclusion $C_c(\mathcal{G}) \hookrightarrow C_c(L) \hookrightarrow C^*(L)$ determines
a $^*$-representation of $C_c(\mathcal{G})$, and so the universal property of
$C^*(\mathcal{G})$ shows that there is a homomorphism $C^*(\mathcal{G}) \to C^*(L)$ that
agrees with the canonical inclusion of $C_c(\mathcal{G})$. Thus the inclusion is
norm-decreasing. For the reverse inequality, the rough idea is to use that the pair
$\big(C_c(L \mathcal{G}^{(0)}), C_c(\mathcal{G}^{(0)} L)\big)$ is a $^*$-Morita context
in the sense of Ara \cite{Ara} between $C_c(L)$ and $C_c(\mathcal{G})$. Fix a
nondegenerate $^*$-representation $\pi$ of $C_c(\mathcal{G})$ in
$\mathcal{B}(\mathcal{H})$. It suffices to show that there is a $^*$-representation
$\tilde\pi$ of $C_c(L)$ such that
\begin{equation}\label{eq:norm preserved}
    \|\tilde\pi(f)\| \ge \|\pi(f)\|\quad\text{ for all $f \in C_c(\mathcal{G})$.}
\end{equation}
For this, define a positive semidefinite sesquilinear form on $C_c(L \mathcal{G}^{(0)})
\odot \mathcal{H}$ by $\big(f \odot h \mid g \odot k) = \big(h \mid \pi(f^*g) k)$. Let
$\widetilde{\mathcal{H}}$ denote the Hilbert space obtained by quotienting out the space
of vectors satisfying $\big(\xi \mid \xi\big) = 0$ and then completing in the norm coming
from the inner-product. For $f \in C_c(L \mathcal{G}^{(0)})$ and $h \in \mathcal{H}$, we
write $f \otimes_{C_c(\mathcal{G})} h$ for the image of $f \odot h$ in
$\widetilde{\mathcal{H}}$. It's not hard to check that $\tilde{\pi}(f)(g
\otimes_{C_c(\mathcal{G})} h) \coloneq (fg) \otimes_{C_c(\mathcal{G})} h$ defines a
$^*$-representation of $C_c(L)$ on $\widetilde{\mathcal{H}}$. To establish~\eqref{eq:norm
preserved}, fix $f \in C_c(\mathcal{G})$ and $\varepsilon > 0$, and choose $h \in
\mathcal{H}$ such that $\|h\| = 1$ and $\|\pi(f)h\| > \|\pi(f)\| - \varepsilon$. Fix an
increasing approximate identity $e_n$ for $C_0(\mathcal{G}^{(0)})$ in
$C_c(\mathcal{G}^{(0)})$ as in the first paragraph of the proof. Then $\pi(e_n) h \to h$
because $\pi$ is nondegenerate. It then follows from the definition of the inner product
on $\widetilde{\mathcal{H}}$ that $e_n \otimes_{C_c(\mathcal{G})} h$ is Cauchy in
$\widetilde{\mathcal{H}}$ and so converges to some $\tilde{h}$. Since
\[
\big(e_n \otimes_{C_c(\mathcal{G})} h \mid e_n \otimes_{C_c(\mathcal{G})} h\big) = \big(h \mid \pi(e_n^*e_n) h\big) \to 1,
\]
we have $\|\tilde{h}\| = 1$. Since $fe_n = f$ for large $n$, we have
\begin{align*}
\textstyle\|\tilde{\pi}(f)\tilde{h}\|^2
    &\textstyle= \lim_n \|\pi(fe_n) \otimes_{C_c(\mathcal{G})} h\|^2\\
    &\textstyle= \lim_n \big(h \mid \pi((fe_n e^*_n f^*))h\big)
    = \big(h \mid \pi(ff^*)h\big)
    = \|\pi(f)h\|^2
\end{align*}
So $\|\tilde{\pi}(f)\tilde{h}\| = \|\pi(f)h\| > \|\pi(f)\| - \varepsilon$, and since
$\|\tilde{h}\| = 1$, we deduce that $\|\tilde{\pi}(f)\| > \|\pi(f)\| - \varepsilon$.
Letting $\varepsilon \to 0$ gives~\eqref{eq:norm preserved}.
\end{proof}

\begin{remark}
Recall that, by Brown's theorem \cite{Brown}, if $A$ is a $\sigma$-unital $C^*$-algebra,
and $P$ is a multiplier projection of $A$, then $PAP \otimes \mathcal{K} \cong APA
\otimes \mathcal{K}$, and then the Brown--Green--Rieffel theorem \cite{BGR} says that
$\sigma$-unital $C^*$-algebras $A$ and $B$ are Morita equivalent if and only if they are
stably isomorphic. We now have a version of equivalence for groupoids, and we know that
the (discrete) full equivalence relation $R_\mathbb{N} \coloneq \mathbb{N} \times
\mathbb{N}$ has $C^*$-algebra isomorphic to $\mathcal{K}$. It is also not too difficult
to see, using universal properties, that if $\mathcal{G}$ and $\mathcal{H}$ are \'etale
groupoids, then $\mathcal{G} \times \mathcal{H}$ is a groupoid, and that $C^*(\mathcal{G}
\times \mathcal{H}) \cong C^*(\mathcal{G}) \otimes C^*(\mathcal{H})$. So it's not
unreasonable to ask whether Brown's theorem and the Brown--Green--Rieffel theorem could
carry over to groupoids. The answer is a qualified yes. Specifically, if $\mathcal{G}$
and $\mathcal{H}$ are \'etale groupoids such that $\mathcal{G}^{(0)}$ and
$\mathcal{H}^{(0)}$ are totally disconnected as topological spaces, then for every
compact open $K \subseteq \mathcal{G}^{(0)}$, we have $K \mathcal{G} K \times
R_\mathbb{N} \cong \mathcal{G} K \mathcal{G} \times R_\mathbb{N}$; and $\mathcal{G}$ and
$\mathcal{H}$ are equivalent if and only if $\mathcal{G} \times R_\mathbb{N} \cong
\mathcal{H} \times R_\mathbb{N}$; the proof follows, almost exactly, the proofs of
Brown's theorem and the Brown--Green--Rieffel theorem \cite{CRS}.
\end{remark}

\chapter{Fundamental structure theory}\label{ch:structure}

In this section, we will discuss the structural properties of $C^*(\mathcal{G})$. When is
it nuclear and when does it satisfy the UCT? When is it/isn't it simple, and more
generally what is its ideal structure? When is it purely infinite?

\section{Amenability, nuclearity and the UCT}

The theory of amenability for groupoids is complicated; it could easily be a five-hour
course all by itself. So we are going to skate over the top of it here. Most of what
appears here is taken from \cite{AR}.

Recall that a discrete group $\Gamma$ is \emph{amenable} if it admits a finitely additive
probability measure $\mu$ with the property that $\mu(g A) = \mu(A)$ for all $A \subseteq
\Gamma$ and $g \in \Gamma$. A discrete group $\Gamma$ is amenable if and only if
$C^*(\Gamma) = C^*_r(\Gamma)$.

Amenability for groupoids is intended as an analogue of amenability for groups, but
unfortunately, the analogies are not so well behaved as we might like. There are a number
of notions of amenability for groupoids, two of the most prominent being measurewise
amenability and topological amenability. These two coincide for \'etale groupoids, by
results of Anantharaman-Delaroche--Renault \cite{AR}, but as we shall see, they are not
equivalent to coincidence of the full and reduced $C^*$-algebras, even for group bundles.

We need some set-up. If $\mathcal{G}$ is an \'etale groupoid, then a \emph{continuous
system of probability measures for $\mathcal{G}$} is a system $\{\lambda^x \mid x \in
\mathcal{G}^{(0)}\}$ of Radon probability measures $\lambda^x$ on $\mathcal{G}$ with the
support of $\lambda^x$ contained in $\mathcal{G}^x$ such that for $f \in
C_c(\mathcal{G})$, the function $x \mapsto \int_\mathcal{G} f\,d\lambda^x$ is continuous.

\begin{remark}
Since each $\mathcal{G}^x$ is discrete, a Radon probability measure $\lambda^x$ on
$\mathcal{G}^x$ amounts to a function $\lambda^x : \mathcal{G}^x \to [0, \infty)$ with
$\sum_\gamma w^x(\gamma) = 1$.
\end{remark}

An \emph{approximate invariant continuous mean} for $\mathcal{G}$ is a net $\lambda_i$ of
continuous systems of probability measures for $\mathcal{G}$ such that the net $\big(M_i
: \gamma \mapsto \|\lambda_i^{r(\gamma)}(\gamma \cdot) -
\lambda_i^{s(\gamma)}(\cdot)\|_1\big)_i$ of functions from $\mathcal{G}$ to $[0,\infty)$
has the property that $M_i|_K \to 0$ uniformly for every compact $K \subseteq
\mathcal{G}$.

\begin{definition}
Let $\mathcal{G}$ be an \'etale groupoid. We say that $\mathcal{G}$ is
\emph{(topologically) amenable} if $\mathcal{G}$ admits an approximate invariant
continuous mean.
\end{definition}

There are a number of equivalent formulations of amenability, particularly in the setting
of \'etale groupoids. Perhaps one of the most useful is the following:

\begin{lemma}[{\cite[Proposition~2.2.13]{AR}}]
Let $\mathcal{G}$ be an \'etale groupoid. Then $\mathcal{G}$ is amenable if and only if
there is a sequence $(h_i)^\infty_{i=1}$ in $C_c(\mathcal{G})$ such that
\begin{enumerate}
\item the maps $x \mapsto \sum_{\gamma \in \mathcal{G}^x} |h_i(\gamma)|^2$ (indexed
    by $i$) converge uniformly to 1 on every compact subset of $\mathcal{G}^{(0)}$;
    and
\item the maps $\alpha \mapsto \sum_{\gamma \in \mathcal{G}^{r(\alpha)}}
    \big|h_i(\alpha^{-1}\gamma) - h_i(\gamma)\big|$ (indexed by $i$) converge
    uniformly to 0 on every compact subset of $\mathcal{G}$.
\end{enumerate}
\end{lemma}

The key point of amenability is the following:

\begin{theorem}[{\cite[Proposition~6.1.8]{AR}}]
Let $\mathcal{G}$ be an \'etale groupoid, and suppose that $\mathcal{G}$ is amenable.
Then $\pi_r : C^*(\mathcal{G}) \to C^*_r(\mathcal{G})$ is injective.
\end{theorem}

From \cite[Theorem~3.1.1]{WilliamsNotes}, we know that if $\Gamma$ is a discrete amenable
group acting on a locally compact Hausdorff space $X$ and $\mathcal{G}$ is the
transformation groupoid, then $C^*(\mathcal{G}) \cong C_0(X) \rtimes \Gamma$ and
$C^*_r(\mathcal{G}) \cong C_0(X) \rtimes_r \Gamma$ coincide. Indeed, if $\Gamma$ if
amenable, then so is $\mathcal{G}$: just pull back a mean on $\Gamma$ to each
$\mathcal{G}^x \cong \{x\} \times \Gamma$ to obtain a (constant) approximately invariant
continuous mean. It is possible for $\mathcal{G}$ to be amenable even when $\Gamma$ is
not: for example, the transformation groupoid of the free group acting on its boundary.

Amenability also has some other very important consequences.

\begin{theorem}[{\cite[Corollary~6.2.14 and Theorem~3.3.7]{AR}}]\label{thm:amen-nuc}
If $\mathcal{G}$ is an \'etale groupoid, then the following are equivalent: $\mathcal{G}$
is amenable; $C^*(\mathcal{G})$ is nuclear; $C^*_r(\mathcal{G})$ is nuclear.
\end{theorem}

It is also possible to study the nuclear dimension of a groupoid $C^*$-algebra in terms
of dynamical properties of the groupoid. For example, if $\mathcal{G}$ is a
transformation groupoid, then finite Rokhlin dimension of the action as discussed in
\cite[Section~5.1]{SzaboNotes} implies finite nuclear dimension for $C^*(\mathcal{G})$.
For more general \'etale groupoids there is a generalisation of Rokhlin dimension, called
\emph{dynamic asymptotic dimension} \cite{GWY} (see \cite[Section~6.2]{SzaboNotes})
which, for principal groupoids, guarantees finite nuclear dimension for the associated
$C^*$-algebra.

\begin{corollary}
If $\mathcal{G}$ and $\mathcal{H}$ are equivalent \'etale groupoids, then $\mathcal{G}$
is amenable if and only if $\mathcal{H}$ is.
\end{corollary}
\begin{proof}
A direct groupoid-theoretic proof of this can be found in \cite{AR}, but we will take a
shortcut: since $\mathcal{G}$ and $\mathcal{H}$ are equivalent, $C^*(\mathcal{G})$ and
$C^*(\mathcal{H})$ are Morita equivalent, and therefore $C^*(\mathcal{G})$ is nuclear if
and only if $C^*(\mathcal{H})$ is nuclear, so the result follows from
Theorem~\ref{thm:amen-nuc}.
\end{proof}

A beautiful result of Tu also relates amenability to the UCT.

\begin{theorem}\label{thm:Tu UCT}
Let $\mathcal{G}$ be an \'etale groupoid. If $\mathcal{G}$ is amenable, then
$C^*_r(\mathcal{G})$ satisfies the UCT.
\end{theorem}

\begin{remark}
It is a very important question whether every nuclear $C^*$-algebra satisfies the UCT.
The previous two theorems say that every nuclear groupoid $C^*$-algebra (associated to an
\'etale groupoid) satisfies the UCT. In fact, results of Barlak and Li \cite{BL} show
that this result extends---significantly---to twisted groupoid $C^*$-algebras. Moreover,
results of Renault \cite{Ren08} show that we can characterise twisted $C^*$-algebras
associated to \'etale effective groupoids amongst arbitrary $C^*$-algebras in purely
$C^*$-algebraic terms: they are the ones that admit a \emph{Cartan subalgebra}. We will
discuss this further in Chapter~\ref{ch:Weyl gpd}.
\end{remark}

Generally speaking, checking amenability using the definition is hard work. Fortunately,
there is a fairly extensive bag of tricks available, and usually the best approach is to
see if any of them apply or can be adapted to the example at hand.
Theorem~\ref{thm:amen-nuc} certainly belongs to this bag; we'll list a few more that come
up particularly frequently.

\begin{proposition}
If $\mathcal{G}$ is a principal \'etale groupoid and is an $F_\sigma$ set in
$\mathcal{G}^{(0)} \times \mathcal{G}^{(0)}$, then the orbit space
$\mathcal{G}^{(0)}/\mathcal{G}$ is a T$_0$ space if and only if each orbit $[x] \coloneq
\{r(\gamma) \mid s(\gamma) = x\}$ is locally closed (that is, each $[x]$ is relatively
open in its closure), and these equivalent conditions imply that $\mathcal{G}$ is
amenable.
\end{proposition}
\begin{proof}
Since $\mathcal{G}$ is principal, it is algebraically isomorphic to
$\mathcal{R}(\mathcal{G})$, and so $\mathcal{R}(\mathcal{G})$ is an $F_\sigma$ in
$\mathcal{G}^{(0)} \times \mathcal{G}^{(0)}$. So all of the conditions (1)--(14) in the
Ramsay--Mackey--Glimm dichotomy \cite[Theorem~2.1]{Ramsay} are equivalent, and in
particular \mbox{(4)\;$\iff$\;(5)} of that theorem shows that
$\mathcal{G}^{(0)}/\mathcal{G}$ is T$_0$ if and only if each $[x]$ is locally closed. It
then follows from \cite[Examples~2.1.4(2)]{AR} that $\mathcal{G}$ is a proper Borel
groupoid, and therefore amenable by \cite[Examples~3.2.2(2) and Theorem~3.3.7]{AR}.
\end{proof}

It follows, in particular, that every discrete equivalence relation is amenable (though
we could also deduce this from nuclearity of $\mathcal{K}$).

\begin{proposition}[{\cite[Proposition~5.3.37]{AR}}]
Let $\mathcal{G}$ be an \'etale groupoid, and suppose that for each $n \in \mathbb{N}$,
$\mathcal{G}_n$ is a closed subgroupoid of $\mathcal{G}$ with $\mathcal{G}^{(0)}
\subseteq \mathcal{G}_n \subseteq \mathcal{G}_{n+1}$ for all $n$. Further suppose that
each $\mathcal{G}_{n+1}$ is a proper $\mathcal{G}_n$-space, and that $\mathcal{G} =
\bigcup_n \mathcal{G}_n$. If each $\mathcal{G}_n$ is amenable, then $\mathcal{G}$ is
amenable.
\end{proposition}

As an example of this, consider the groupoid $\mathcal{R}_{2^\infty}$ of
Example~\ref{ex:2infty relation}. For each $n$, let $\mathcal{G}_n \coloneq \{(v x, w x)
\mid |v| = |w| = n, x \in X\}$. Then each $\mathcal{G}_n$ is closed (in fact clopen) in
$\mathcal{G}$ and contains $\mathcal{G}^{(0)}$, and the $\mathcal{G}_n$ are nested. In a
fixed $\mathcal{G}_n$ each orbit is finite, of size $2^n$, and so
$\mathcal{G}^{(0)}/\mathcal{G}_n$ is a standard Borel space. So \cite[Examples~2.1.4(2),
Examples~3.2.2(2) and Theorem~3.3.7]{AR} above show that each $\mathcal{G}_n$ is
amenable. Since $\mathcal{G} = \bigcup_n \mathcal{G}_n$, we deduce that $\mathcal{G}$ is
amenable.

\begin{proposition}[{\cite[Corollary~4.5]{RenWil}}]
Suppose that $\mathcal{G}$ is an \'etale groupoid and $c : \mathcal{G} \to \Gamma$ is a
continuous homomorphism into a discrete amenable group. If the clopen subgroupoid
$\ker(c) \subseteq \mathcal{G}$ is amenable, then $\mathcal{G}$ is amenable.
\end{proposition}

The above result for $\Gamma$ a discrete \emph{abelian} group was first proved by
Spielberg in \cite[Proposition~9.3]{Spielberg}. Spielberg's proof passed through
$C^*$-algebra theory, proving that $C^*(\mathcal{G})$ is nuclear, and deducing that
$\mathcal{G}$ is amenable from Theorem~\ref{thm:amen-nuc}. The Renault--Williams proof,
by contrast, is entirely groupoid theoretic. Moreover, the Renault--Williams result is
more general even than the one stated above: see \cite[Theorem~4.2]{RenWil}.

\begin{example}
Deaconu--Renault groupoids for actions of $\mathbb{N}^k$ are amenable: if $\mathcal{G}$
is a Deaconu--Renault groupoid over $\mathbb{N}^k$, then the map $c(x, p-q, y) \coloneq
p-q$ is a continuous cocycle into the abelian, and hence amenable, group $\Gamma$. An
argument very similar to the one used above to see that the $\mathcal{G}_n$ in
Example~\ref{ex:2infty relation} are amenable, shows that for each $n \in \mathbb{N}^k$
the subgroupoid $\mathcal{G}_n \coloneq \{(x, 0, y) \mid T^n x = T^n y\}$ is a proper
Borel groupoid, and then that $\ker(c) = \bigcup \mathcal{G}_n$ is amenable (see
\cite{SimWil2} for details).
\end{example}

\begin{example}
It follows that every graph groupoid is amenable, since it is the Deaconu--Renault
groupoid of the shift map on $E^\infty$.
\end{example}

\begin{proposition}
If $\mathcal{G}$ is an amenable \'etale groupoid and $\mathcal{H}$ is an open or a closed
subgroupoid of $\mathcal{G}$, then $\mathcal{H}$ is also amenable.
\end{proposition}

The rough idea here is to verify that an approximate invariant continuous mean for
$\mathcal{G}$ restricts to one for $\mathcal{H}$.

We finish the section with an example due to Willett that shows that, unlike the
situation for groups, in the setting of groupoids it is not the case that amenability is
equivalent to coincidence of the full and reduced $C^*$-algebras.

\begin{example}[{\cite[Lemma~2.8]{Willett}}]\label{eg:Willett}
Let $\mathbb{F}_2$ denote the free group on two generators. For $n \in \mathbb{N}$, let
$K_n$ denote the intersection of all the normal subgroups of $\mathbb{F}_2$ that have
index at most $n$ in $\mathbb{F}_2$. Willett shows that $\mathbb{F}_2$ is the infinite
union of the $K_n$. For each $n$, let $\Gamma_n \coloneq \mathbb{F}_2/K_n$, and let
$\Gamma_\infty = \mathbb{F}_2$. Let $\mathcal{G}^{(0)} \coloneq \mathbb{N} \cup
\{\infty\}$, the 1-point compactification of $\mathbb{N}$, and let $\mathcal{G}$ be the
group bundle $\bigcup_{x \in \mathcal{G}^{(0)}} \Gamma_x \times \{x\}$. For each $\gamma
\in \mathcal{G}_\infty = \mathbb{F}_2$, and each $n \in \mathbb{N}$, let $W(\gamma, n)
\coloneq \{(\pi_m(\gamma), m) \mid m \ge n\}$. Then
\[
\{W(\gamma, n) \mid \gamma \in \mathbb{F}_2, n \in \mathbb{N}\} \cup \big\{\{(\gamma, n)\} \mid n \in \mathbb{N}, \gamma \in \Gamma_n\big\}
\]
is a basis for a locally compact Hausdorff topology on $\mathcal{G}$ under which it is
\'etale. (Each fibre $\mathcal{G}_x$ is discrete in the relative topology.) This groupoid
is not amenable, because an approximate invariant mean on $\mathcal{G}$ would restrict to
an invariant mean on $\mathbb{F}_2$. However, Willet proves that $C^*(\mathcal{G}) =
C^*_r(\mathcal{G})$ by showing that the universal norm of $f \in C_c(\mathcal{G})$ is
given by $\sup_{n \in \mathbb{N}} \big\|f|_{\Gamma_n}\big\|$.
\end{example}

It is still not known, for example, whether a minimal, or even transitive, action for
which the full and reduced $C^*$-algebras coincide must have an amenable transformation
groupoid.

\section{Effective groupoids and uniqueness}

Recall that an action of a discrete group $\Gamma$ on a locally compact Hausdorff space
$X$ is \emph{effective} if, for each $g \in \Gamma$, the set $\{x \in X \mid g\cdot x =
x\}$ has empty interior.

In the corresponding transformation groupoid, the basic open sets are the bisections of
the form $\{g\} \times U$ where $U$ ranges over a base for the topology on
$\mathcal{G}^{(0)}$. So we can reinterpret effectiveness of an action in terms of the
transformation groupoid as follows: the action of $\Gamma$ on $X$ is effective if and
only if the interior of the isotropy in $\mathcal{G}$ is equal to $\mathcal{G}^{(0)}$.
This leads us to a definition.

\begin{definition}
Let $\mathcal{G}$ be an \'etale groupoid. We say that $\mathcal{G}$ is \emph{effective}
if $\operatorname{Iso}(\mathcal{G})^\circ = \mathcal{G}^{(0)}$.
\end{definition}

A Baire category argument shows that an action of a countable discrete group is effective
if and only if the points in $X$ at which the isotropy is trivial are dense in $X$. We
will need the equivalent for second-countable \'etale groupoids. We first need a
technical lemma that will come up again later.

\begin{lemma}\label{lem:separate}
Let $\mathcal{G}$ be an \'etale groupoid, and suppose that $\gamma \in \mathcal{G}$
satisfies $r(\gamma) \not= s(\gamma)$ and that $U$ is a bisection containing $\gamma$.
Then there is an open neighbourhood $V$ of $s(\gamma)$ such that $r(UV) \cap V =
\emptyset$.
\end{lemma}
\begin{proof}
We prove the contrapositive. That is, we suppose that $U$ is an open bisection and that
$\gamma \in U$, and that for every neighbourhood $V$ of $s(\gamma)$, we have $r(UV) \cap
V \not= \emptyset$. Choose a descending neighbourhood base $V_i$ at $s(\gamma)$ with each
$V_i \subseteq s(U)$. Since each $r(UV_i) \cap V_i$ is nonempty, for each $i$ we can
choose $\gamma_i \in UV_i$ with $r(\gamma_i) \in V_i$. Since $s|_U$ is a homeomorphism,
the $UV_i$ form a neighbourhood base at $\gamma$, and so $\gamma_i \to \gamma$. In
particular, $r(\gamma) = \lim_i r(\gamma_i)$. Since each $r(\gamma_i) \in V_i$ and the
$V_i$ are a neighbourhood base at $s(\gamma)$, we deduce that $r(\gamma_i) \to
s(\gamma)$; so $s(\gamma) = r(\gamma)$.
\end{proof}

\begin{lemma}[{\cite[Proposition~3.6]{Ren08}}]\label{lem:effective<->top prin}
Let $\mathcal{G}$ be a second-countable \'etale groupoid. Then $\mathcal{G}$ is effective
if and only if $\{x \mid \mathcal{G}^x_x = \{x\}\}$ is dense in $\mathcal{G}^{(0)}$.
\end{lemma}
\begin{proof}
First suppose that $\operatorname{Iso}(\mathcal{G})^\circ \not= \mathcal{G}^{(0)}$. Then
there is an open $U \subseteq \operatorname{Iso}(\mathcal{G})$ that is not contained in
$\mathcal{G}^{(0)}$. Since $\mathcal{G}^{(0)}$ is closed, $U \setminus \mathcal{G}^{(0)}$
is open and nonempty, so we can assume that $U$ has trivial intersection with
$\mathcal{G}^{(0)}$. Since $s$ is an open map, $s(U)$ is an open subset of
$\mathcal{G}^{(0)}$ such that $\mathcal{G}^x_x \setminus \{x\} \supseteq U x$ is nonempty
for all $x \in U$. So $\{x \mid \mathcal{G}^x_x = \{x\}\}$ is not dense in
$\mathcal{G}^{(0)}$.

Now suppose that $\operatorname{Iso}(\mathcal{G})^\circ = \mathcal{G}^{(0)}$. Let $U$ be
an open bisection that does not intersect $\mathcal{G}^{(0)}$. We claim that $U \setminus
\operatorname{Iso}(\mathcal{G})$ is open. To see this, suppose that $\gamma \in U
\setminus \operatorname{Iso}(\mathcal{G})$. We have $r(\gamma) \not= s(\gamma)$, and so
Lemma~\ref{lem:separate} shows that there is an open neighbourhood $V$ of $s(\gamma)$
contained in $s(U)$ such that $r(UV) \cap V$ is empty. Now $UV$ is an open neighbourhood
of $\gamma$ in $U \setminus \operatorname{Iso}(\mathcal{G})$. Since $U \cap
\operatorname{Iso}(\mathcal{G})$ has empty interior, we see that $U \setminus
\operatorname{Iso}(\mathcal{G})$ is dense in $U$. It follows that $A_U \coloneq r(U
\setminus \operatorname{Iso}(\mathcal{G})) \cup (\mathcal{G}^{(0)} \setminus
\overline{r(U)})$ is an open dense subset of $\mathcal{G}^{(0)}$. By definition, we have
$A_U \subseteq \mathcal{G}^{(0)} \setminus \{x \in \mathcal{G}^{(0)} \mid \mathcal{G}^x_x
\cap U \not= \emptyset\}$.

Now take a countable cover $\{U_i\}$ of $\mathcal{G} \setminus \mathcal{G}^{(0)}$ by open
bisections. By the preceding paragraph, the sets $A_{U_i} \subseteq \mathcal{G}^{(0)}$
are open dense sets. So the Baire category theorem implies that $\bigcap_i A_{U_i}$ is
dense in $\mathcal{G}^{(0)}$. By construction of the $A_{U_i}$, we have $\mathcal{G}^x_x
\cap U_j = \emptyset$ for every $x \in \bigcap_i A_{U_i}$ and every $j \in \mathbb{N}$.
Since the $U_i$ cover $\mathcal{G} \setminus \mathcal{G}^{(0)}$, it follows that
$\bigcap_i A_i \subseteq \{x \in \mathcal{G}^{(0)} \mid \mathcal{G}^x_x = \{x\}\}$, and
so the latter is dense as claimed.
\end{proof}

\begin{remark}
In the literature, the condition that $\{x \in \mathcal{G}^{(0)} \mid \mathcal{G}^x_x =
\{x\}\}$ is dense has gone by many names, including ``topologically principal," and
``topologically free;" but both of these terms have also been used elsewhere for
different concepts. So one has to be careful with these terms in the literature: in any
given article, check what definition is being used.
\end{remark}

Since, in an \'etale groupoid, the unit space is a clopen subset of $\mathcal{G}$, the
map $f \mapsto f|_{\mathcal{G}^{(0)}}$ is a map from $C_c(\mathcal{G})$ to
$C_c(\mathcal{G}^{(0)})$. We regard $C_c(\mathcal{G}^{(0)})$ as an abelian subalgebra of
$C_c(\mathcal{G})$. We will see later that this restriction map extends to a faithful
conditional expectation of $C^*_r(\mathcal{G})$ onto $C_0(\mathcal{G}^{(0)})$. But to
exploit this, we need some preliminary work.

\begin{lemma}\label{lem:cutdown}
Let $\mathcal{G}$ be an effective \'etale groupoid, and suppose that $\pi :
C_c(\mathcal{G}) \to \mathcal{B}(\mathcal{H})$ is a $^*$-representation that is injective
on $C_c(\mathcal{G}^{(0)})$. For each $f \in C_c(\mathcal{G})$ and each $\varepsilon >
0$, there exists $h \in C_c(\mathcal{G}^{(0)})$ such that $\|h\| = 1$, $hfh = h
f|_{\mathcal{G}^{(0)}} h \in C_c(\mathcal{G}^{(0)})$ and $\|\pi(hfh)\| \ge
\big\|f|_{\mathcal{G}^{(0)}}\big\| - \varepsilon$. In particular,
$\big\|\pi(f|_{\mathcal{G}^{(0)}})\big\| \le \|\pi(f)\|$.
\end{lemma}
\begin{proof}
Fix $f \in C_c(\mathcal{G})$, and $\varepsilon > 0$. It suffices to show that $\|\pi(f)\|
\ge \big\|f|_{\mathcal{G}^{(0)}}\big\| - \varepsilon$.

Since $\pi$ is injective on $C_c(\mathcal{G}^{(0)})$, we have
$\big\|\pi(f|_{\mathcal{G}^{(0)}})\big\| = \|f_{\mathcal{G}^{(0)}}\|_\infty$. Thus, since
$\mathcal{G}$ is effective, Lemma~\ref{lem:effective<->top prin} shows that there exists
$x \in \mathcal{G}^{(0)}$ such that $f(x) \ge \big\|f|_{\mathcal{G}^{(0)}}\big\| -
\varepsilon$ and $\mathcal{G}^x_x = \{x\}$.

Let $f_0 \coloneq f|_{\mathcal{G}^{(0)}}$. Then $f - f_0 \in C_c(\mathcal{G})$, and so
Lemma~\ref{lem:bisection span}, shows that we can write $f - f_0 = \sum^n_{i=1} f_i$
where each $f_i$ is supported on a bisection $U_i$ that does not intersect
$\mathcal{G}^{(0)}$. For each $i \not= 0$ such that $x \not\in
s(\operatorname{supp}(f_i))$, choose an open neighbourhood $V_i$ of $x$ in
$\mathcal{G}^{(0)}$ such that $V_i \cap s(\operatorname{supp}(f_i)) = \emptyset$. Then
for any $h \in C_c(V_i)$ we have $h f_i h = 0$. For each $i \not= 0$ such that $x \in
s(\operatorname{supp}(f_i))$, the unique element $\gamma$ of $U_i x$ belongs to
$\mathcal{G}_x \setminus \{x\}$. Since $\mathcal{G}^x_x = \{x\}$, we deduce that
$r(\gamma) \not= x$. By Lemma~\ref{lem:separate}, we can choose a neighbourhood $V_i$ of
$x$ with $V_i \subseteq s(U_i)$ and $r(U_i V_i) \cap V_i = \emptyset$. In particular,
$r(\operatorname{supp}(f_i)V_i) \cap V_i = \emptyset$. Hence Lemma~\ref{lem:bisection
conv} shows that for any $h \in C_c(V_i)$ we have $h f_i h = 0$.

Now let $V \coloneq \bigcap^n_{i=1} V_i$. Then $V$ is an open neighbourhood of $x$ and so
we can choose $h \in C_c(V)$ such that $h(x) = 1$ and $\|h\|_\infty = 1$. Since $\pi$ is
injective on $C_c(\mathcal{G}^{(0)})$, it is isometric on $C_c(\mathcal{G}^{(0)})$. We
have $h \in C_c(V_i)$ for each $i$, and so $h f_i h = 0$ for $i \not= 0$ by choice of the
$V_i$. Hence $hfh = hf_0h$, giving
\[
\|\pi(f)\|
    \ge \|\pi(hfh)\|
    = \|\pi(h f_0 h)\|
    = \|h f_0 h\|
    \ge (hf_0h)(x)
    = f(x)
    \ge \|\pi(f_0)\| - \varepsilon. \qedhere
\]
\end{proof}

\begin{proposition}\label{prp:FCE}
Let $\mathcal{G}$ be an \'etale groupoid. There is a faithful conditional expectation
$\Phi : C^*_r(\mathcal{G}^{(0)}) \to C_0(\mathcal{G}^{(0)})$ such that $\Phi(f) =
f|_{\mathcal{G}^{(0)}}$ for all $f \in C_c(\mathcal{G})$. We have $j(\Phi(a)) =
j(a)|_{\mathcal{G}^{(0)}}$ for all $a \in C^*_r(\mathcal{G})$.
\end{proposition}
\begin{proof}
Proposition~\ref{prp:j map} shows that $\big\|f|_{\mathcal{G}^{(0)}}\big\|_\infty \le
\|f\|_r$ for every $f \in C_c(\mathcal{G})$. By Corollary~\ref{cor:norm relation}, we
have $\big\|f|_{\mathcal{G}^{(0)}}\big\|_\infty = \big\|f|_{\mathcal{G}^{(0)}}\big\|_r$,
and we deduce that $f \mapsto f|_{\mathcal{G}^{(0)}}$ is a norm-decreasing idempotent
linear map from $C_c(\mathcal{G})$ to $C_0(\mathcal{G}^{(0)})$. It therefore extends to a
idempotent linear map $\Phi$ of norm~1 of $C^*_r(\mathcal{G})$ onto
$C_0(\mathcal{G}^{(0)})$. By \cite[Theorem~II.6.10.2]{Blackadar}, this $\Phi$ is a
conditional expectation. Since $j(\Phi(f)) = j(f|_{\mathcal{G}^{(0)}})$ for $f \in
C_c(\mathcal{G})$, continuity of $j$ and $\Phi$ give $j(\Phi(a)) =
j(a)|_{\mathcal{G}^{(0)}}$ for all $a \in C^*_r(\mathcal{G})$.

To see that $\Phi$ is faithful, suppose that $a \not= 0$. Then there exist $x \in
\mathcal{G}^{(0)}$ and $\gamma \in \mathcal{G}_x$ such that
$\big(\pi_x(a^*a)\delta_\gamma \mid \delta_\gamma\big) \not= 0$. Applying the unitary
equivalence between $\pi_x$ and $\pi_{r(\gamma)}$ obtained in Proposition~\ref{prp:reg
reps}, we see that $\big(\pi_{r(\gamma)}(a^*a)\delta_{r(\gamma)} \mid
\delta_{r(\gamma)}\big) \not= 0$. That is, $j(a^*a)(r(\gamma)) \not= 0$. Hence
$j(\Phi(a^*a)) = \Phi(a^*a)|_{\mathcal{G}^{(0)}} \not= 0$, and we conclude that
$\Phi(a^*a) \not= 0$.
\end{proof}

We can now prove our main theorem for this section.

\begin{theorem}\label{thm:uniqueness}
Let $\mathcal{G}$ be an effective \'etale groupoid. If $\phi : C^*_r(\mathcal{G}) \to A$
is a $C^*$-homomorphism that is injective on $C_0(\mathcal{G}^{(0)})$, then it is
injective on all of $C^*_r(\mathcal{G})$.
\end{theorem}
\begin{proof}
By Lemma~\ref{lem:cutdown}, we have $\|\pi(f)\| \ge
\big\|\pi(f|_{\mathcal{G}^{(0)}})\big\|$ for all $f \in C_c(\mathcal{G})$, and so there
is a well-defined linear map $\Psi : \pi(C^*_r(\mathcal{G})) \to
\pi(C_0(\mathcal{G}^{(0)}))$ such that $\Psi(\pi(f)) = \pi(f|_{\mathcal{G}^{(0)}})$ for
$f \in C_c(\mathcal{G})$. It follows from continuity that the faithful conditional
expectation $\Phi$ of Proposition~\ref{prp:FCE} satisfies $\Psi \circ \pi = \pi \circ
\Phi$. Now we follow the standard argument, using injectivity of $\phi$ on
$C_0(\mathcal{G}^{(0)})$ at the third implication:
\[
\phi(a) = 0 \Rightarrow \Psi(\phi(a^*a)) = 0 \Rightarrow \phi(\Phi(a^*a)) = 0 \Rightarrow \Phi(a^*a) = 0 \Rightarrow a = 0.\qedhere
\]
\end{proof}

\begin{remark}
An equivalent restatement of Theorem~\ref{thm:uniqueness} is that if $\mathcal{G}$ is a
effective \'etale groupoid, then every nontrivial ideal of $C^*_r(\mathcal{G})$ has
nonzero intersection with $C_0(\mathcal{G}^{(0)})$.
\end{remark}

\section{Invariant sets, ideals, and simplicity}

Our aim in this section is to shed some light on the ideal structure of
$C^*(\mathcal{G})$, and to characterise simplicity of $C^*(\mathcal{G})$ when
$\mathcal{G}$ is an amenable \'etale groupoid. When $\mathcal{G}$ is not assumed
amenable, things become more complicated: most of the statements in the section remain
true for either the full or the reduced $C^*$-algebra, but typically not both, and care
is required.

We will say that a subset $U$ of $\mathcal{G}^{(0)}$ is \emph{invariant} if
$r(\mathcal{G} U) \subseteq U\}$. Observe that if $\mathcal{G}$ is a transformation
groupoid, then an open $U \subseteq \mathcal{G}^{(0)}$ is invariant precisely if $C_0(U)
\subseteq C_0(\mathcal{G}^{(0)})$ is an invariant ideal as in
\cite[Section~3.1.2]{WilliamsNotes}.

\begin{lemma}\label{lem:I gives inv set}
Let $\mathcal{G}$ be an \'etale groupoid, and let $I$ be an ideal of $C^*(\mathcal{G})$.
Then there is an open invariant subset $\operatorname{supp}(I) \subseteq
\mathcal{G}^{(0)}$ such that $I \cap C_0(\mathcal{G}^{(0)}) = \{f \in
C_0(\mathcal{G}^{(0)}) \mid f(x) = 0 \text{ for all } x \in \mathcal{G}^{(0)} \setminus
\operatorname{supp}(I)\} \cong C_0(\operatorname{supp}(I)$.
\end{lemma}
\begin{proof}
Since $I \cap C_0(\mathcal{G}^{(0)})$ is an ideal of a commutative $C^*$-algebra, it has
the form $I \cap C_0(\mathcal{G}^{(0)}) = C_0(\operatorname{supp}(I))$ for an open set
$\operatorname{supp}(I) \subseteq \mathcal{G}^{(0)}$. We just have to show that this set
is invariant. For this, suppose that $x \in \operatorname{supp}(I)$. Choose $f \in I \cap
C_0(\mathcal{G}^{(0)})$ such that $f(x) \not= 0$. Fix $\gamma \in G_x$; we must show that
$r(\gamma) \in \operatorname{supp}(I)$. For this, take an open bisection $U$ containing
$\gamma$ and fix $h \in C_c(U)$ with $h(\gamma) = 1$. Lemma~\ref{lem:bisection conv}
shows that $hfh^*$ is supported on $U \mathcal{G}^{(0)} U^{-1} = r(U) \subseteq
\mathcal{G}^{(0)}$. So $hfh^* \in I \cap C_0(\mathcal{G}^{(0)})$. Since $(hfh)(r(\gamma))
= h(\gamma) f(x) h^*(\gamma^{-1}) = f(x) \not= 0$, we deduce that $r(\gamma) \in
\operatorname{supp}(I)$.
\end{proof}

If $U$ is an open invariant subset of $\mathcal{G}^{(0)}$, then $\mathcal{G} U$ is an
open subgroupoid of $\mathcal{G}$ and so a locally compact Hausdorff \'etale groupoid in
the relative topology. Similarly, $\mathcal{G} \setminus \mathcal{G} U$ is a closed
subgroupoid of $\mathcal{G}$, and hence again a locally compact \'etale groupoid in the
subspace topology.

\begin{proposition}\label{prp:exact seq}
Let $\mathcal{G}$ be an \'etale groupoid, and let $U$ be an open invariant subset of
$\mathcal{G}^{(0)}$. Define $W \coloneq \mathcal{G}^{(0)} \setminus U$. The inclusion
$C_c(\mathcal{G} U) \hookrightarrow C_c(\mathcal{G})$ extends to an injective
$C^*$-homomorphism $i_U : C^*(\mathcal{G} U) \hookrightarrow C^*(\mathcal{G})$. The image
$I_U \coloneq i_U(C^*(\mathcal{G} U))$ is an ideal of $C^*(\mathcal{G})$, and is
generated as an ideal by $C_c(U) \subseteq C_0(\mathcal{G}^{(0)})$. There is a
homomorphism $\pi_U : C^*(\mathcal{G}) \to C^*(\mathcal{G} W)$ satisfying $\pi_U(f) =
f|_{\mathcal{G} W}$ for all $f \in C_c(\mathcal{G})$. Moreover the sequence
\[
0 \longrightarrow C^*(\mathcal{G} U) \stackrel{i_U}{\longrightarrow} C^*(\mathcal{G}) \stackrel{\pi_U}{\longrightarrow} C^*(\mathcal{G} W) \longrightarrow 0
\]
is exact.
\end{proposition}
\begin{proof}
The inclusion $C_c(\mathcal{G} U) \hookrightarrow C_c(\mathcal{G}) \hookrightarrow
C^*(\mathcal{G})$ is a $^*$-homomorphism, so the universal property of $C^*(\mathcal{G})$
shows that there is a homomorphism $i_U : C^*(\mathcal{G} U) \to C^*(\mathcal{G})$ as
required.

To see that the image of $i_U$ is an ideal, observe that if $f \in C_c(\mathcal{G}_U)$,
and if $g \in C_c(\mathcal{G})$ is supported on a bisection, then $\operatorname{supp}(g
* f) \subseteq \operatorname{supp}(g)\operatorname{supp}(f) \subseteq \mathcal{G} (\mathcal{G} U) = \mathcal{G} U$;
similarly (or by taking adjoints) we have $\operatorname{supp}(f * g) \subseteq
\mathcal{G}_U$. So $i_U(C_c(\mathcal{G} U))$ is an algebraic 2-sided ideal of
$C^*(\mathcal{G})$, and hence $I_U$ is an ideal by continuity. Since $C_c(U)$ contains an
approximate identity for $C^*(\mathcal{G} U)$, the ideal $I_U$ is generated as an ideal
by $C_c(U)$.

Since every $^*$-homomorphism $\pi$ of $C_c(\mathcal{G})$ into a $C^*$-algebra $B$ can be
composed with a faithful representation of $B$ to obtain a $^*$-representation of
$C_c(\mathcal{G})$ that achieves the same norm on every element, it suffices to show that
there is a $^*$-homomorphism $\pi : C_c(\mathcal{G}) \to B$ for some $C^*$-algebra $B$
such that $\|\pi(f)\| \ge \|f\|_{C^*(\mathcal{G} U)}$ for every $f \in
C_c(\mathcal{G}_u)$. To see this observe that $C_c(\mathcal{G} U)$ is an algebraic ideal
of $C_c(\mathcal{G})$, and so for each $f \in C_c(\mathcal{G})$, there is a linear map
$\pi(f) : C_c(\mathcal{G} U) \to C_c(\mathcal{G} U)$ given by $\pi(f) g = f * g$. If we
regard $C^*(\mathcal{G} U)$ as a Hilbert bimodule over itself with inner product $\langle
a, b\rangle_{C^*(\mathcal{G} U)} = a^*b$, then $\langle \pi(f) a,
b\rangle_{C^*(\mathcal{G} U)} = a^* \pi(f^*) b = \langle a, \pi(f^*)
b\rangle_{C^*(\mathcal{G} U)}$; so $\pi(f)$ is an adjointable operator on
$C^*(\mathcal{G} U)_{C^*(\mathcal{G} U)}$. From this we see that $\pi$ is a
$^*$-homomorphism into the algebra $\mathcal{L}(C^*(\mathcal{G} U)_{C^*(\mathcal{G} U)})$
of adjointable operators on $C^*(\mathcal{G} U)_{C^*(\mathcal{G} U)}$. For $f \in
C_c(\mathcal{G} U)$, we have
\[\textstyle
\|\pi(f)\| \ge \big\|\pi(f) \frac{f^*}{\|f\|}\big\| = \|f\|
\]
as required.

The map $f \mapsto f|_W$ is a $^*$-homomorphism of $C_c(\mathcal{G})$ onto
$C_c(\mathcal{G} W)$ and hence determines a $^*$-homomorphism from $C_c(\mathcal{G})$ to
$C^*(\mathcal{G} W)$; so once again the universal property gives a homomorphism $\pi :
C^*(\mathcal{G}) \to C^*(\mathcal{G}_W)$ that extends restriction of functions. Clearly
$\ker\pi$ contains $C_c(\mathcal{G} U)$ and hence the image of $i_U$. In particular,
$\pi$ induces a homomorphism $\tilde\pi : C^*(\mathcal{G})/I_U \to C^*(\mathcal{G} W)$.
To see that this homomorphism is injective, observe that since $C_c(\mathcal{G} U)
\subseteq I_U$, if $f,g \in C_c(\mathcal{G})$ satisfy $f|_{\mathcal{G} W} =
g|_{\mathcal{G} W}$, then $f - g \in I_U$. Hence there is a well-defined
$^*$-homomorphism $\phi : C_c(\mathcal{G} W) \to C^*(\mathcal{G})/I_U$ such that
$\phi(f|_W) = f + I_U$ for all $f \in C_c(\mathcal{G})$. The universal property of
$C^*(\mathcal{G} W)$ shows that $\phi$ extends to a homomorphism $\phi : C^*(\mathcal{G}
W) \to C^*(\mathcal{G})/I_U$. The image of $\phi$ contains the image of
$C_c(\mathcal{G})$ in the quotient and so $\phi$ is surjective. Since $\tilde{\pi}_U
\circ \phi$ is the identity map on $C_c(\mathcal{G} W)$ we see that $\tilde{\pi}_U \circ
\phi$ is the identity homomorphism; so surjectivity of $\phi$ ensures that
$\tilde{\pi}_U$ is injective, and therefore that $\ker(\pi_U) = I_U$.
\end{proof}

The preceding proposition holds for general groupoids, but the proof requires the
Disintegration Theorem. For reduced $C^*$-algebras, the corresponding maps $i^r_U$ and
$\pi^r_U$ between reduced $C^*$-algebras exist, $i^r_U$ is injective, $\pi^r_U$ is
surjective, and $\ker\pi^r_U$ contains the image of $i^r_U$ (these statements are all
relatively easy to prove using the properties of regular representations). But the
sequence is not necessarily exact. The first example of this was given by Skandalis
\cite[Appendix, p.~35]{Ren:ideals}, but Willet's example, Example~\ref{eg:Willett}, also
gives an instance of the failure of exactness: let $\mathcal{G}$ be Willett's groupoid.
Since $\mathcal{G}$ is a group bundle, the set $\mathbb{N} \subseteq \mathbb{N} \cup
\{\infty\} = \mathcal{G}^{(0)}$ is an open invariant subset. Since $\mathcal{G}
\mathbb{N}$ is a (discrete) bundle of finite groups, it is amenable (just take normalised
counting measure on each fibre), so $C^*(\mathcal{G} \mathbb{N}) = C^*_r(\mathcal{G}
\mathbb{N})$, and Willett's result says that $C^*(\mathcal{G}) = C^*_r(\mathcal{G})$. So
Proposition~\ref{prp:exact seq} shows that the sequence
\[
    0 \to C^*_r(\mathcal{G} \mathbb{N}) \longrightarrow C^*_r(\mathcal{G}) \longrightarrow C^*(\mathbb{F}_2) \to 0
\]
is exact. Since $C^*_r(\mathbb{F}_2)$ is a proper quotient of $C^*(\mathbb{F}_2)$, we
deduce that
\[
    0 \to C^*_r(\mathcal{G} \mathbb{N}) \longrightarrow C^*_r(\mathcal{G}) \longrightarrow C_r^*(\mathbb{F}_2) \to 0
\]
is not exact.

\smallskip

We will say that $\mathcal{G}$ is \emph{strongly effective} if $\mathcal{G} W$ is
effective for every closed invariant subset of $\mathcal{G}^{(0)}$. This is a strictly
stronger condition than effectiveness: Consider the action of $\mathbb{Z}$ on its own
1-point compactification $\mathbb{Z} \cup \{\infty\}$ determined by continuous extension
of the translation action of $\mathbb{Z}$ on itself. The resulting transformation
groupoid $\mathcal{G}$ is effective because the only point with nontrivial isotropy is
$\infty$; but $\{\infty\}$ is a closed invariant subset of $\mathcal{G}^{(0)}$ and
clearly $\mathcal{G} \{\infty\} \cong \mathbb{Z}$ is not effective.

\begin{theorem}\label{thm:ideals}
Let $\mathcal{G}$ be an amenable \'etale groupoid. The map $U \mapsto I_U$ is an
injection from the set of open invariant subsets of $\mathcal{G}^{(0)}$ to the set of
ideals of $C^*(\mathcal{G})$. It is bijective if and only if $\mathcal{G}$ is strongly
effective.
\end{theorem}
\begin{proof}
If $U,V$ are distinct open invariant sets, then $I_U$ and $I_V$ are distinct because $I_U
\cap C_0(\mathcal{G}^{(0)}) = C_0(U) \not= C_0(V) = I_V \cap C_0(\mathcal{G}^{(0)})$. So
$U \mapsto I_U$ is injective.

First suppose that $\mathcal{G}$ is strongly effective. Fix an ideal $I$ of
$C^*(\mathcal{G})$. We must show that $I = I_{\operatorname{supp}(I)}$. To see this,
first observe that Proposition~\ref{prp:exact seq} shows that
$I_{\operatorname{supp}(I)}$ is generated as an ideal by $C_c(\operatorname{supp}(I))$,
which is a subset of $I \cap C_0(\mathcal{G}^{(0)})$ by definition. So
$I_{\operatorname{supp}(I)} \subseteq I$. Thus the quotient map induces a homomorphism
$\tilde{q} : C^*(\mathcal{G})/I_{\operatorname{supp}(I)} \to C^*(\mathcal{G})/I$. Let $W
\coloneq \mathcal{G}^{(0)} \setminus \operatorname{supp}(I)$. Proposition~\ref{prp:exact
seq} gives an isomorphism $\tilde{\pi} : C^*(\mathcal{G})/I_{\operatorname{supp}(I)}
\cong C^*(\mathcal{G} W)$ extending restriction of functions on $C_c(\mathcal{G})$. Since
$I \cap C_0(\mathcal{G}^{(0)}) = C_0(\operatorname{supp}(I))$, we see that
$\tilde{q}\circ \tilde{\pi}$ is injective on $C_0(W)$. Since $\mathcal{G}$ is strongly
effective, $\mathcal{G} W$ is effective, and it then follows from
Theorem~\ref{thm:uniqueness} that $\tilde{\pi}$ is injective. So $I =
I_{\operatorname{supp}{I}}$.

Now suppose that $\mathcal{G}$ is not strongly effective. Fix a closed invariant set $W
\subseteq \mathcal{G}^{(0)}$ such that $\mathcal{H} \coloneq \mathcal{G} W$ is not
effective. We will construct a representation $\psi$ of $C^*(\mathcal{H})$ such that
$\psi$ is faithful on $C_0(\mathcal{H}^{(0)})$ but is not faithful on $C^*(\mathcal{H})$.

Recall that the \emph{orbit} $[x]$ of $x \in \mathcal{H}^{(0)}$ is the set $\{r(\gamma)
\mid \gamma \in \mathcal{H}_x\}$. For each $x \in \mathcal{H}^{(0)}$, there is a linear
map $\epsilon_x : C_c(\mathcal{H}) \to \mathcal{B}(\ell^2([x]))$ given by
$\epsilon_x(f)\delta_y = \sum_{\gamma \in \mathcal{H}_y} f(\gamma) \delta_{r(\gamma)}$.
For $f,g \in C_c(\mathcal{H})$ and $y,z \in [x]$, we have
\[
\big(\epsilon_x(f)\epsilon_x(g)\delta_y \mid \delta_z\big)
    = \sum_{\beta \in \mathcal{H}_y} g(\beta) \big(\epsilon_x(f)\delta_{r(\alpha)} \mid \delta_z\big)
    = \sum_{\beta \in \mathcal{H}_y, \alpha \in \mathcal{H}_{r(\beta)}^z} f(\alpha)g(\beta).
\]
On the other hand,
\[
\big(\epsilon_x(f * g)\delta_y \mid \delta_z\big)
    = \sum_{\gamma \in \mathcal{H}^z_y} (f * g)(\gamma)
    = \sum_{\gamma \in \mathcal{H}^z_y} \sum_{\alpha\beta = \gamma} f(\alpha)g(\beta)
    = \sum_{\gamma \in \mathcal{H}^z_y, \alpha \in \mathcal{H}^z} f(\alpha)g(\alpha^{-1}\gamma).
\]
The map $(\alpha, \beta) \mapsto (\alpha, \alpha\beta)$ is a bijection from $\{(\alpha,
\beta) \mid \beta \in \mathcal{H}_y, \alpha\in \mathcal{H}_{r(\beta)}^z\}$ to $\{(\alpha,
\gamma) \mid \gamma \in \mathcal{H}^z_y, \alpha \in \mathcal{H}^z\}$ (the inverse is
$(\alpha,\gamma) \mapsto (\alpha, \alpha^{-1}\gamma)$). So the two sums are equal, and
therefore $\epsilon_x$ is multiplicative. Likewise,
\[
\big(\epsilon_x(f^*) \delta_y \mid \delta_z\big)
    = \sum_{\gamma \in \mathcal{H}^z_y} f^*(\gamma)
    = \sum_{\gamma \in \mathcal{H}^z_y} \overline{f(\gamma^{-1})}
    = \sum_{\eta \in \mathcal{H}^y_z} \overline{f(\eta)}
    = \big(\delta_y \mid \epsilon_x(g)\delta_z\big).
\]
So $\epsilon_x$ is a $^*$-representation of $C_c(\mathcal{H})$. The universal property of
$C^*(\mathcal{H})$ therefore shows that $\epsilon_x$ extends to a representation of
$C^*(\mathcal{H})$. Let $\psi \coloneq \bigoplus_{x \in \mathcal{H}^{(0)}} \epsilon_x$.
If $f \in C_0(\mathcal{H}^{(0)})\setminus \{0\}$, say $f(x) \not= 0$, then $\epsilon_x(f)
\delta_x = f(x)\delta_x \not= 0$, so $\psi$ is faithful on $C_0(\mathcal{H}^{(0)})$.

Since $\mathcal{H}$ is not effective, there is an open bisection $U$ contained in
$\operatorname{Iso}(\mathcal{H}) \setminus \mathcal{H}^{(0)}$. Fix $f \in C_c(r(U))$, and
define $\tilde{f} \in C_c(U)$ by
\[
\tilde{f}(\gamma) = f(r(\gamma))\quad\text{ for all $\gamma \in U$.}
\]
Since $U$ and $r(U)$ are open in $\mathcal{H}$, we can regard $f$ and $\tilde{f}$ as
elements of $C_c(\mathcal{H})$. Since $U$ is contained in
$\operatorname{Iso}(\mathcal{H})$, for $x \in \mathcal{H}^{(0)}$ and $y \in [x]$, we have
\begin{align*}
\epsilon_x(\tilde{f}) \delta_y
    &= \begin{cases}
        \tilde{f}(Uy) \delta_y &\text{ if $y \in r(U)$} \\
        0 &\text{ otherwise}
    \end{cases} \\
    &= \begin{cases}
        f(y) \delta_y &\text{ if $y \in r(U)$}\\
        0 &\text{ otherwise}
    \end{cases}\\
    &= \epsilon_x(f) \delta_y.
\end{align*}
So $f - \tilde{f} \in \ker(\psi)$, and $f - \tilde{f} \not= 0$ because
$\operatorname{supp}(f) \subseteq \mathcal{G}^{(0)}$ and $\operatorname{supp}(\tilde{f})
\subseteq U \subseteq \mathcal{G} \setminus \mathcal{G}^{(0)}$.

Now consider the ideals $I_{\mathcal{G}^{(0)} \setminus W}$ and $\ker(\psi \circ
\pi_{\mathcal{G}^{(0)} \setminus W})$. We have just seen that they have identical
intersection with $C_0(\mathcal{G}^{(0)})$ (namely $C_0(\mathcal{G}^{(0)} \setminus W)$),
but are not equal. So $U \mapsto I_U$ is not a bijection.
\end{proof}

\begin{remark}
\begin{itemize}
\item Whether or not $\mathcal{G}$ is amenable or strongly effective, the map $U
    \mapsto I_U$ is an injection from the collection of open invariant sets of
    $\mathcal{G}^{(0)}$ to the space of ideals of $C^*(\mathcal{G})$.
\item We could replace amenability of $\mathcal{G}$ with the requirement that
    $C^*(\mathcal{G} W) = C^*_r(\mathcal{G} W)$ for every closed invariant $W$ in
    both Proposition~\ref{prp:exact seq} and Theorem~\ref{thm:ideals}. Conversely, if
    there exists a closed invariant set $W$ such that $C^*(\mathcal{G} W) \not=
    C^*_r(\mathcal{G} W)$, then $U \mapsto I_U$ is not surjective because the kernel
    of $\pi_r^{\mathcal{G} W} \circ \pi_{\mathcal{G}^{(0)} \setminus W}$ is not in
    its range.
\item In the non-amenable case, the map $U \mapsto \ker(\pi^r_U)$ remains an
    injection from open invariant sets to ideals of the reduced $C^*$-algebra. It is
    possible for this map to be bijective even if $\mathcal{G}$ is not effective: for
    example, the reduced $C^*$-algebra of the free group is simple.
\end{itemize}
\end{remark}

We say that a groupoid $\mathcal{G}$ is \emph{minimal} if for every $x \in
\mathcal{G}^{(0)}$ the orbit $[x]$ is dense in $\mathcal{G}^{(0)}$.

\begin{lemma}\label{lem:minimal}
Let $\mathcal{G}$ be a topological groupoid. Then $\mathcal{G}$ is minimal if and only if
the only open invariant subsets of $\mathcal{G}^{(0)}$ are $\emptyset$ and
$\mathcal{G}^{(0)}$.
\end{lemma}
\begin{proof}
If $\emptyset \not= W \subsetneq \mathcal{G}^{(0)}$ is a nontrivial open invariant set,
then for any $x \in W$ we have $\overline{[x]} \subseteq W \not= \mathcal{G}^{(0)}$, and
so $\mathcal{G}$ is not minimal.

If $\mathcal{G}$ is minimal, then the only nonempty closed invariant subset of
$\mathcal{G}^{(0)}$ is $\mathcal{G}^{(0)}$. Hence the only open invariant subsets of
$\mathcal{G}^{(0)}$ are $\emptyset$ and $\mathcal{G}^{(0)}$.
\end{proof}

We therefore obtain the following characterisation of simplicity \cite{BCFS}.

\begin{theorem}\label{thm:simple}
Let $\mathcal{G}$ be an amenable \'etale groupoid. Then $C^*(\mathcal{G})$ is simple if
and only if $\mathcal{G}$ is effective and minimal.
\end{theorem}
\begin{proof}
First suppose that $\mathcal{G}$ is effective and minimal. Since it is minimal, the only
nonempty closed invariant set is $\mathcal{G}^{(0)}$ and so $\mathcal{G}$ is (trivially)
strongly effective. So Theorem~\ref{thm:ideals} and Lemma~\ref{lem:minimal} show that the
only ideals of $C^*(\mathcal{G})$ are $I_\emptyset = \{0\}$ and $I_{\mathcal{G}^{(0)}} =
C^*(\mathcal{G})$.

If $\mathcal{G}$ is not minimal, then Lemma~\ref{lem:minimal} gives a nontrivial open
invariant subset of $\mathcal{G}^{(0)}$ and so Theorem~\ref{thm:ideals} gives a
nontrivial ideal. Likewise, if $\mathcal{G}$ is not minimal, then $U \mapsto I_U$ is
injective but not bijective by Theorem~\ref{thm:ideals}; so it is not surjective. Since
$\{0\} = I_\emptyset$ and $C^*(\mathcal{G}) = I_{\mathcal{G}^{(0)}}$ are in its range, it
follows that $C^*(\mathcal{G})$ has a nontrivial ideal.
\end{proof}

We also obtain a sufficient condition for reduced $C^*$-algebras.

\begin{proposition}
If $\mathcal{G}$ is an effective, minimal, \'etale groupoid, then $C^*_r(\mathcal{G})$ is
simple.
\end{proposition}
\begin{proof}
Let $I$ be a nonzero ideal of $C^*_r(\mathcal{G})$; we must show that $I =
C^*_r(\mathcal{G})$. By Theorem~\ref{thm:uniqueness}, we have $I \cap
C_0(\mathcal{G}^{(0)}) \not= \{0\}$. Now if $q_I : C^*_r(\mathcal{G}) \to
C^*_r(\mathcal{G})/I$ is the quotient map, and $\pi_r : C^*(\mathcal{G}) \to
C^*_r(\mathcal{G})$ is the canonical surjection, then $J \coloneq \ker(q_I \circ \pi_r)
\cap C_0(\mathcal{G}^{(0)}) \supseteq I \cap C_0(\mathcal{G}^{(0)})$ is nonzero. So
Lemma~\ref{lem:I gives inv set} shows that $\operatorname{supp}(J)$ is a nonempty open
invariant set. Since $\mathcal{G}$ is minimal, it follows that $\operatorname{supp}(J) =
\mathcal{G}^{(0)}$ and so $C_0(\mathcal{G}^{(0)}) \subseteq J$. So
$\pi_r(C_0(\mathcal{G}^{(0)})) = C_0(\mathcal{G}^{(0)})$ is contained in $I$, and we
deduce that $I = C^*_r(\mathcal{G})$.
\end{proof}

\section{Pure infiniteness}

In this section we briefly discuss a result of Anantharaman-Delaroche giving a sufficient
condition for $C^*(\mathcal{G})$ to be purely infinite. There is no improving on her
argument, so the treatment here is more or less exactly the same as in \cite{A-D}.

\begin{definition}[{\cite[Definition~2.1]{A-D}}]
Let $\mathcal{G}$ be an \'etale groupoid. We say that $\mathcal{G}$ is \emph{locally
contracting at $x \in \mathcal{G}^{(0)}$} if for every open neighbourhood $V$ of $x$,
there is an open set $W \subseteq V$ and an open bisection $U$ such that $\overline{W}
\subseteq s(U)$ and $r(U\overline{W}) \subsetneq W$. We say that $\mathcal{G}$ is
\emph{locally contracting} if it is locally contracting at $x$ for every $x \in
\mathcal{G}^{(0)}$.
\end{definition}

\begin{theorem}[{\cite[Proposition~2.4]{A-D}}]
Suppose that $\mathcal{G}$ is an effective, locally contracting, \'etale groupoid. Then
$C^*_r(\mathcal{G})$ is purely infinite.
\end{theorem}
\begin{proof}
It suffices to show that for every $a \in A_+$, the hereditary subalgebra generated by
$a$ contains an infinite projection. So fix $a \in A_+$. We may assume without loss of
generality that the faithful conditional expectation $\Phi : C^*_r(\mathcal{G}) \to
C_0(\mathcal{G}^{(0)})$ of Proposition~\ref{prp:FCE} satisfies $\|\Phi(a)\| = 1$. Since
$C_c(\mathcal{G})$ is dense in $C^*(\mathcal{G})$, we can choose $b \in C_c(\mathcal{G})
\cap A^+$ such that $b \le a$ and $\|a - b\| < \frac{1}{4}$. Since $\Phi$ is norm
decreasing we deduce that $\|\Phi(a) - \Phi(b)\| < \frac{1}{4}$, and so $\|\Phi(b)\| >
\frac{3}{4}$. Lemma~\ref{lem:cutdown} with $\varepsilon = (\|\Phi(b)\| - 3/4)/2$ gives a
function $h \in C_c(\mathcal{G}^{(0)})$ such that $\|h\| = 1$, $h b h = h \Phi(b) h \in
C_0(\mathcal{G}^{(0)})$, and $\|h b h\| > \frac{3}{4}$. Since $hbh \le b \le a$, it
suffices to find an infinite projection $p$ and a partial isometry $w$ in
$C^*(\mathcal{G})$ such that $w p w^* \le h \Phi(b) h =: b_0$.

Using that $\mathcal{G}$ is locally contracting, we choose an open $V$ with $\overline{V}
\subseteq \{x \in \mathcal{G}^{(0)} \mid b_0(x) > 3/4\}$ and a bisection $B$ with
$\overline{V} \subseteq s(B)$ and $r(B \overline{V}) \subsetneq V$. Let $T_B : s(B) \to
r(B)$ be the homeomorphism $T_B(s(\gamma)) = r(\gamma)$ for $\gamma \in B$. Then
$T_B(\overline{V}) = r(B \overline{V})$ is a compact subset of $V$ and is not all of $V$.
So we can choose $k \in C_c(V)$ such that $k$ is identically~1 on $T_B(\overline{V})$.
Define $x \in C_c(BV)$ by $x(\gamma) = k(s(\gamma))$ for $\gamma \in BV$. We have $x^*x =
k^2$, and in particular $x^*x$ is identically~1 on $r(B \overline{V}) \supseteq
r(\operatorname{supp} x)$. Hence $x^*x x = x$.

So $x$ is a scaling element. There is a standard trick for constructing a projection from
such an element: Define $v$ in the minimal unitisation of $C^*_r(\mathcal{G})$ by $v = x
+ (1 - x^*x)^{1/2}$. We have
\[
v^*v = x^*x + x^*(1 - x^*x)^{1/2} + (1 - x^*x)^{1/2} x + (1 - x^*x).
\]
Since $(1 - x^*x) x = 0$, every $(1 - x^*x)^n x = 0$ and then by continuity $f(1 - x^*x)
x = 0$ for every $f \in C(\sigma(1 - x^*x))$. In particular, $(1 - x^*x)^{1/2} x = 0 =
x^*(1 - x^*x)^{1/2}$, and so $v^*v = 1$. Consequently, $vv^*$ is a projection. We compute
\begin{align*}
vv^* &= xx^* + x(1-x^*x)^{1/2} + (1-x^*x)^{1/2}x^* + (1 - x^*x)\\
    &= 1 + xx^* - x^*x + x(1-x^*x)^{1/2} + (1-x^*x)^{1/2}x^*\\
    &= 1 - (x^*x - xx^* - x(1-x^*x)^{1/2} - (1-x^*x)^{1/2}x^*).
\end{align*}
So $p \coloneq x^*x - xx^* - x(1-x^*x)^{1/2} - (1-x^*x)^{1/2}x^* = 1 - vv^*$ is a
projection in $C^*_r(\mathcal{G})$. We have $\Phi(p) = x^*x - xx^*$ which is nonzero, and
so $p$ is nonzero. Also, since $r(\operatorname{supp}(x)) \subseteq
s(\operatorname{supp}(x)) \subseteq V \subseteq \{z \in \mathcal{G}^{(0)} \mid b_0(z) >
3/4\}$, we see that $p$ belongs to the hereditary subalgebra generated by $b$.

To see that $p$ is an infinite projection, argue exactly as above, but with $V$ replaced
by a nonempty open subset of $\operatorname{supp}(x^*x) \setminus
\operatorname{supp}(xx^*)$ to obtain a scaling element $y$ in $p C^*(\mathcal{G}) p$.
Then the calculations we performed above for $v$ show that $w \coloneq y + (p -
y^*y)^{1/2}$ is a partial isometry with $w^*w = p$ and $ww^* < p$. So $p$ is infinite as
required.
\end{proof}

When $\mathcal{G}$ is also minimal, we can verify that $\mathcal{G}$ is locally
contracting by verifying it at any one unit $x$.

\begin{lemma}
Let $\mathcal{G}$ be an \'etale groupoid. If $\mathcal{G}$ is minimal, then $\mathcal{G}$
is locally contracting at some point $x$ if and only if $\mathcal{G}$ is locally
contracting.
\end{lemma}
\begin{proof}
The ``if'' implication is trivial.

For the ``only if'', suppose that $\mathcal{G}$ is locally contracting at $x$ and fix $y
\in \mathcal{G}^{(0)}$. Fix an open neighbourhood $V$ of $y$. Since $\mathcal{G}$ is
minimal, there is an open bisection $B$ such that $r(B) \subseteq V$ and $x \in s(B)$.
Since $\mathcal{G}$ is locally contracting at $x$, there is an open $W$ containing $x$
with $\overline{W} \subseteq s(B)$ and an open bisection $U$ such that $\overline{W}
\subseteq s(U)$ and $r(U\overline{W}) \subsetneq W$. Now $r(BW) = B W B^{-1}$ is an open
neighbourhood of $y$ with $\overline{r(BW)} \subseteq V$, and $B U B^{-1}$ is a bisection
satisfying
\[
r(B U B^{-1} \overline{r(BW)}) = r(B U B^{-1} B \overline{W} B^{-1}) = r(B U\overline{W} B^{-1}) = r(B r(U\overline{W}))
    \subsetneq r(BW).\qedhere
\]
\end{proof}

\begin{remark}
Brown, Clark and Sierakowski \cite{BCS} have proved that if $\mathcal{G}$ is an \'etale,
effective, minimal groupoid, then $C^*_r(\mathcal{G})$ is purely infinite if and only if
every element of $C_0(\mathcal{G}^{(0)})$ is infinite in $C^*_r(\mathcal{G})$.
\end{remark}

\chapter{Cartan pairs, and Dixmier--Douady theory for Fell Algebras}\label{ch:Weyl gpd}

In this chapter we first discuss the beautiful reconstruction theorem of Renault
\cite{Ren08} that shows that an effective groupoid and twist can be recovered from the
associated twisted groupoid algebra. This builds on previous work of Kumjian
\cite{Kum86}, and develops ideas that go back to Feldman and Moore in the context of von
Neumann algebras \cite{FM1, FM2, FM3}. We will then discuss an application of this theory
to the classification of Fell algebras up to spectrum-preserving Morita equivalence
\cite{aHKS}.

\section{Kumjian--Renault theory}\label{sec:KR-theory}

The aim in this section is to outline Renault's construction for recovering an \'etale
groupoid from its reduced $C^*$-algebra together with the canonical abelian subalgebra
$C_0(\mathcal{G}^{(0)})$. This is a $C^*$-algebraic analogue of Feldman--Moore theory for
von Neumann algebras of Borel equivalence relations.

We will omit almost all of the proofs in this section. The details are due to Kumjian and
then Renault and can be found in \cite{Kum86, Ren08}.

To get the most out of this theory, we need to introduce twisted groupoid $C^*$-algebras.
In Renault's original work \cite{Ren80}, twisted groupoid $C^*$-algebras were determined
by continuous normalised $2$-cocycles on $\mathcal{G}$; that is, continuous maps $\sigma
: \mathcal{G}^{(2)} \to \mathbb{T}$ satisfying $\sigma(r(\gamma), \gamma) = 1 =
\sigma(\gamma, s(\gamma))$ for all $\gamma$ and satisfying the cocycle identity
$\sigma(\alpha,\beta)\sigma(\alpha\beta, \gamma) = \sigma(\beta,\gamma)\sigma(\alpha,
\beta\gamma)$ for every composable triple $(\alpha, \beta, \gamma)$.

The twisted convolution algebra is defined as $C_c(\mathcal{G}, \sigma) =
C_c(\mathcal{G})$ as a vector space, but with multiplication and involution given by $(f
* g)(\gamma) = \sum_{\alpha\beta = \gamma} \sigma(\alpha,\beta) f(\alpha)g(\beta)$, and
$f^*(\gamma) = \sigma(\gamma^{-1}, \gamma) \overline{f(\gamma^{-1})}$. However, Kumjian
subsequently observed that the notion of a twisted groupoid $C^*$-algebra that most
naturally leads to an analogue of Feldman--Moore theory comes from a \emph{twist}.

\begin{definition}
Let $\mathcal{G}$ be an \'etale groupoid. A \emph{twist} over $\mathcal{G}$ is a sequence
$\mathcal{G}^{(0)} \times \mathbb{T} \stackrel{i}{\longrightarrow} \mathcal{E}
\stackrel{\pi}{\longrightarrow} \mathcal{G}$, where $\mathcal{G}^{(0)} \times \mathbb{T}$
is regarded as a trivial group bundle with fibres $\mathbb{T}$, $\mathcal{E}$ is a
locally compact Hausdorff groupoid, and $i$ and $\pi$ are continuous groupoid
homomorphisms that restrict to homeomorphisms of unit spaces (we identify
$\mathcal{E}^{(0)}$ with $\mathcal{G}^{(0)}$ via $i$) such that
\begin{itemize}
\item  $i$ is injective,
\item $\mathcal{E}$ is a locally trivial $\mathcal{G}$-bundle in the sense that every
    point $\alpha \in \mathcal{G}$ has a bisection neighbourhood $U$ on which there
    exists a continuous section $S : U \to \mathcal{E}$ satisfying $\pi \circ S =
    \operatorname{id}_U$, and such that the map $(\alpha, z) \mapsto i(r(\alpha), z)
    S(\alpha)$ is a homeomorphism of $U \times \mathbb{T}$ onto $\pi^{-1}(U)$;
\item $i(\mathcal{G}^{(0)} \times \mathbb{T})$ is central in $\mathcal{E}$ in the
    sense that $i(r(\varepsilon),z) \varepsilon = \varepsilon i(s(\varepsilon), z)$
    for all $\varepsilon \in \mathcal{E}$ and $z \in \mathbb{T}$; and
\item $\pi^{-1}(\mathcal{G}^{(0)}) = i(\mathcal{G}^{(0)} \times \mathbb{T})$.
\end{itemize}
\end{definition}

If $\mathcal{G} = \Gamma$ is a discrete group, then a twist over $\mathcal{G}$ as defined
above is precisely a central extension of $\Gamma$.

\begin{notation}
If $\mathcal{E}$ is a twist over $\mathcal{G}$, $\varepsilon \in \mathcal{E}$ and $z \in
\mathbb{T}$, we will write $z \cdot \varepsilon \coloneq i(r(\varepsilon),
z)\varepsilon$, and $\varepsilon \cdot z = \varepsilon i(s(\varepsilon), z)$; so $z \cdot
\varepsilon = \varepsilon \cdot z$ because $i(\mathcal{G}^{(0)} \times \mathbb{T})$ is
central in $\mathcal{E}$.
\end{notation}

\begin{lemma}\label{lem:T-space}
If $\mathcal{E} \to \mathcal{G}$ is a twist, and $\varepsilon, \delta \in \mathcal{E}$
satisfy $\pi(\varepsilon) = \pi(\delta)$, then there is a unique $z \in \mathbb{T}$ such
that $z \cdot \varepsilon = \delta$.
\end{lemma}
\begin{proof}
We have $\pi(\varepsilon^{-1} \delta) = s(\delta) \in \mathcal{G}^{(0)}$, so
$\varepsilon^{-1}\delta = i(\{s(\delta)\} \times z)$ for some $z \in \mathbb{T}$; there
is just one such $z$ as $i$ is injective. We then have $(z\cdot \varepsilon)^{-1}\delta =
(\varepsilon^{-1}\delta) i(s(\delta), \overline{z}) = s(\delta)$. Multiplying on the
right by $z\cdot\varepsilon$ gives the result.
\end{proof}

\begin{example}
The cartesian-product groupoid $\mathcal{G} \times \mathbb{T}$ is a twist over
$\mathcal{G}$ in the obvious way. This is called the trivial twist over $\mathcal{G}$.
\end{example}

\begin{example}\label{eg:cocycle from section}
More generally, if $\sigma$ is a continuous normalised $2$-cocycle on $\mathcal{G}$, then
$\mathcal{G} \times \mathbb{T}$ can be made into a groupoid $\mathcal{E}_\sigma$ with the
usual unit space and range and source maps, but with multiplication and inversion given
by $(\alpha,w)(\beta,z) = (\alpha\beta, \sigma(\alpha\beta)wz)$ and $(\alpha, w)^{-1} =
(\alpha, \overline{\sigma(\alpha^{-1}, \alpha) w})$. Since $\sigma(r(\gamma), \gamma)) =
1 = \sigma(\gamma, s(\gamma))$ for all $\gamma$, the set inclusion $\mathcal{G}^{(0)}
\times \mathbb{T} \hookrightarrow \mathcal{E}_\sigma$ is a groupoid homomorphism, as is
the projection $\pi : \mathcal{E}_\sigma \to \mathcal{G}$ given by $\pi(\gamma, z) =
\gamma$. It is routine to check that $\mathcal{E}_\sigma$ is then a twist over $\Gamma$
with respect to $i$ and $\pi$.
\end{example}

\begin{remark}
We can recover the cohomology class of $\sigma$ from the twist $\mathcal{E}_\sigma \to
\mathcal{G}$ as follows: choose any continuous section $S$ for $\sigma$. For
$(\alpha,\beta) \in \mathcal{G}^{(2)}$, we have
$\sigma\big(S(\alpha)S(\beta)S((\alpha\beta)^{-1}\big) = r(\alpha) \in
\mathcal{G}^{(0)}$, and so Lemma~\ref{lem:T-space} shows that there is a unique element
$\omega(\alpha,\beta) \in \mathbb{T}$ such that $S(\alpha)S(\beta)S(\alpha\beta)^{-1} =
(r(\alpha), \omega(\alpha,\beta))$. The map $\omega$ defined in this way is a continuous
2-cocycle. If $S'$ is another continuous section for $\sigma$, then $\omega^{-1}\omega'$
is equal to the 2-coboundary obtained from the 1-cochain $b$ determined by
$S(\alpha)^{-1}S'(\alpha) = (r(\alpha), b(\alpha))$. Thus the cocycles obtained from
distinct choices of $S$ are cohomologous. Taking $S(\gamma) = (\gamma, 1)$ for all
$\gamma$, yields $\omega = \sigma$, so the cohomology class $[\sigma]$ of $\sigma$ is
equal to that of any cocycle obtained from a continuous section $S : \mathcal{G} \to
\mathcal{E}_\sigma$.

More generally, if $\mathcal{E}$ is a twist over $\mathcal{G}$ that admits a continuous
section $S : \mathcal{G} \to \mathcal{E}$ for the quotient map, then there is a
$2$-cocycle $\sigma$ on $\mathcal{G}$ defined by $S(\alpha)S(\beta)S(\alpha\beta)^{-1} =
i(s(\alpha), \sigma(\alpha,\beta))$. There is then an isomorphism $\mathcal{E} \cong
\mathcal{E}_\sigma$ that is equivariant for $i$ and $q$. So $\mathcal{E}$ is isomorphic
to a twist coming from a cocycle. But it is not clear that every $\mathcal{E}$ admits a
continuous section, so the notion of a twist is formally more general than that of a
continuous 2-cocycle.
\end{remark}

\begin{definition}
If $\mathcal{E}$ is a twist over the \'etale groupoid $\mathcal{G}$, then we write
\[
\Sigma_c(\mathcal{G}; \mathcal{E}) \coloneq \{f \in C_c(\mathcal{E}) \mid f(z\cdot \varepsilon) = zf(\varepsilon)\text{ for all }\varepsilon \in \mathcal{E}\text{ and }z \in \mathbb{T}\}.
\]
\end{definition}

\begin{remark}\label{rmk:line bundle}
Each twist $\mathcal{E}$ over $\mathcal{G}$ determines a complex line bundle
$\widetilde{\mathcal{E}}$ over $\mathcal{G}$ as follows: Define an equivalence relation
$\sim$ on $\mathcal{E} \times \mathbb{C}$ by $(\delta, w) \sim (\varepsilon, z)$ if
$\pi(\delta) = \pi(\varepsilon)$, $|w| = |z|$ and either $w = z = 0$ or $(w/|w|) \cdot
\delta = (z/|z|)\cdot \varepsilon$. Then $\widetilde{\mathcal{E}} := \mathcal{E}/{\sim}$
is a line-bundle over $\mathcal{G}$ with respect to the fibre map $p :
\widetilde{\mathcal{E}} \to \mathcal{G}$ given by $p([\delta,w]) = \pi(\delta)$.
\end{remark}

For $\gamma \in \mathcal{G}$, any choice of $\delta \in \pi^{-1}(\gamma)$ determines a
homeomorphism $\mathbb{T} \cong \pi^{-1}(\gamma) \subseteq \mathcal{E}$ given by $z
\mapsto z\cdot \delta$. Since Haar measure on $\mathbb{T}$ is rotation invariant, the
measure on $\pi^{-1}(\gamma)$ obtained by pulling back Haar measure on $\mathbb{T}$ is
independent of our choice of $\delta \in \pi^{-1}(\gamma)$. We endow each $\mathcal{E}^x$
with the measure $\lambda^x$ that agrees with this pulled back copy of Haar measure on
$\pi^{-1}(\gamma)$ for each $\gamma \in \mathcal{G}^x$ (so each $\pi^{-1}(\gamma)$ has
measure 1).

\begin{lemma}
The space $\Sigma_c(\mathcal{G}; \mathcal{E})$ is a $^*$-algebra under the operations
\[
f * g(\varepsilon) = \int_{\mathcal{E}^{r(\varepsilon)}} f(\delta) g(\delta^{-1}\varepsilon) \,d\lambda^{r(\varepsilon)}
    \text{ and }
f^*(\varepsilon) = \overline{f(\varepsilon^{-1})}.
\]
For any $\varepsilon \in \mathcal{G}$, $f,g \in \Sigma_c(\mathcal{G}; \mathcal{E})$ and
any choice of (not necessarily continuous) section $\alpha \mapsto \tilde{\alpha}$ for
$\pi|_{\mathcal{G}^{r(\epsilon)}}$, we have
\begin{equation}\label{eq:representative}
f * g(\varepsilon) = \sum_{\beta \in \mathcal{G}^{r(\varepsilon)}} f(\tilde\beta) g(\tilde\beta^{-1}\varepsilon).
\end{equation}
There is an isomorphism
\[
C_c(\mathcal{G}^{(0)}) \cong D_0 \coloneq \{f \in \Sigma_c(\mathcal{G}; \mathcal{E}) \mid \operatorname{supp}(f) \subseteq i(\mathcal{G}^{(0)} \times \mathbb{T})\}
\]
that carries $f \in C_c(\mathcal{G}^{(0)})$ to the function $\tilde{f} : i(x, z) \mapsto
zf(x)$.
\end{lemma}
\begin{proof}
We verify~\eqref{eq:representative}: if $\beta \in \mathcal{G}^{r(\varepsilon)}$ and
$\delta, \delta' \in \pi^{-1}(\beta)$, then $\delta' = z\delta$ for some $z \in
\mathbb{T}$, and hence
\[
f(\delta)g(\delta^{-1}\varepsilon)
    = f(\bar{z}\cdot \delta')g(z \cdot (\delta')^{-1}\varepsilon)
    = \bar{z} f(\delta') z g((\delta')^{-1}\varepsilon)
    = f(\delta') g((\delta')^{-1}\varepsilon).
\]
So each $\int_{\delta \in \pi^{-1}(\beta)} f(\delta)g(\delta^{-1}\varepsilon)
d\lambda^{r(\beta)}(\delta)$ collapses to $f(\tilde{\beta})
g(\tilde{\beta}^{-1}\varepsilon)$.

From here, that $\Sigma_c(\mathcal{G}; \mathcal{E})$ is a $^*$-algebra follows from
calculations similar to the ones that show that $C_c(\mathcal{G})$ is a $^*$-algebra.

Since $x \mapsto i(x,1)$ is a section for $\pi$ on $\mathcal{G}^{(0)}$, the final
assertion follows from~\eqref{eq:representative}.
\end{proof}

\begin{remark}
Kumjian points out that there is an isomorphism of $\Sigma_c(\mathcal{G}; \mathcal{E})$
with the space of compactly supported continuous sections of the complex line bundle
$\widetilde{\mathcal{E}}$ over $\mathcal{G}$ described in Remark~\ref{rmk:line bundle}.
This isomorphism carries $f \in \Sigma_c(\mathcal{G}; \mathcal{E})$ to the section
$\tilde{f}$ given by $\tilde{f}(\alpha) = [\tilde{\alpha}, f(\tilde{\alpha})]$ for any
choice of $\tilde{\alpha}$ in $\pi^{-1}(\alpha)$.
\end{remark}

We define the regular representations $\pi_x$, $x \in \mathcal{G}^{(0)}$ of
$C_c(\mathcal{G}; \mathcal{E})$ on the spaces $L^2(\mathcal{G}_x; \mathcal{E}_x)$ of
square-integrable $\mathbb{T}$-equivariant functions on $\mathcal{E}_x$ by extension of
the convolution formula. We define $C^*_r(\mathcal{G}; \mathcal{E})$ to be the completion
of the (injective) image of $\Sigma_c(\mathcal{G}; \mathcal{E})$ in the direct sum of
these representations, and $\|\cdot\|_r$ the $C^*$-norm in this $C^*$-algebra. Arguments
very similar to the ones for untwisted algebras give the following:

\begin{theorem}\label{thm:twisted csa}
For any $f \in \Sigma_c(\mathcal{G}; \mathcal{E})$, the set
\[
\big\{\|\pi(f)\| \mathbin{\big|} \pi\text{ is a $^*$-representation of }\Sigma_c(\mathcal{G}; \mathcal{E})\big\}
\]
is bounded above. Taking the supremum gives a pre-$C^*$-norm $\|\cdot\|$ on
$\Sigma_c(\mathcal{G}; \mathcal{E})$, and we define $C^*(\mathcal{G}; \mathcal{E})$ to be
the completion in this norm. We have $\|\cdot\|_\infty \le \|\cdot\|_r \le \|\cdot\|$ on
$\Sigma_c(\mathcal{G}, \mathcal{E})$, with equality on functions supported on
$\pi^{-1}(U)$ for any bisection $U$. If $\mathcal{G}$ is amenable, then $\|\cdot\|_r$ and
$\|\cdot\|$ agree. The map $f \mapsto f|_{\mathcal{E}^{(0)}}$ from $D_0$ to
$C_c(\mathcal{G}^{(0)})$ extends to an isomorphism of the completion of $D_0$, in either
norm, with $C_0(\mathcal{G}^{(0)})$.
\end{theorem}

We will write $D$ for the completion in $C^*(\mathcal{E}; \mathcal{G})$ and $D_r$ for the
completion in $C^*_r(\mathcal{G})$.

\begin{remark}
For the trivial twist $\mathcal{G} \times \mathbb{T}$, the map $\gamma \mapsto
(\gamma,1)$ is a continuous section for $\pi : \mathcal{G} \times \mathbb{T} \to
\mathcal{G}$. The cocycle obtained from this section as in Example~\ref{eg:cocycle from
section} is the trivial one. So we can use the formula~\eqref{eq:representative} to see
that $C^*_r(\mathcal{G}; \mathcal{G} \times \mathbb{T}) \cong C^*_r(\mathcal{G})$ in the
canonical way.
\end{remark}

In this section, we are interested in $C^*_r(\mathcal{G}; \mathcal{E})$ and the
subalgebra $D_r$.

\begin{proposition}\label{prp:twisted FCE}
Let $\mathcal{G}$ be an effective \'etale groupoid and $\mathcal{E}$ a twist over
$\mathcal{G}$. There is a faithful conditional expectation $\Phi : C^*_r(\mathcal{G};
\mathcal{E}) \to D_r$ that extends restriction of functions in $\Sigma_c(\mathcal{G};
\mathcal{E})$ to $i(\mathcal{G}^{(0)} \times \mathbb{T})$. This is the only conditional
expectation from $C^*_r(\mathcal{G}; \mathcal{E})$ to $D_r$.
\end{proposition}
\begin{proof}[Proof sketch]
The proof of existence follows the outline of Proposition~\ref{prp:FCE}. To see that
$\Phi$ is the unique conditional expectation onto $D_r$, first observe that the
expectation property says that if $\Psi : C^*_r(\mathcal{G}; \mathcal{E}) \to D_r$ is a
conditional expectation, then for any $a \in C^*_r(\mathcal{G}; \mathcal{E})$ and any $b
\in D_r$, we have
\begin{equation}\label{eq:commute}
\Psi(ab) = \Psi(a)b = b \Psi(a) = \Psi(ba).
\end{equation}
Arguing as in Lemma~\ref{lem:cutdown}, we show that for each $a \in
\Sigma_c(\mathcal{G})$ whose support does not intersect $i(\mathcal{G}^{(0)} \times
\mathbb{T})$, and each unit $x \in \mathcal{G}^{(0)}$, we can find an element $h \in D_r$
such that $h a h = 0$ and $h(x)
> 0$. Using that $s(\operatorname{supp}(a))$ is compact and a partition-of-unity argument, we find
finitely many $h_i$ such that $a \sum_i h^2_i = a$ and $h_i a h_i = 0$ for all $i$. This
gives $\Psi(a) = \Psi(a \sum h^2_i) = \sum \Psi(h_i a h_i) = 0$. So $\Psi$ agrees with
$\Phi$ on the space $D^\perp_0$ of elements of $\Sigma_c(\mathcal{G}; \mathcal{E})$ whose
support does not intersect $i(\mathcal{G}^{(0)} \times \mathbb{T})$, and it agrees with
$\Phi$ on $D_0$ because every conditional expectation is the identity map on its range.
Since $\Sigma_c(\mathcal{G}; \mathcal{E}) = D_0 + D_0^\perp$, we deduce that $\Phi$ and
$\Psi$ agree on all of $\Sigma_c(\mathcal{G}; \mathcal{E})$, and so are equal.
\end{proof}

As in the untwisted case, each element of $C^*_r(\mathcal{G}; \mathcal{E})$ determines a
$\mathbb{T}$-equivariant function $j \in C_0(\mathcal{E})$. One way to see this is to fix
a section (we do not require continuity) $\gamma \mapsto \tilde\gamma$ for the map $\pi :
\mathcal{E} \to \mathcal{G}$, so that $\pi(\tilde{\gamma}) = \gamma$ for all $\gamma$. If
$\delta,\varepsilon \in \mathcal{E}$ satisfy $\pi(\delta) = \pi(\varepsilon)$, then there
is a unique $[\delta,\varepsilon] \in \mathbb{T}$ such that $\delta =
[\delta,\varepsilon]\cdot \varepsilon$. In particular, if $(\alpha,\beta) \in
\mathcal{G}^{(2)}$, then $\pi(\tilde{\alpha}\tilde{\beta}) = \alpha\beta =
\pi(\widetilde{\alpha\beta})$. It is not too hard to see that for each $x \in
\mathcal{G}^{(0)}$ there is a representation of $\Sigma_c(\mathcal{G}; \mathcal{E})$ on
$\ell^2(\mathcal{G}_x)$ satisfying
\[
\tilde\pi_x(f) \delta_\beta
    = \sum_{\alpha \in \mathcal{G}_{r(\beta)}} [\tilde{\alpha}\tilde{\beta}, \widetilde{\alpha\beta}] f(\tilde{\alpha}) \delta_{\alpha\beta},
\]
and this representation is unitarily equivalent to the regular representation $\pi_x$ of
$\Sigma_c(\mathcal{G}; \mathcal{E})$. With this representation in hand, the argument of
Proposition~\ref{prp:j map} carries across to the twisted setting.

If $A$ is a $C^*$-algebra and $B$ is a subalgebra of $A$, we shall say that $n \in A$ is
a \emph{normaliser} of $B$ if $nBn^* \cup n^*Bn \subseteq B$. We write $N(B)$ for the
collection of all normalisers of $B$. We say that $B$ is \emph{regular} in $A$ if $A$ is
generated as a $C^*$-algebra by $N(B)$.

\begin{proposition}\label{prp:D Cartan}
If $\mathcal{G}$ is an \'etale, effective groupoid, and $\mathcal{E}$ is a twist over
$\mathcal{G}$, then $D_r$ is a regular maximal abelian subalgebra of $C^*_r(\mathcal{G};
\mathcal{E})$ that contains an approximate unit for $C^*(\mathcal{G}; \mathcal{E})$.
\end{proposition}
\begin{proof}[Proof sketch]
Clearly $D_r$ is an abelian algebra. To see that it is maximal abelian, suppose that $a$
belongs to its complement. Then $j(a)$ must be nonzero at some $\varepsilon \in
\mathcal{E} \setminus \pi^{-1}(\operatorname{Iso}(\mathcal{G})) =
\operatorname{Iso}(\mathcal{E})$. So we can choose $h \in D_0 \subseteq D_r$ such that
$h(r(\varepsilon)) = 1$ and $h(s(\varepsilon)) = 0$. Now $j(ah)(\varepsilon) = 0$ whereas
$j(ha)(\varepsilon) =a(\varepsilon) \not= 0$. To see that $D_r$ is regular, we use the
multiplication formula to see that if $n \in C_c(\pi^{-1}(U))$ for some bisection $U$ of
$\mathcal{G}$, then for $h \in C_c(\mathcal{G}^{(0)})$ (regarded as an element of $D_0$
using the isomorphism of Theorem~\ref{thm:twisted csa}) we have $\operatorname{supp}(n h
n^*) \subseteq \operatorname{supp}(n) \operatorname{supp}(h) \operatorname{supp}(n)^{-1}
\subseteq \pi^{-1}(UU^{-1}) \subseteq \pi^{-1}(\mathcal{G}^{(0)})$, and similarly for
$n^* h n$.
\end{proof}

Following Renault \cite{Ren08}, we shall say that a pair $(A, B)$ of $C^*$-algebras is a
\emph{Cartan pair} and say that $B$ is a \emph{Cartan subalgebra} of $A$ if $A$ is a
$C^*$-algebra, $B$ is a $C^*$-subalgebra of $A$ containing an approximate unit for $A$,
there is a faithful conditional expectation of $A$ onto $B$, $B$ is a maximal abelian
subalgebra of $A$, and $B$ is regular in $A$. We can reinterpret the preceding result as
saying that $(C^*(\mathcal{G}; \mathcal{E}), D_r)$ is a Cartan pair. Our main objective
here is to prove that every Cartan pair has this form.

Given a Cartan pair $(A, B)$, and given $n \in N(B)$, we write
\[
\operatorname{dom}(n) \coloneq \overline{\{\phi \in \widehat{B} \mid \phi(n^*n) > 0\}}
    \quad\text{ and }\quad
\operatorname{ran}(n) \coloneq \overline{\{\phi \in \widehat{B} \mid \phi(nn^*) > 0\}}.
\]
We have $\operatorname{ran}(n) = \operatorname{dom}(n^*)$.

\begin{proposition}
Let $(A, B)$ be a Cartan pair of $C^*$-algebras. For each $n \in N(B)$, there is a
homeomorphism $\alpha_n : \operatorname{dom}(n) \to \operatorname{ran}(n)$ satisfying
$\alpha_n(\phi)(n b n^*) = \phi(b n^*n)$. There is an equivalence relation $\sim$ on
$\{(n,\phi) \mid n \in N(B), \phi \in \operatorname{dom}(n)\}$ such that $(n,\phi) =
(m,\psi)$ if and only if $\phi = \psi$ and there is a neighbourhood $U$ of $\phi$ such
that $\alpha_n|_U = \alpha_m|_U$. The set
\[
    \mathcal{G}_{(A,B)} \coloneq \{[n,\phi] \mid n \in N(B), \phi \in \operatorname{dom}(n)\}
\]
is a groupoid with unit space
\[
\mathcal{G}^{(0)}_{(A,B)} = \{[b,\phi] \mid b \in B, \phi \in \operatorname{supp}(B)\},
\]
and groupoid structure given by
\begin{gather*}
r([n,\phi]) = [nn^*, \alpha_n(\phi)],\qquad s([n,\phi)] = [n^*n, \phi],\\
[m, \alpha_n(\phi)][n,\phi] = [mn,\phi]\quad\text{ and }\quad [n,\phi]^{-1} = [n^*,
\alpha_n(\phi)].
\end{gather*}
\end{proposition}
\begin{proof}[Proof sketch]
Take the polar decomposition $n = v|n|$ of $n$ in $A^{**}$ and observe that if $f$
belongs to the dense subalgebra $C_c(\operatorname{dom}(n)^\circ)$ of $I_{n^*n} \coloneq
C_0(\operatorname{dom}(n)^\circ)$, then $|n|$ is invertible on $\operatorname{supp}(f)$,
and so we can write $f = |n| g |n|^*$ for some $g \in I_{n^*n}$. Hence $vfv^* = v |n| g
|n|^* v^* = a g a^* \in B$. Applying the same reasoning to $n^*$ we see that conjugation
by $v$ determines an isomorphism of commutative $C^*$-algebras $I_{n^*n} \cong I_{nn^*}$,
and therefore induces a homeomorphism $\alpha_n$ between their spectra. It is clear that
$\sim$ is an equivalence relation, and that $r,s$ are well defined. To check that the
multiplication is well defined, we use the observation that if $v$, $w$ are the partial
isometries appearing in the polar decompositions of $m,n$, and if $\operatorname{dom}(m)
= \operatorname{ran}(n)$, then $vw$ is the partial isometry appearing in the polar
decomposition of $mn$; and also that for $h \in C_c(\operatorname{supp}(n))$ and $\phi$
satisfying $\phi(h) \not= 0$, we have $[n, \phi] = [nh,\phi] = [(h\circ\alpha_{n}^{-1})
n, \phi]$. It is easy to check that inversion is well-defined using that if $n = v|n|$,
then $n^* = v^*|n^*|$. Associativity of multiplication comes from associativity of
multiplication in $A$, and the inverse property follows directly from the definitions of
$r$ and $s$.
\end{proof}

\begin{theorem}\label{thm:Weyl gpd}
The groupoid $\mathcal{G}_{(A, B)}$ becomes an \'etale groupoid under the topology with
basic open sets $Z(n,U) \coloneq \{[n,\phi] \mid \phi \in U\}$ indexed by elements $n \in
N(B)$ and open sets $U \subseteq \operatorname{dom}(n)$. If $A = C^*(\mathcal{G};
\mathcal{E})$ and $B = D_0$ for some twist $\mathcal{G}^{(0)} \times \mathbb{T} \to
\mathcal{E} \to \mathcal{G}$, then there is an isomorphism $\theta : \mathcal{G} \cong
\mathcal{G}_{(A, B)}$ such that for any $\gamma \in \mathcal{G}$, any open bisection $U$
containing $\gamma$ whose closure is a compact bisection and any $n \in
\Sigma_c(\mathcal{G}; \mathcal{E})$ that is nonzero everywhere on $\pi^{-1}(U)$, we have
$\theta(\gamma) = [n, \widehat{s(\gamma)}]$.
\end{theorem}
\begin{proof}[Proof sketch]
It is fairly straightforward to check that the sets $Z(n, U)$ are a base for a locally
compact Hausdorff topology, and that $\mathcal{G}_{(A, B)}$ is \'etale in this topology.
If $A = C^*(\mathcal{G}; \mathcal{E})$ and $B = D_r$, then the defining property of
$\alpha_n$ (namely $\alpha_n(\phi)(n h n^*) = \phi(h n^*n)$) shows that if $n$ is nonzero
on all of $\pi^{-1}(U)$, then $\alpha_n(\widehat{s(\gamma)}) = \widehat{r(\gamma)}$ for
$\gamma \in U$. It follows from direct computations that $\theta$ is an algebraic
isomorphism. The pre-image of a given $Z(n, U)$ under $\theta$ is the open bisection
$\operatorname{supp}(n) U$, so $\theta$ is continuous. Moreover, for a given $\gamma \in
\mathcal{G}$ we can choose a compact bisection $K$ containing $\gamma$, an open bisection
$U$ containing $K$ and an element $n \in C_c(U)$ that is identically $1$ on $K$, and
$\theta$ is then a continuous bijection of the compact set $K$ onto $Z(n, K)$, and hence
a homeomorphism between these sets. Since the interiors of compact bisections form a base
for the topology on $\mathcal{G}$, it follows that $\theta$ is open.
\end{proof}

As in \cite{Ren08}, we call $\mathcal{G}_{(A,B)}$ the \emph{Weyl groupoid} of the Cartan
pair $(A,B)$.

\begin{proposition}
Let $(A, B)$ be a Cartan pair of $C^*$-algebras. There is an equivalence relation
$\approx$ on the set $\{(n, \phi) \in N(B) \times \widehat{B} \mid \psi(n^*n) > 0\}$ such
that $(m, \phi) \approx (n, \phi)$ if $\phi = \psi$, and there exist $b, b' \in B$ with
$\phi(b), \phi(b') > 0$ and $mb = nb'$. The set
\[
    \mathcal{E}_{(A,B)} \coloneq \{(n, \phi) \in N(B) \times \widehat{B} \mid \psi(n^*n) > 0\}/\approx
\]
is a groupoid with unit space $\{\llbracket h, \phi\rrbracket \mid h \in
C_0(\mathcal{G}^{(0)}), \phi \in \operatorname{dom}(h)\}$, range and source maps
$r(\llbracket n,\phi\rrbracket) = \llbracket nn^*, \phi\rrbracket$, $s(\llbracket
n,\phi\rrbracket) = \llbracket n^*n, \phi\rrbracket$, multiplication $\llbracket m,
\alpha_n(\phi)\rrbracket \llbracket n, \phi\rrbracket = \llbracket mn, \phi\rrbracket$,
and inversion $\llbracket n, \phi\rrbracket^{-1} = \llbracket n^*,
\alpha_n(\phi)\rrbracket$.
\end{proposition}
\begin{proof}[Proof sketch]
This is relatively straightforward; the only potential sticking point is well-definedness
of multiplication, but this follows from the fact that if $(n,\phi) = (n',\phi)$ and $(m,
\alpha_n(\phi)) = (m',\alpha_n(\phi))$, say $nc = n'c'$ and $mb = m'b'$, then we can
assume (by multiplying by some $h > 0$ supported on $\alpha_n(\operatorname{supp}(c) \cap
\operatorname{supp}(c'))$) that $\operatorname{supp}(b) = \operatorname{supp}(b')
\subseteq \operatorname{ran}(n)$, and then note that $b n = n (b\circ\alpha_n)$ and $b'n'
= n' (b'\circ\alpha_{n'})$, so that $mn (b\circ\alpha_n c) = m'n'
(b'\circ\alpha_{n'}c')$.
\end{proof}

\begin{proposition}
Let $(A, B)$ be a Cartan pair. There is a locally compact Hausdorff topology on
$\mathcal{E}_{(A,B)}$ with basic open sets $\widetilde{Z}(n, U) = \{\llbracket n,
\phi\rrbracket \mid \phi \in U\}$ indexed by $n \in N(B)$ and open sets $U \subseteq
\operatorname{dom}(n)$. The groupoid $\mathcal{E}_{(A, B)}$ is a topological groupoid in
this topology. There is an injective continuous groupoid homomorphism $i_{(A, B)} :
\widehat{B} \times \mathbb{T} \to \mathcal{E}_{(A, B)}$ given by $i_{(A, B)}(\phi, z) =
\llbracket b, \phi\rrbracket$ for any $b \in B$ such that $\phi(b) = z$, and there is
continuous surjective groupoid homomorphism $\pi_{(A,B)} : \mathcal{E}_{(A,B)} \to
\mathcal{G}_{(A,B)}$ such that $\pi(\llbracket n, \phi\rrbracket) = [n,\phi]$. The
sequence $\widehat{B} \times \mathbb{T} \to \mathcal{E}_{(A, B)} \to \mathcal{G}_{(A,
B)}$ is a twist over $\mathcal{G}_{(A, B)}$.
\end{proposition}
\begin{proof}[Proof sketch]
This is largely a matter of straightforward checking of details. The key point is that if
$(m,\phi) \approx (n,\phi)$, say $mb = nb'$ where $b, b'\in B$, then by multiplying
$b,b'$ by some positive $h$ with $\phi(h) = 1$, we can assume that
$\operatorname{supp}(b) = \operatorname{supp}(b') = U$, say, and then $\alpha_m|_U =
\alpha_{mb}|_U = \alpha_{nb'}|_U = \alpha_n|_U$. This shows that the map $\llbracket
n,\phi\rrbracket \to [n,\phi]$ makes sense and is well defined. To see why $i_{(A,B)}$ is
injective, and its image is precisely the kernel of $\pi_{(A, B)}$, observe that for any
$b,c \in C_0(\mathcal{G}^{(0)})$ with $\phi(b), \phi(b') \not= 0$, we have $[b,\phi] =
[c,\phi]$, but we have $\llbracket b, \phi\rrbracket = \llbracket c, \phi\rrbracket$ if
and only if $\phi(b)/|\phi(b)| = \phi(c)/|\phi(c)|$: for if so, then $z =
\phi(b)/|\phi(b)|$ satisfies $\phi(\bar{z} b), \phi(\bar{z}c) > 0$ and $b (\bar{z} c) =
c(\bar{z} b)$.
\end{proof}

The following theorem, which is the desired Feldman--Moore-type theorem, in our setting,
is due to Renault. For $\mathcal{G}$ principal, it was first proved by Kumjian
\cite{Kum86}.

\begin{theorem}[{Renault, \cite[Theorem~5.9]{Ren08}}]\label{thm:Renault-Feldman-Moore}
Suppose that $\mathcal{G}$ is an effective \'etale groupoid. Suppose that $\mathcal{E}$
is a twist over $\mathcal{G}$. Then there is an isomorphism $\zeta :
\mathcal{E}_{C^*_r(\mathcal{G}; \mathcal{E}), D_r} \to \mathcal{E}$ such that the diagram
\[
\begin{tikzpicture}[>=latex]
    \node (tl) at (0,1.5) {$\widehat{D_r} \times \mathbb{T}$};
    \node (tm) at (5,1.5) {$\mathcal{E}_{C^*_r(\mathcal{G}; \mathcal{E}), D_r}$};
    \node (tr) at (10,1.5) {$\mathcal{G}_{C^*_r(\mathcal{G}; \mathcal{E}), D_r}$};
    \node (bl) at (0,0) {$\mathcal{G}^{(0)} \times \mathbb{T}$};
    \node (bm) at (5,0) {$\mathcal{E}$};
    \node (br) at (10,0) {$\mathcal{G}$};
    \draw[->] (tl) to node[pos=0.5, above] {$\scriptstyle i_{(C^*_r(\mathcal{G}; \mathcal{E}), D_r)}$} (tm);
    \draw[->] (tm) to node[pos=0.5, above] {$\scriptstyle \pi_{(C^*_r(\mathcal{G}; \mathcal{E}), D_r)}$} (tr);
    \draw[->] (bl) to node[pos=0.5, above] {$\scriptstyle i$} (bm);
    \draw[->] (bm) to node[pos=0.5, above] {$\scriptstyle \pi$} (br);
    \draw[->] (tl) to node[pos=0.5, right] {$\scriptstyle \cong$} (bl);
    \draw[->] (tm) to node[pos=0.5, right] {$\scriptstyle \zeta$} (bm);
    \draw[->] (tr) to node[pos=0.5, right] {$\scriptstyle \theta$} (br);
\end{tikzpicture}
\]
commutes. In particular, the map
\[
\big(\mathcal{G}^{(0)} \times \mathbb{T} \to \mathcal{E} \to \mathcal{G}\big) \longmapsto \big(C^*_r(\mathcal{G}; \mathcal{E}), D_r\big)
\]
induces a bijection between isomorphism classes of twists and isomorphism classes of
Cartan pairs of $C^*$-algebras.
\end{theorem}

\begin{corollary}\label{cor:unique FCE}
If $(A, B)$ is a Cartan pair, then there is only one conditional expectation of $A$ onto
$B$.
\end{corollary}
\begin{proof}
We proved in Proposition~\ref{prp:twisted FCE} that the expectation extending restriction
of functions is the unique expectation from $C^*_r(\mathcal{G}; \mathcal{E})$ to $D_r$.
\end{proof}

In \cite{BL}, Barlak and Li showed how to extend Tu's result (Theorem~\ref{thm:Tu UCT})
that the $C^*$-algebras of amenable groupoids always belong to the UCT class to twisted
groupoids. Building on this, work of Takeishi and van Erp--Williams on nuclearity of
$C^*$-algebras of groupoid Fell bundles, and combining this with
Theorem~\ref{thm:Renault-Feldman-Moore}, Barlak and Li made substantial progress on ``UCT
question": does every nuclear $C^*$-algebra belong to the UCT class?

\begin{theorem}[{Barlak--Li \cite[Theorem~1.1 and Corollary~1.2]{BL}}]
If $\mathcal{G}$ is an \'etale groupoid and $\mathcal{E}$ is a twist over $\mathcal{G}$,
and if $C^*_r(\mathcal{G}; \mathcal{E})$ is nuclear, then $C^*_r(\mathcal{G};
\mathcal{E})$ belongs to the UCT class. In particular, if $A$ is a separable nuclear
$C^*$-algebra that is Morita equivalent to a $C^*$-algebra with a Cartan subalgebra, then
it is in the UCT class.
\end{theorem}

Another interesting application of Theorem~\ref{thm:Renault-Feldman-Moore} is a beautiful
theorem of Matsumoto and Matui \cite{MM}. Recall that the Cuntz--Krieger algebra of an
irreducible $\{0,1\}$ matrix $A$ (that is not a permutation matrix) agrees with the
$C^*$-algebra of the graph whose adjacency matrix is $A$ \cite{EW, KPRR}, and therefore
with the groupoid $C^*$-algebra associated to the graph groupoid described in
Example~\ref{ex:graph gpd}. A fundamental result of Franks \cite{Franks}, building on his
previous work with Bowen \cite{BF} and on work of Parry and Sullivan \cite{PS}, says that
the (two-sided) shift spaces of $\{0,1\}$-matrices $A$ and $B$ are flow equivalent if and
only if $\operatorname{coker}(1 - A^t) \cong \operatorname{coker}(1 - B^t)$ and, if
nonzero, $\det(1 - A^t)$ and $\det(1 - B^t)$ have the same sign (positive or negative).
Building on previous work of Cuntz \cite{Cuntz}, R{\o}rdam proved \cite{Rordam} that the
Cuntz--Krieger algebras $\mathcal{O}_A$ and $\mathcal{O}_B$ are stably isomorphic if and
only if $\operatorname{coker}(1 - A^t) = \operatorname{coker}(1 - B^t)$. So the
Cuntz--Krieger algebra ``forgets" some information about flow-equivalence. But, using
Renault's results and previous work of Matui, Matsumoto and Matui proved
part~(\ref{it:MM1}) and the equivalence of (\ref{it:MM2a})~and~(\ref{it:MM2b}) in the
following theorem in 2013. The final two equivalences, proved later, require the groupoid
equivalence theorem of \cite{CRS}, and the results of \cite{BCW}.

Recall that if $E$ is a directed graph, then the Tomforde stabilisation graph $SE$ is
obtained by attaching an infinite head at each vertex of $E$ \cite{T}. The resulting
``stabilised" symbolic dynamical system can be described as follows: if $X_E$ is the
shift space of $E$ (namely $E^\infty$ endowed with the usual shift map), then $X_{SE}$ is
given as a space by $X_{SE} = X_E \times \{0, 1, 2, \dots\}$, and the dynamics
$\overline{\sigma}$ on $X_{SE}$ is given by
\[
\overline{\sigma}(x,n) =
    \begin{cases}
        (x, n-1) &\text{ if $n \ge 1$}\\
        (\sigma(x), 0) &\text{ if $n = 1$.}
    \end{cases}
\]

We will call $(X_{SE}, \overline{\sigma})$ the \emph{stabilisation} of $(X_E, \sigma_E)$,
and say that $(X_E, \sigma_E)$ and $(X_F, \sigma_F)$ are \emph{stably orbit equivalent}
if their stabilisations are orbit equivalent.

\begin{theorem}
Let $A,B$ be irreducible $\{0,1\}$-matrices that are not permutation matrices.
\begin{enumerate}
\item\label{it:MM1} The following are equivalent:
    \begin{enumerate}
    \item The one-sided shift-spaces determined by $A$ and $B$ are continuously
        orbit equivalent;
    \item The Deaconu--Renault groupoids of the one-sided shift maps associated
        to $A$ and $B$ are isomorphic;
    \item There is an isomorphism $\mathcal{O}_A \cong \mathcal{O}_B$ that
        carries $C_0(\mathcal{G}_A^{(0)})$ to $C_0(\mathcal{G}_B^{(0)})$.
    \end{enumerate}
\item The following are equivalent
    \begin{enumerate}
    \item\label{it:MM2a} The two-sided shift-spaces determined by $A$ and $B$ are
        flow equivalent;
    \item\label{it:MM2b} The Cartan pairs $(\mathcal{O}_A \otimes \mathcal{K},
        C_0(\mathcal{G}_A^{(0)}) \otimes c_0)$ and $(\mathcal{O}_B \otimes
        \mathcal{K}, C_0(\mathcal{G}_B^{(0)}) \otimes c_0)$ are isomorphic as
        Cartan pairs;
    \item The groupoids $\mathcal{G}_A \times \mathcal{R}_\mathbb{N}$ and
        $\mathcal{G}_B \times \mathcal{R}_\mathbb{N}$ are isomorphic.
    \item The dynamical systems determined $(X_E, \sigma_E)$ and $(X_F,
        \sigma_F)$ are stably orbit equivalent.
\end{enumerate}
\end{enumerate}
\end{theorem}

\section{A Dixmier--Douady theorem for Fell
algebras}\label{sec:DD}

In what follows, given a $C^*$-algebra $A$, we shall write $\hat{A}$ for the space of
unitary equivalence classes of irreducible representations of $A$. We give $\hat{A}$ the
initial topology obtained from the quotient map from $\hat{A}$ to
$\operatorname{Prim}(A)$, where $\operatorname{Prim}(A)$ is given the Jacobson topology.

Recall that a $C^*$-algebra $A$ is \emph{liminary} or Type~I if every irreducible
representation $\pi : A \to \mathcal{B}(\mathcal{H})$ has image
$\mathcal{K}(\mathcal{H})$. A positive element $a$ of a liminary $C^*$-algebra is a
\emph{continuous trace element} if $\pi(a)$ has finite trace for every $\pi \in \hat{A}$
and the map $\pi \mapsto \operatorname{Tr}(\pi(a))$ is continuous. A
\emph{continuous-trace $C^*$-algebra} is a liminary $C^*$-algebra that is generated as an
ideal by its continuous-trace elements. The spectrum $\hat{A}$ of a continuous-trace
$C^*$-algebra is always Hausdorff. The Dixmier--Douady theorem \cite{DD} says that for a
given locally compact Hausdorff space $X$, the continuous-trace $C^*$-algebras with
spectrum $X$ are classified up to spectrum-preserving Morita equivalence by the
Dixmier--Douady invariant, which is an element of the cohomology group $H^3(X,
\mathbb{Z})$; moreover, the invariant is exhausted in the sense that each class in
$H^3(X, \mathbb{Z})$ occurs as the Dixmier--Douady invariant of some continuous-trace
algebra with spectrum $X$.

Raeburn and Taylor subsequently gave a very nice description of the continuous-trace
$C^*$-algebra with given Dixmier--Douady invariant $\delta \in H^3(X, \mathbb{Z})$ using
groupoids:

\begin{example}[Raeburn--Taylor \cite{RT}]\label{eg:RT construction}
Recall that a \v{C}ech $2$-cocycle on a locally compact Hausdorff space $X$ consists of a
cover of $X$ by open sets $U_i$ and a collection of continuous $\mathbb{T}$-valued
functions $c_{ijk}$ defined on triple-overlaps $U_{ijk} = U_i \cap U_j \cap U_k$ such
that $c_{iij}, c_{ijj} \equiv 1$ for each $i,j$, and such that on nonempty quadruple
overlaps $U_{ijkl}$ we have $c_{ijk}c_{ikl} = c_{ijl} c_{jkl}$. A coboundary is a cocycle
of the form $(\delta b)_{ijk}(x) \coloneq b_{ij}(x) b_{ik}(x)^{-1}b_{jk}(x)$ for some
collection of continuous functions $b_{ij} : U_{ij} \to \mathbb{T}$ defined on
double-overlaps. The \v{C}ech cohomology group is the quotient $\check{H}(X, \mathbb{T})$
of the group of 2-cocycles by the subgroup of 2-coboundaries. Given a \v{C}ech
$2$-cocycle on $X$, we can form an equivalence relation $R$ with unit space $\bigsqcup
\{i\} \times U_i \times \{i\}$ and with elements $\{(i, x, j) \mid x \in U_{ij}\}$, where
$r(i,x,j) = (i, x,i)$ and $s(i, x, j) = (j,x,j)$. We then construct a twist over $R$ by
putting $\mathcal{E} = R \times \mathbb{T}$ and defining multiplication on $\mathcal{E}$
by $\big((i,x,j), w\big)\big((j,x,k), z\big) = \big((i,x,k), c_{ijk}(x)wz\big)$. Note
that this twist comes from a continuous $2$-cocycle on $R$. Raeburn and Taylor \cite{RT}
proved that the $C^*$-algebra of this twist has Dixmier--Douady invariant equal to the
cohomology class of $c$.
\end{example}

In this section we will give a brief overview of how, using groupoids and the
construction of the preceding section, we can obtain a version of the Dixmier--Douady
theorem for Fell algebras based on the Raeburn--Taylor construction of the preceding
example. The details appear in \cite{aHKS}, though of course the ideas there owe a great
deal to the previous work of Dixmier--Douady \cite{DD}, Raeburn--Taylor \cite{RT}, and
the excellent monograph on Dixmier--Douady theory by Raeburn--Williams \cite{tfb}. The
details of the material in this section involve significant extra background and set-up,
so I will give almost no proofs, and just touch on the main points of the construction.

\begin{definition}
A $C^*$-algebra is called a \emph{Fell algebra} if it is liminary, and for every $[\pi]
\in \hat{A}$, there exists $b \in A^+$ and a neighbourhood $U$ of $[\pi]$ such that
$\psi(b)$ is a rank-1 projection whenever $[\psi] \in U$.
\end{definition}

Roughly speaking, this says that lots of elements of $A$ have the same rank under nearby
irreducible representations. So it should be related to the continuous-trace condition.
Indeed, it turns out that a Fell algebra is a continuous-trace algebra if and only if it
has Hausdorff spectrum. Theorem~3.3 of \cite{aHKS} says that $A$ is a Fell algebra if and
only if it is liminary and generated as an ideal by elements $a \in A_+$ such that $a A
a$ is abelian, and that this in turn happens if and only if there is a set $S$ of ideals
of $A$ each element of which is Morita equivalent to a commutative $C^*$-algebra and such
that $\bigcup S$ spans a dense subspace of $A$.

We now show how to obtain an equivalence relation from a Fell algebra.

\begin{proposition}\label{prp:construct diagonal}
Let $A$ be a Fell algebra. Choose a sequence $d_i$ of positive elements of $A$ with
$\|d_i\| = 1$ such that each $d_i A d_i$ is abelian, and such that the $A d_i A$ generate
$A$. For each $i$, let $a_i \coloneq d_i \otimes \theta_{i,i} \in A \otimes \mathcal{K}$.
Then $\sum_i a_i$ converges to an element $a$ of $\mathcal{M}(A \otimes \mathcal{K})$.
This element is full in the sense that $\overline{\operatorname{span}} (A \otimes
\mathcal{K})a(A \otimes \mathcal{K}) = A \otimes \mathcal{K}$. Moreover $C \coloneq a (A
\otimes \mathcal{K}) a$ and $D \coloneq \sum_i a_i (A \otimes \mathcal{K}) a_i \cong
\bigoplus_i d_i A d_i$ form a Cartan pair $(C, D)$.
\end{proposition}

By Corollary~\ref{cor:unique FCE}, if $(C, D)$ is a Cartan pair, then there is a unique
expectation $\Phi : A \to B$. Given $\phi \in \widehat{D}$, the composition $\phi \circ
\Phi$ gives a pure state of $C$, and then the GNS construction yields an irreducible
representation. So we obtain a well-defined map $\sigma : \widehat{D} \to \widehat{C}$,
which we call the \emph{spectral map}.

\begin{proposition}
Let $A$ be a Fell algebra, and choose a sequence $d_i$ as in the preceding proposition.
The Weyl groupoid $\mathcal{G}_{C, D}$ of Theorem~\ref{thm:Weyl gpd} is isomorphic to the
equivalence relation $R(\sigma)$ determined by the spectral map: $R(\sigma) =
\{(\phi,\psi) \in \widehat{D} \mid \sigma(\phi) = \sigma(\psi)\}$.
\end{proposition}

Using this, we are able to characterise diagonal-preserving Morita equivalence of Fell
algebras in terms of groupoid equivalence.

We shall say that twists $\mathcal{E}_1 \to \mathcal{G}_1$ and $\mathcal{E}_2 \to
\mathcal{G}_2$ are equivalent if there is a linking groupoid $L$ for $\mathcal{G}_1$ and
$\mathcal{G}_2$ and a twist $\mathcal{L}$ over $L$ such that reduction of $L$ and
$\mathcal{L}$ to $\mathcal{G}^{(0)}_i \subseteq L^{(0)}$ yields a twist
$\mathcal{G}^{(0)}_i \times \mathbb{T} \to \mathcal{G}^{(0)}_i \mathcal{L}
\mathcal{G}^{(0)}_i \to \mathcal{G}^{(0)}_i L \mathcal{G}^{(0)}_i$ that is isomorphic to
$\mathcal{G}^{(0)}_i \times \mathbb{T} \to \mathcal{E}_i \to \mathcal{G}_i$. This is the
natural extension of the notion of groupoid equivalence to twists.

\begin{proposition}
If $(C_1, D_1)$ and $(C_2, D_2)$ are Cartan pairs in which $C_1$ and $C_2$ are Fell
algebras, then $C_1$ and $C_2$ are Morita equivalent if and only if the twists
$\mathbb{T} \times \widehat{D}_i \to \mathcal{E}_{C_i, D_i} \to \mathcal{G}_{C_i, D_i}$
($i = 1,2$) are equivalent twists.
\end{proposition}

The Dixmier--Douady invariant of a continuous-trace $C^*$-algebra with spectrum $X$ is an
element of a cohomology group. For us, the collection of equivalence classes of twists
will act as a proxy for this cohomology group. (Theorem~\ref{thm:twists iso} describes
how these two groups are related in the continuous-trace setting.)

We say that twists $\mathcal{G}^{(0)} \times \mathbb{T} \to \mathcal{E} \to \mathcal{G}$
and $\mathcal{G}^{(0)} \times \mathbb{T} \to \mathcal{F} \to \mathcal{G}$ over the same
\'etale groupoid $\mathcal{G}$ are \emph{isomorphic} if there is a groupoid isomorphism
$\zeta : \mathcal{E} \to \mathcal{F}$ such that the diagram
\[
\begin{tikzpicture}[>=latex]
    \node (tl) at (0,1.5) {$\mathcal{G}^{(0)} \times \mathbb{T}$};
    \node (tm) at (3,1.5) {$\mathcal{E}$};
    \node (tr) at (6,1.5) {$\mathcal{G}$};
    \node (bl) at (0,0) {$\mathcal{G}^{(0)} \times \mathbb{T}$};
    \node (bm) at (3,0) {$\mathcal{F}$};
    \node (br) at (6,0) {$\mathcal{G}$};
    \draw[->] (tl) to (tm);
    \draw[->] (tm) to (tr);
    \draw[->] (bl) to (bm);
    \draw[->] (bm) to (br);
    \draw[->] (tl) to node[pos=0.5, right] {$\scriptstyle=$} (bl);
    \draw[->] (tm) to node[pos=0.5, right] {$\scriptstyle \zeta$} (bm);
    \draw[->] (tr) to node[pos=0.5, right] {$\scriptstyle=$} (br);
\end{tikzpicture}
\]
commutes.

For the following result, we need to describe the pullback construction for twists over a
given relation $R$. Let $R$ be an equivalence relation, and suppose that $\mathcal{E}$
and $\mathcal{E}'$ are twists over $R$. Define an equivalence relation $\sim$ on
\[
\mathcal{E} \mathbin{_\pi\times_{\pi'}} \mathcal{E}' := \{(\varepsilon, \varepsilon') \in \mathcal{E} \times \mathcal{E}' \mid
\pi(\varepsilon) = \pi'(\varepsilon')\}
\]
by $(\varepsilon, \varepsilon') \sim (\delta, \delta')$ if and only if there exists $z
\in \mathbb{T}$ such that $z\cdot\varepsilon = \delta$ and $\overline{z}\cdot
\varepsilon' = \delta'$.

The pullback $\mathcal{E} * \mathcal{E}'$ is defined as
\[
\mathcal{E} * \mathcal{E}' := \big(\mathcal{E} \mathbin{_\pi\times_{\pi'}} \mathcal{E}'\big)/{\sim}.
\]
This is a twist over $R$ with respect to the map $\pi*\pi' : \mathcal{E}
* \mathcal{E}' \to R$ given by $(\pi*\pi')([\varepsilon, \varepsilon']) = \pi(\varepsilon)$, and the map
$i * i' : R^{(0)} \times \mathbb{T} \to \mathcal{E} * \mathcal{E}'$ given by $(i*i')(x,
z) := [i(x,z), i'(x, 1)]$.

\begin{lemma}
Let $R$ be a topological equivalence relation. Then the collection of isomorphism classes
of twists over $R$ becomes an abelian group $\operatorname{Tw}_R$ with identity element
equal to the class of the trivial twist, and with group operation given by $[\mathcal{E}]
+ [\mathcal{E}'] := [\mathcal{E} * \mathcal{E}']$.
\end{lemma}

If $A$ is a Fell algebra, we write $\mathcal{S}$ for the sheaf of germs of continuous
$\mathbb{T}$-valued functions on $\widehat{A}$. One can then form the sheaf cohomology
group $H^2(\widehat{A}, \mathcal{S})$. If $\widehat{A}$ is Hausdorff, then
$H^2(\widehat{A}, \mathcal{S})$ is isomorphic to $H^3(\widehat{A}, \mathbb{Z})$.

\begin{theorem}\label{thm:twists iso}
Let $R$ be a topological equivalence relation. Then there is a homomorphism $\rho_R :
\operatorname{Tw}_R \to H^2(\widehat{A}, \mathcal{S})$. If $A$ is a Fell algebra, $(C_1,
D_1)$ and $(C_2, D_2)$ are two Cartan pairs constructed as in
Proposition~\ref{prp:construct diagonal} and $\mathcal{E} \coloneq \mathcal{E}_{C_1, D_1}
\to \mathcal{G}_{C_1, D_1}$ and $\mathcal{E}' \coloneq \mathcal{E}_{C_2, D_2} \to
\mathcal{G}_{C_2, D_2}$ are the twists obtained from these two pairs, then
$\rho_R([\mathcal{E} \to R]) = \rho_{R'}([\mathcal{E}' \to R'])$.
\end{theorem}

We denote the element $\rho_R([\mathcal{E} \to R]) \in H^2(\widehat{A}, \mathcal{S})$
obtained from any Cartan pair constructed as in Proposition~\ref{prp:construct diagonal}
by $\delta(A)$. For our final result, we need to recall that if $A$ is a $C^*$-algebra,
then its spectrum $\widehat{A}$ is a locally compact locally Hausdorff space whose every
open subset is again locally compact. Dixmier calls such spaces \emph{quasi locally
compact} \cite{Dixmier}, but I'm going to call them \emph{locally locally compact}.

\begin{theorem}\label{thm:DD Fell}
\begin{enumerate}
\item Let $A$ and $A'$ be Fell algebras. Then $A$ and $A'$ are Morita equivalent if
    and only if there is a homeomorphism $\widehat{A} \cong \widehat{A'}$ such that
    the induced isomorphism $H^2(\widehat{A}, \mathcal{S}) \cong H^2(\widehat{A'},
    \mathcal{S})$ carries $\delta(A)$ to $\delta(A')$.
\item\label{it:DD Fell surj} If $X$ is a locally locally compact, locally Hausdorff
    space, and $\delta \in H^2(X, \mathcal{S})$, then there exist a Fell algebra $A$
    and a homeomorphism $\widehat{A} \to X$ such that the induced isomorphism
    $H^2(\widehat{A}, \mathcal{S}) \cong H^2(X, \mathcal{S})$ carries $\delta(A)$ to
    $\delta$.
\end{enumerate}
\end{theorem}

More or less by definition of the invariant $\delta(A)$, the proof of Theorem~\ref{thm:DD
Fell}(\ref{it:DD Fell surj}) is very closely related to the Raeburn--Taylor construction:
Take $\delta \in H^2(X, \mathcal{S})$, represent it by a \v{C}ech cocycle $c$ defined on
an open cover $X = \bigcup_{i \in I} U_i$ by Hausdorff neighbourhoods. Let $Y :=
\bigsqcup_{i \in I} \{i\} \times U_i$, and define $\psi : Y \to X$ by $\psi(i,x) = x$.
The cocycle $c$ then determines a continuous cocycle $\sigma$ on $R(\psi)$ as in
Example~\ref{eg:RT construction}, and the $C^*$-algebra of the resulting twist is then a
Fell algebra with invariant $\delta$.

\begin{remark}
I have refrained from calling $\delta(A)$ the Dixmier--Douady invariant of $A$ because,
unfortunately, if $A$ is a continuous-trace $C^*$-algebra, it is not clear that
$\delta(A)$ is equal to the classical Dixmier--Douady invariant.
\end{remark}


\begin{thebibliography}{99}

\bibitem{A-D} C. Anantharaman-Delaroche, \emph{Purely infinite
    {$C^*$}-algebras arising from dynamical systems}, Bull. Soc. Math. France
    \textbf{125} (1997), 199--225.

\bibitem{AR} C. Anantharaman-Delaroche and
    J. Renault, Amenable groupoids, Foreword by Georges Skandalis and Appendix B
    by E. Germain, L'Enseignement Math\'ematique, Geneva, 2000, 196.

\bibitem{Ara} P. Ara, \emph{Morita equivalence for rings with involution},
    Algebr. Represent. Theory \textbf{2} (1999), 227--247.

\bibitem{BL} S. Barlak and X. Li, \emph{Cartan subalgebras and the UCT
    problem}, preprint 2015 (arXiv:1511.02697 [math.OA]).

\bibitem{Blackadar} B. Blackadar, Operator algebras, Theory of $C\sp
    *$-algebras and von Neumann algebras, Operator Algebras and Non-commutative Geometry,
    III, Springer-Verlag, Berlin, 2006, xx+517.

\bibitem{BF} R. Bowen and J. Franks, \emph{Homology for zero-dimensional
    nonwandering sets}, Ann. of Math. (2) \textbf{106} (1977), 73--92.

\bibitem{BCS} J. Brown, L.O. Clark and A. Sierakowski, \emph{Purely
    infinite {$C^\ast$}-algebras associated to \'etale
 groupoids}, Ergodic Theory Dynam. Systems \textbf{35} (2015), 2397--2411.

\bibitem{BCFS} J. Brown, L.O. Clark, C. Farthing and A. Sims,
    \emph{Simplicity of algebras associated to \'etale groupoids}, Semigroup Forum
    \textbf{88} (2014), 433--452.

\bibitem{Brown} L.G. Brown, \emph{Stable isomorphism of hereditary subalgebras of
    {$C\sp*$}-algebras}, Pacific J. Math. \textbf{71} (1977), 335--348.

\bibitem{BGR} L.G. Brown, P. Green and M.A. Rieffel, \emph{Stable
    isomorphism and strong {M}orita equivalence of {$C\sp*$}-algebras}, Pacific J. Math.
    \textbf{71} (1977), 349--363.

\bibitem{BCW} N. Brownlowe, T.M. Carlsen and M.F.
    Whittaker, \emph{Graph algebras and orbit equivalence}, Ergodic Theory Dynam. Systems
    \textbf{37} (2017), 389--417.

\bibitem{CRS} T.M. Carlsen, E. Ruiz, and A. Sims, \emph{Equivalence and stable
    isomorphism of groupoids, and diagonal-preserving stable isomorphisms of graph $C^*$-algebras
    and Leavitt path algebras}, Proc. Amer. Math. Soc. \textbf{145} (2017), 1581--1592.

\bibitem{Conway} J.B. Conway, A course in functional analysis,
    Springer--Verlag, New York, 1990, xvi+399.

\bibitem{Cuntz} J. Cuntz, \emph{The classification problem for the {$C\sp
    \ast$}-algebras {$\mathcal{O}_A$}}, Pitman Res. Notes Math. Ser., 123, Geometric
    methods in operator algebras (Kyoto, 1983), 145--151, Longman Sci. Tech., Harlow,
    1986.

\bibitem{Dixmier} J. Dixmier, {$C\sp*$}-algebras, Translated from the
    French by Francis Jellett, North-Holland Mathematical Library, Vol. 15, North-Holland
    Publishing Co., Amsterdam, 1977, xiii+492.

\bibitem{DD} J. Dixmier and A. Douady, \emph{Champs continus d'espaces
    hilbertiens et de {$C^{\ast} $}-alg\`ebres}, Bull. Soc. Math. France \textbf{91}
    (1963), 227--284.

\bibitem{EW} M. Enomoto and Y. Watatani, \emph{A graph theory for
    {$C\sp{\ast} $}-algebras}, Math. Japon. \textbf{25} (1980), 435--442.

\bibitem{Exel} R. Exel, \emph{Inverse semigroups and combinatorial {$C\sp
    \ast$}-algebras}, Bull. Braz. Math. Soc. (N.S.) \textbf{39} (2008), 191--313.

\bibitem{FM1} J. Feldman and C.C. Moore, \emph{Ergodic equivalence
    relations, cohomology, and von {N}eumann algebras}, Bull. Amer. Math. Soc.
    \textbf{81} (1975), 921--924.

\bibitem{FM2} J. Feldman and C.C. Moore, \emph{Ergodic equivalence
    relations, cohomology, and von {N}eumann algebras. {I}}, Trans. Amer. Math. Soc.
    \textbf{234} (1977), 289--324.

\bibitem{FM3} J. Feldman and C.C. Moore, \emph{Ergodic equivalence
    relations, cohomology, and von {N}eumann algebras. {II}}, Trans. Amer. Math. Soc.
    \textbf{234} (1977), 325--359.

\bibitem{Franks} J. Franks, \emph{Flow equivalence of subshifts of finite type},
    Ergodic Theory Dynam. Systems \textbf{4} (1984), 53--66.

\bibitem{GWY} E. Guentner, R. Willett and G. Yu, \emph{Dynamic asymptotic dimension:
    relation to dynamics, topology, coarse geometry, and {$C^*$}-algebras}, Math. Ann.
    \textbf{367} (2017), 785--829.

\bibitem{Hahn} P. Hahn, \emph{Haar measure for measure groupoids}, Trans. Amer.
    Math. Soc. \textbf{242} (1978), 1--33.

\bibitem{aHKS} A. an Huef, A. Kumjian and A. Sims, \emph{A
    {D}ixmier--{D}ouady theorem for {F}ell algebras}, J. Funct. Anal. \textbf{260} (2011),
    1543--1581.

\bibitem{Kum86} A. Kumjian, \emph{On {$C\sp \ast$}-diagonals}, Canad. J. Math.
    \textbf{38} (1986), 969--1008.

\bibitem{KPRR} A. Kumjian, D. Pask, I. Raeburn and J. Renault,
    \emph{Graphs, groupoids, and {C}untz-{K}rieger algebras}, J. Funct. Anal.
    \textbf{144} (1997), 505--541.

\bibitem{MM} K. Matsumoto and H. Matui, \emph{Continuous orbit
    equivalence of topological {M}arkov shifts
 and {C}untz-{K}rieger algebras}, Kyoto J. Math. \textbf{54} (2014), 863--877.

\bibitem{MRW} P.S. Muhly, J.N. Renault and D.P. Williams,
    \emph{Equivalence and isomorphism for groupoid {$C\sp \ast$}-algebras}, J. Operator
    Theory \textbf{17} (1987), 3--22.

\bibitem{Paravicini} W. Paravicini, \emph{Induction for {B}anach algebras, groupoids
    and {${\rm KK}^{\rm  ban}$}}, J. K-Theory \textbf{4} (2009), 405--468.

\bibitem{Paterson} A.L.T. Paterson, Groupoids, inverse
    semigroups, and their operator algebras, Birkh\"auser Boston Inc., Boston, MA, 1999,
    xvi+274.

\bibitem{PS} B. Parry and D. Sullivan, \emph{A topological
    invariant of flows on {$1$}-dimensional spaces}, Topology \textbf{14} (1975),
    297--299.

\bibitem{Raeburn} I. Raeburn, Graph algebras, Published for the
    Conference Board of the Mathematical Sciences, Washington, DC, 2005, vi+113.

\bibitem{RT} I. Raeburn and J.L. Taylor, \emph{Continuous trace
    {$C\sp \ast$}-algebras with given {D}ixmier--{D}ouady class}, J. Austral. Math. Soc.
    Ser. A \textbf{38} (1985), 394--407.

\bibitem{tfb} I. Raeburn and D.P. Williams, Morita
    equivalence and continuous-trace {$C\sp *$}-algebras, American Mathematical Society,
    Providence, RI, 1998, xiv+327.

\bibitem{Ramsay} A. Ramsay, \emph{The {M}ackey--{G}limm dichotomy for foliations and
    other {P}olish groupoids}, J. Funct. Anal. \textbf{94} (1990), 358--374.

\bibitem{Ren80} J.N. Renault, A groupoid approach to {$C\sp{\ast}
    $}-algebras, Springer, Berlin, 1980, ii+160.

\bibitem{Ren:ideals} J.N. Renault, \emph{The ideal structure of groupoid crossed product
    {$C\sp \ast$}-algebras}, J. Operator Theory \textbf{25} (1991), 3--36.

\bibitem{Ren08} J.N. Renault, \emph{Cartan subalgebras in {$C^*$}-algebras}, Irish
    Math. Soc. Bulletin \textbf{61} (2008), 29--63.

\bibitem{RenWil} J.N. Renault and D.P. Williams, \emph{Amenability of groupoids arising
    from partial semigroup actions and topological higher rank graphs}, Trans. Amer.
    Math. Soc. \textbf{369} (2017), 2255--2283.

\bibitem{Rordam} M. R{\o}rdam, \emph{Classification of {C}untz--{K}rieger
    algebras}, $K$-Theory \textbf{9} (1995), 31--58.

\bibitem{SW} A. Sims and D.P. Williams, \emph{Renault's equivalence
    theorem for reduced groupoid {$C^\ast$}-algebras}, J. Operator Theory \textbf{68}
    (2012), 223--239.

\bibitem{SimWil2} A. Sims and D.P. Williams, \emph{The primitive ideals of some
    \'etale groupoid {$C^*$}-algebras}, Algebr. Represent. Theory \textbf{19} (2016),
    255--276.

\bibitem{Spielberg} J. Spielberg, \emph{Groupoids and {$C^*$}-algebras for
    categories of paths}, Trans. Amer. Math. Soc. \textbf{366} (2014), 5771--5819.

\bibitem{StadlerOuchie} M.M. Stadler and M. O'uchi, \emph{Correspondence of groupoid
    $C^*$-algebras}, J. Operator Theory \textbf{42} (1999), 103--119.

\bibitem{SzaboNotes} G. Szabo, \emph{Notes on Rokhlin dimension}, to appear in
    the volume ``Operator algebras and dynamics: groupoids, crossed products
    and Rokhlin dimension" in \emph{Advanced Courses in Mathematics. CRM Barcelona.}.

\bibitem{T} M. Tomforde, \emph{Stability of {$C\sp \ast$}-algebras
    associated to graphs}, Proc. Amer. Math. Soc. \textbf{132} (2004), 1787--1795.

\bibitem{Tu-UCT} J.-L. Tu, \emph{La conjecture de Baum--Connes pour les feuilletages
    moyennables}, $K$-Theory \textbf{17} (1999), 215--264.

\bibitem{Tu} J.-L. Tu, \emph{Non-Hausdorff groupoids, proper actions and $K$-theory},
    Doc. Math. \textbf{9} (2004), 565--597.

\bibitem{Willett} R. Willett, \emph{A non-amenable groupoid whose maximal and
    reduced {$C^*$}-algebras are the same}, M\"unster J. Math. \textbf{8} (2015), 241--252.

\bibitem{TFB^2} D.P. Williams, Crossed products of {$C{\sp
    \ast}$}-algebras, American Mathematical Society, Providence, RI, 2007, xvi+528.

\bibitem{WilliamsNotes} D.P. Williams, \emph{A primer on crossed products}, to appear in
    the volume ``Operator algebras and dynamics: groupoids, crossed products
    and Rokhlin dimension" in \emph{Advanced Courses in Mathematics. CRM Barcelona.}.
\end{thebibliography}

\newcommand{\etalchar}[1]{$^{#1}$}
\def\cprime{$'$}
\providecommand{\bysame}{\leavevmode\hbox to3em{\hrulefill}\thinspace}
\providecommand{\MR}{\relax\ifhmode\unskip\space\fi MR }
\providecommand{\MRhref}[2]{%
  \href{http://www.ams.org/mathscinet-getitem?mr=#1}{#2}
} \providecommand{\href}[2]{#2}

\end{document}